\documentclass[a4paper,twoside]{article}
\usepackage[utf8]{inputenc}
\usepackage[tbtags]{amsmath}                 
\usepackage{amssymb}                
\usepackage{amsbsy}
\usepackage{geometry}                
\usepackage[style=1]{mdframed}
\usepackage{amsthm}
\usepackage{amsfonts}
\usepackage{mathrsfs}                
\usepackage{bm}                      
\usepackage{cite}

\usepackage{lettrine}                

\usepackage[toc,page]{appendix}
\usepackage{enumerate}
\usepackage{titlesec}  
\usepackage{booktabs} 
\usepackage{threeparttable}
\usepackage{multirow} 
\usepackage{indentfirst}
\usepackage{xcolor}
\usepackage[numbers, sort&compress]{natbib}
\usepackage{fancyhdr}
\usepackage{enumitem}
\usepackage{hyperref}
\usepackage{graphicx}
 \usepackage{caption}
 \usepackage{extarrows}
\hypersetup{
    colorlinks=true,
    linkcolor=teal,
    filecolor=magenta, 
    urlcolor=cyan,
    citecolor=teal,
}

\titleformat{\subsection}{\large\bfseries}{\thesubsection\enspace}{1.5pt}{}
\titleformat{\subsubsection}{\it}{\thesubsubsection\enspace}{1.5pt}{}
\fancypagestyle{plain}
{
	\fancyhf{}
	
}

 \numberwithin{equation}{section}

\newtheorem{theorem}{Theorem}[section]            

\newtheorem{proposition}{Proposition}[section]
\newtheorem{lemma}{Lemma}[section]

\newtheorem{remark}{Remark}[section]

\newcommand{\curl}{\operatorname{curl}}
 \newcommand{\divv}{\operatorname{div}}
\newcommand{\tr}{\operatorname{tr}}
\newcommand{\cmode}[1]{\left\lvert {#1} \right\rvert}
\newcommand{\mode}[1]{\left\lVert {#1} \right\rVert}
\newcommand{\pa}{\partial}

\newcommand{\ot}{{\Omega_t}}
\newcommand{\gt}{{\Gamma_t}}
\newcommand{\gb}{{\Gamma_b}}
\newcommand{\dt}{D_t}
\newcommand{\at}{{\tau_t}}
\newcommand{\nt}{{n_t}}
\newcommand{\ab}{{\tau_b}}
\newcommand{\nb}{{n_b}}
\newcommand{\ddt}{\frac{\mathrm{d}}{\mathrm{d}t}}
\newcommand{\id}{\,\mathrm{d}}
\newcommand{\nat}{\nabla_\at}
\newcommand{\nnt}{\nabla_\nt}

\newcommand{\pe}{P\left(E_\ell(t)\right)}

\newcommand{\pall}{P\left(\mathcal{E}(t)\right)}

\newcommand{\lt}[1]{\left\lVert {#1} \right\rVert_{L^2(\ot)}}

\newcommand{\lgt}[1]{\left\lVert {#1} \right\rVert_{L^2(\gt)}}
\newcommand{\eo}{P\left(\mathcal{E}(0)\right)}
\newcommand{\rhs}{\left(\tr(\rho(\nabla v)^2)+\frac{1}{\rho}
(\dt\rho)^2\right)}

\newcommand{\llx}[1]{\left|\left.{#1} \right|_{X_i}\right|}
\newcommand{\xc}[1]{\left({#1}-\left.{#1}\right|_{X_i}\right)}
\newcommand{\tha}[1]{\mode{{#1}}_{\tilde{H}^\frac{1}{2}(\gt)}}
\newcommand{\thb}[1]{\mode{{#1}}_{\tilde{H}^{-\frac{1}{2}}
(\gt)}}

\newcommand{\st}{{S_t}}
\newcommand{\dgt}{{d_\gt}}
\newcommand{\gtx}{{\Gamma_{t*}}}

\newcommand{\lam}{\Lambda(\gtx,s,\delta,\pi/2)}
\newcommand{\g}{\mathrm{g}}
\newcommand{\nk}[1]{#1^{k}}
\newcommand{\nkj}[1]{#1^{k+1}}
\newcommand{\pstk}[1]{\left(\Phi_\st^{k{#1}}\right)^{-1}}
\newcommand{\dgtk}[1]{d_\gt^{k{#1}} }
\newcommand{\pallk}{\mathcal{P}_0 P\left(\mathcal{E}^{k+1}(t)\right)}

\newcommand{\bk}[1]{\overline{#1^{k}}}
\newcommand{\bkj}[1]{\overline{#1^{k+1}}}
\newcommand{\tkj}[1]{\widetilde{#1^{k+1}}}
\newcommand{\tk}[1]{\widetilde{#1^k}}

\newcommand{\gth}[1]{\overline{#1}}
\newcommand{\oth}[1]{\widetilde{#1}}

\newcommand{\po}{ \mathcal{P}_0}
\newcommand{\rv}{\nkj{\mathring{v}}}
\newcommand{\rp}{\nk{\rho}}
\begin{document}

\bibliographystyle{gbt7714-2005} 
\setlength{\baselineskip}{14pt}

\title
{\textbf{\large  Local Well-Posedness for  Compressible  Capillary-Gravity\\ Water Waves with 
 Acute Contact Angles
}}

 \author{Jingchi Huang$^\dag$\quad  Shanmu Li$^\ddag$ \quad Chao Wang$^*$}  

\footnotetext{Email:  
\it $^\dag$ huangjch25@mail.sysu.edu.cn\quad 
\it $^\ddag$ lishanmu@math.pku.edu.cn\quad
\it $^*$wangchao@math.pku.edu.cn}
      
 \date{}
\maketitle

\begin{abstract}

    Our purpose is to investigate the local well-posedness of the compressible Euler equations in a two-dimensional bounded corner domain with  acute contact angles. This configuration describes a free surface intersecting the fixed bottom at two points, where the fluid is subject to a gravitational field and the interface between the fluid and air is influenced by capillary forces. When the contact angles are less than $\pi/2$, we establish a local existence theory for the solution, with dissipation effects occurring at the contact points. The main analytical challenge arises from  contact point singularities, which renders previous methods for dealing with compressible free boundary problems inadequate. To overcome this, we first establish the geometric structure for the compressible Euler equations, an approach originally introduced by Shatah and Zeng \cite{Shatah2008} for incompressible fluids. Additionally, we provide a singularity analysis for the wave equations in the corner domain, which ensures the validity of calculations near the corner. Finally, based on the geometric structure and singularity analysis, we obtain a priori energy estimates. Using these estimates, we also prove the local well-posedness of the system in a geometric formulation. To our knowledge, this is the first result addressing compressible Euler equations with a free boundary that involves contact points.

\end{abstract}





\linespread{1.4}

\section{Introduction} 
\subsection{Equations of motion}
In this paper, we consider a compressible  fluid in a bounded two-dimensional corner domain $\ot$ with  an upper free surface $\gt$ and a fixed bottom $\gb$. This time-dependent domain $\ot$ contains two intersection points $X_i,\, i=l,r$ (contact points) on the left and right side: $\gt\cap\gb=\{X_i,i=l,r\}$. The contact angles between the free surface and the bottom are denoted by $\theta_i(t),\, i=l,r$.
 Moreover,  the fixed bottom $\gb$ is assumed to be smooth enough, and reduces to  a line segment near the contact points $X_i(i=l,r)$ for the sake of simplicity. The situation is drawn in Figure~\ref{fig1}.
 \begin{figure}[htbp]
    \centering
    \includegraphics[width=0.46\linewidth]{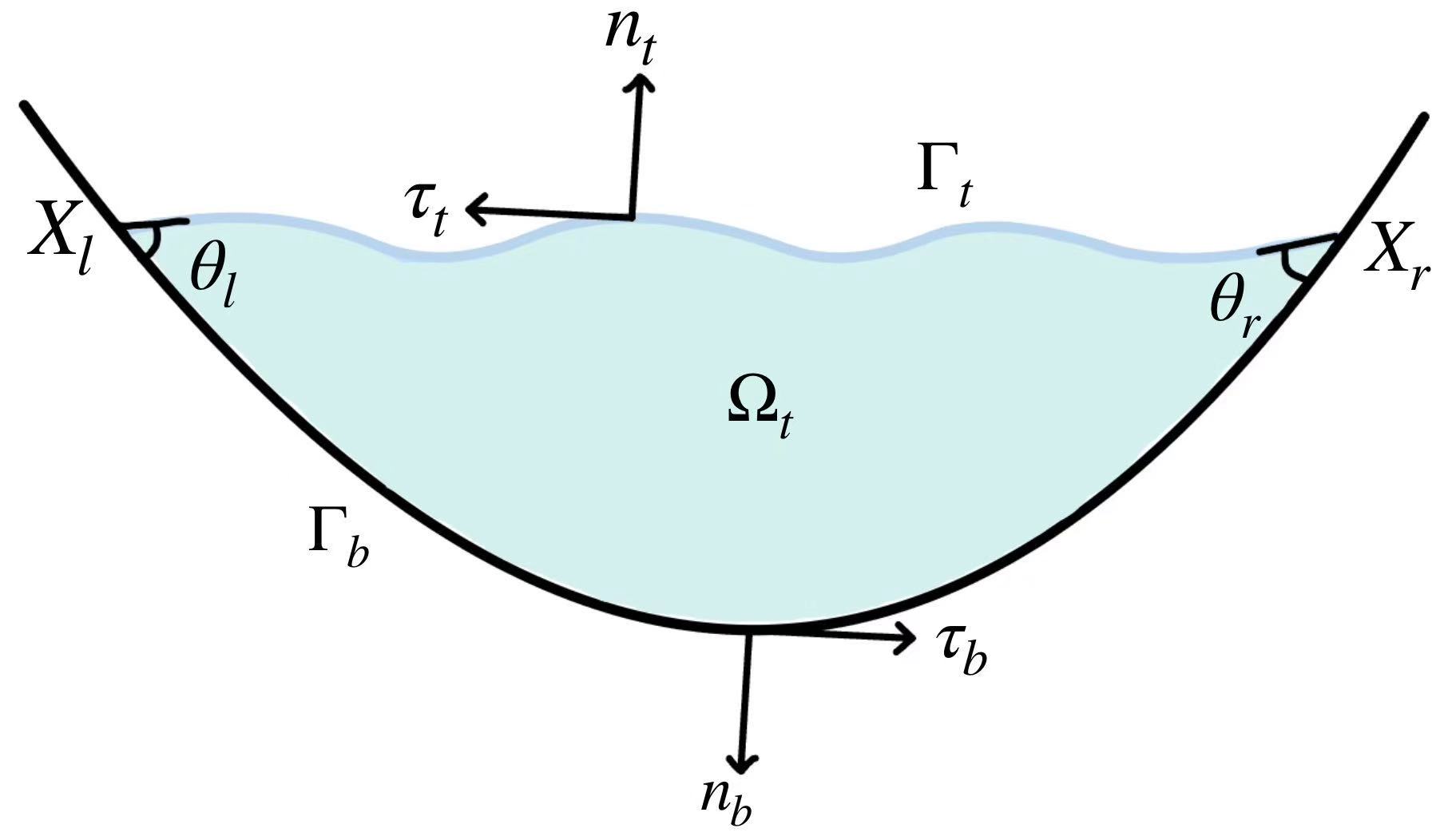}
    \caption{The corner domain}
    \label{fig1}
\end{figure}

The motion of an inviscid compressible liquid with a moving free boundary is modeled by the Euler equations:
 \begin{equation}
     \begin{cases}
         \rho (\pa_t v+v\cdot\nabla v)+\nabla P=-\rho \g e_2 & \text{in }\ot,\\ \pa_t\rho+\divv(\rho v)=0& \text{in }\ot,
     \end{cases}\label{E1}
 \end{equation}
where $v=(v_1,v_2)^T$ denotes the velocity field,  $P$  denotes the pressure of the fluid,  
and  $\rho$ denotes the density. We assume that the fluid is under the influence of gravity $-\rho \g e_2$, with constant $\g>0$ and $e_2=(0,1)^T$.
 The  equations from \eqref{E1} are conservation laws for momentum and mass. Here, we consider the free boundary problem, for which  the boundary conditions are 
 \begin{equation}\label{E3}
     \begin{cases}
         D_t= \pa_t+v\cdot\nabla \text{ is tangent to }\gt,\\
          P|_{\gt}=\sigma\kappa,\\
         v\cdot n_b|_{\gb}=0,
     \end{cases}
 \end{equation}
where $\dt$ is the material derivative,  $\sigma$ is the coefficient of surface tension,  $\kappa$ is the mean curvature of $\gt$,  and  $n_b$ is the unit outward normal vector of $\gb$. The first boundary condition in \eqref{E3}  is the classical kinematic condition on the free surface. The second one states that the fluid stress is proportional to the mean curvature of the moving surface, with the proportionality constant defining the surface tension parameter $\sigma$. The last one describes that the velocity along the fixed bottom $\gb$ is always tangential, which means there are no fluid particles transported at the bottom.
 
Surface tension effects necessitate additional boundary conditions at the contact points. The dynamic behavior of the contact line or point is significantly  delicate. The equilibrium contact angle $ \theta_s$ between the solid wall and the fluid is characterized by  Young's equation \cite{young 1805}\[\cos \theta_s=\frac{[\varsigma]}{\sigma}=\frac{\varsigma_1-\varsigma_2}{\sigma},\] where the  energetic parameters $\varsigma_1,\varsigma_2\in \mathbb{R}$ measure the free-energy per unit length associated to the solid-vapor and solid-fluid interactions, respectively.
As established  by \cite{E weinan},  the dynamic contact angle $\theta_i$ and the equilibrium angle $\theta_s$ are related via 
 \begin{equation}\label{E4}
    \mp \beta_c v\cdot\ab=\sigma(\cos \theta_s-\cos\theta_i) \quad  \text{at }X_i, i=l,r.
 \end{equation}
  Here, $v\cdot\tau_b$ denotes the fluid speed at the contact point along $\gb$,  and $\beta_c$ denotes the effective friction coefficient.
    This condition shows that the slip velocity is dominated by the unbalanced Young stress \cite{young 1805}, and it is indeed an effective variation of Young's law for stationary contact angles.

 We consider isentropic dynamics and
 use an equation-of-state to model the motion of a compressible liquid, which is  given by
 \begin{equation}\label{E5}
     P=\mathrm{A}\rho^\gamma-\mathrm{B}, \text{ for parameters }\mathrm{A},\mathrm{B}>0, \gamma>1.
 \end{equation} 
 Moreover, we assume the initial density $\rho_0$ is strictly positive, i.e., there  exists a positive constant $c_0>0$ such that
\begin{equation}\label{E1.15}
     \rho_0\geq c_0>0\quad \text{ in }\quad \overline{\Omega}_0.
 \end{equation}
To simplify the notation, 
 it is convenient to set $\mathrm{A}=\frac{1}{2}$ and $\gamma=2$ so that $\frac{\nabla P}{\rho}=\nabla \rho$. 

 Using the equation-of-state \eqref{E5},  the compressible Euler equations \eqref{E1} with boundary conditions \eqref{E3}-\eqref{E4} are equivalently written as the following system of velocity $v$ and density $\rho$,
 \begin{equation}
     \begin{cases}
         D_t v+\nabla\rho=- \g e_2 &\text{ in }\ot,\\
         \frac{1}{\rho}D_t\rho+\divv v=0 &\text{ in }\ot,\\
         \frac{1}{2}\rho^2=\sigma\kappa+\mathrm{B} &\text{ on }\gt,\\
         \dt=\pa_t+v\cdot\nabla \text{ is tangent to }\gt&\text{ on }\gt,\\
         v\cdot n_b=0&\text{ on }\gb,\\
         \mp\beta_cv\cdot\ab=\sigma(\cos \theta_s-\cos\theta_i) & \text{ at }X_i(i=l,r). 
     \end{cases}\label{E2}
 \end{equation}
The evolution system \eqref{E2} is supplemented with initial data consisting of
 \[(v,\rho)|_{t=0}=(v_0,\rho_0),\quad \Omega(0)=\Omega_0,\quad  X_i|_{t=0}=X_i(0), i=l,r,\]along with the following non-degeneracy conditions at contact points:  \begin{equation}
   - \left.  \left(\nabla_\ab\rho+\g e_2\cdot\ab\right)\right|_{(0,X_l(0))}
      \geq c_2\quad \text{ and }\quad  (\nabla_\ab \rho+\g e_2\cdot\ab)|_{(0,X_r(0))}\geq c_3, \label{E1.16}
 \end{equation} for some constants $c_2,c_3>0$. Meanwhile,  we introduce 
 the following compatibility conditions at $t=0$ for $j=0,1,2,3$:
\begin{equation}\label{E1.17}
        \begin{cases}
        \left.\dt^jP(\rho)\right|_{t=0}=\left.\sigma\dt^j\kappa\right|_{t=0}& \text{ on }\Gamma_0,\\
           \left. \dt^j (v\cdot\nb)\right|_{t=0}=0&\text{ on }\gb,\\
            \left.\mp \beta_cD_t^j v\cdot\ab\right|_{t=0}=\left.\sigma D_t^j(\cos\theta_s-\cos\theta_i )\right|_{t=0} &  \text{ at }X_i(0),\ i=l,r.
        \end{cases}
    \end{equation}

\medskip

Before we state the main results of this paper, we first define the   energy functional via \begin{equation}\label{E1.4}
    \begin{aligned}
        \mathcal{E}(t)=&\sum_{j=0}^3 \lt{\dt^j \rho}^2+\sum_{j=0}^2\lt{\nabla \dt^j\rho}^2+\sigma\sum_{j=0}^1\mode{\nat\dt^j\nnt\rho}^2_{L^2(\gt)}\\&+\mode{\omega}^2_{H^2(\ot)}+\mode{v}^2_{H^\frac{3}{2}(\ot)}+\mode{\gt}_{H^1}^2, 
    \end{aligned}
\end{equation}
and the dissipation term at the corner points via   \begin{equation}\label{E1.5}
    \begin{aligned}
       F(t)=&
    \sum_{i=l,r}\llx{(\sin\theta_i)\nat\nnt\rho}^2+\sum_{i=l,r}\left|(\sin\theta_i)(\nat\dt\nnt\rho-2(\nat v\cdot\nt)\Delta_\gt\rho)\big|_{X_i}\right|^2\\&+\left.\Big(\frac{1}{\sin\theta_i}\left(-\nabla_\ab\rho-\g e_2\cdot\ab\right)\rho|\dt^2\rho|^2\Big)\right|_{X_r}^{X_l}. 
    \end{aligned}
\end{equation}

\medskip

Our main results are stated as follows:
\begin{theorem}\label{Theorem2}

Assume that $(\rho_0, v_0, \Gamma_0)\in H^3(\Omega_0)\times H^{2+}(\Omega_0)\times H^{4.5}$ such that $\mathcal{E}(0)<\infty$, the initial contact angles satisfy $\theta_i(0)\in(0,\pi/2)$ for $i=l,r$, and the initial density $\rho_0$ complies with \eqref{E1.15} and \eqref{E1.16}. If, in addition, the compatibility conditions \eqref{E1.17} are satisfied, then there exists a time $T>0$ such that the system \eqref{E1}--\eqref{E4} admits a unique solution $(v,\rho,\ot)$ on $[0,T]$ satisfying
\[
\sup_{t\in[0,T]}\mathcal{E}(t)+\int_0^T F(t)\,\mathrm{d}t
\le
P\bigl(\mathcal{E}(0)\bigr)+\int_0^T P\bigl(\mathcal{E}(t)\bigr)\,\mathrm{d}t,
\]
where $P$ is a polynomial with positive coefficients depending only on $\sigma$, $\beta_c$, $ \mode{\gb}_{H^{4.5}}$ and $\g$, and where $T$ depends only on the same quantities and on the norm of the initial data.

\end{theorem}

\begin{remark}
 The results remain valid for the general case $\gamma>1,\mathrm{A}>0$, as one may reduce the general setting $\gamma>1$ to the special case $\gamma=2$. Define a new variable $\tilde{\rho}=\frac{\gamma}{\gamma-1}\frac{P+\mathrm{B}}{\rho}$ to retrieve the analysis developed for $\gamma=2$, owing to $\frac{\nabla P}{\rho}=\nabla \left(\frac{\gamma}{\gamma-1}\frac{P+\mathrm{B}}{\rho}\right)=\nabla\tilde{\rho}\ \text{ and }\ \frac{\dt\rho}{\rho}=\frac{1}{\gamma-1}\frac{\dt\tilde{\rho}}{\tilde{\rho}}.$  
The structure of system \eqref{E2}  is preserved under this transformation. Consequently, it suffices to study only \eqref{E2}, rather than the full system \eqref{E1}–\eqref{E4}.
\end{remark}

\begin{remark}
Here, condition \eqref{E1.16} ensures positivity of the final term in \(F(t)\). This novel structure arises specifically from compressibility. Such initial-data assumptions can be realized by taking the initial density to be near a constant in a neighborhood of the contact points.  
\end{remark}

\begin{remark}

To overcome the difficulties arising from  vorticity, 
we perform a Helmholtz decomposition of the vector field $v=\nabla \mathcal{V}+\nabla^\perp \mathcal{U}$, where $\nabla \mathcal{V}$ is the  irrotational part and $\nabla^\perp \mathcal{U}$ is the divergence-free part. This allows $v$ to be bounded by its curl, divergence and normal component. A key feature is that the vorticity $\omega$ is governed by a transport equation, so that $\omega$ and $\dt\omega$ enjoy  the same regularity, which enables us to   avoid the   derivative loss for $\omega$ during the linearization. By virtue of the foregoing arguments, we may also extend the main conclusions established in \cite{MingWang ARMA2024} to the setting of incompressible flows with vorticity.



\end{remark}

A major contribution of this work is the development of a new geometric framework for compressible free boundary problems. Classical approaches based mainly on Lagrangian coordinates  are substantially  difficult to apply for the  situation we considered. 
Inspired by the geometric formulation introduced by Shatah and Zeng \cite{Shatah2008} for incompressible flows, we extend this approach to the compressible Euler equations with boundary singularities and demonstrate that the geometric structure provides an effective framework, which is also applicable to classical compressible free boundary flows.

\subsection{Previous work}

A substantial body of work has been  devoted to establishing   local well-posedness   for water waves under various configurations: with or without
gravity, with or without surface tension, and in both finite and  infinite depth settings. 
The water waves problem for small data began with the early works of Nalimov \cite{Nalimov1974} (infinite bottom) and Yosihara \cite{Yosihara1982,Yosihara1983} (finite bottom)   in two dimensions. The first local well-posedness results for  general initial data in both  two and three dimensions were obtained by Wu \cite{Wu1997,Wu1999} who showed that 
the strong Taylor sign condition \begin{equation*}
   -\frac{\pa P}{\pa \nt}\geq c_0>0
\end{equation*}
always holds for infinite depth water wave problem. Subsequent research has extended and refined these results by considering  improved  regularity, varied bottom topography,  non-zero vorticity 
 and different approaches, we refer the reader to  Christodoulou and Lindblad  \cite{ChristodoulouLindblad2000}, Lannes  \cite{Lannes2005}, Lindblad \cite{Lindblad 2005}, Coutand and Shkoller  \cite{CountandShkoller2007} (without derivative loss), Alazard et al.  \cite{Alazard2014,Alazard2011,Alazard2016} (low regularity solutions),   Shatah and Zeng  \cite{Shatah2008}, Zhang and Zhang \cite{ZhangZhang2008},  Ambrose and Masmoudi \cite{Ambrose2005}, and the references therein. 


Compared to the classical theory on smooth domains, water waves problems with boundary singularities represent a completely new topic which started in recent years. 
 When the free interface is allowed to be non-$C^1$ and surface tension is neglected, 
  Kinsey and Wu \cite{KinseyWu2018} obtained  energy estimates that remain valid  for a class of  interfaces with angled crests, only a degenerate Taylor inequality
  \[ -\frac{\pa P}{\pa \nt}\geq 0\] holds, with  degeneracies at the crest point. They use a conformal mapping to convert the boundary singularity to the singularity for the mapping itself, after which the properties of the conformal
 mapping are carefully analyzed. 
  Later on, 
 Wu \cite{Wu2019} proved the local well-posedness of two-dimensional gravity waves for the same case. The
 crest angle is less than $\pi/2$ here and doesn’t change with respect to time. Further generalizations and related developments can be found in \cite{CordobaEncisoGrubic2023,Agrawal2020}.

We now turn to  the contact-line problem (or contact point problem in two dimensions), where the free surface and the fixed bottom intersect, i.e., \[\gt\cap\gb\neq \varnothing.\] 
 When the surface tension is ignored, as a pioneer work for this problem,  Alazard et al. \cite{Alazard2016} studied the case that the contact angle with respect to a vertical wall is equal to $\pi/2$, which keeps Taylor stability criterion non-degenerate. 
 Poyferr\'e \cite{Poyferre2019} established a priori estimate in general $n$-dimensional bounded domains with vorticity,  the angles are chosen small enough to  ensure the solution has sufficient regularity near the corner, where elliptic regularity works as in smooth domains. Very recently, Ming \cite{Ming2026}  proved a weighted  a priori energy estimate  in the absence of gravity, here, a conformal mapping is used to transform the equation for the mean curvature into an equivalent equation in a flat strip with some weights. Concerning the problem with surface tension, the case of more general angles has been tackled by Ming and Wang \cite{MingWang SIAM2020,MingWang ARMA2024,MingWang CPAM2021}, who studied the Dirichlet–Neumann operator associated with such a configuration, in two dimensions, and gave a complete description of its singularities at the corner. 
 For the contact-line problems of Navier–Stokes system, 
 Guo and Tice \cite{GuoTice2018Stokesflow} derived a priori estimates  for Stokes flow. Tice and Zheng  \cite{ZhengTice2017} investigated the local well-posedness for two-dimensional Stokes flow. Subsequently, 
 Guo and Tice \cite{Guo-Tice2024} analyzed the global dynamics and exponential stability of equilibrium states for a two-dimensional incompressible viscous fluid with moving contact points in an open-top container. 

 The compressible Euler equations are classified into gas and liquid cases according to the sound speed near the free boundary. 
 The gas case describes that the density vanishes near the free boundary, the rate of degeneracy (physical vacuum singularity) has attracted a great deal of attention, the local well-posedness results we refer to \cite{LuoXinZeng2014,CoutandShkollerCPAM2011,CoutandSComressARMA2012,IfirmTataru2024,JangMasmoudi2009,JangMasmoudi2015}. 
 For the compressible  liquid with free interface,  the  local well-posedness   restricted to bounded fluid domains was  proved by Lindblad \cite{Lindbladcompress2005}, using Nash–Moser iteration.  For subsequent works on  liquid-vacuum free interface problems in bounded  domains, we refer  the reader to \cite{Coutand2013,GinsbergLindbladLuo2020,LindbladLuoCPAM2018,WangZhangZhaoCMAA}. In the case of an unbounded fluid domain, Trakhinin \cite{Trakhinin2009} was first to prove the local well-posedness via Nash–Moser iteration which leads to a loss of regularity from initial data to solution. Later, Luo \cite{LuoAP2018}  established the a priori energy estimates for the compressible
gravity water wave with vorticity,  the corresponding  local well-posedness is obtained by Luo and Zhang \cite{LuoZhangJDE2022}. In contrast to contact-line problems for incompressible fluids, to the best of our knowledge, no literature addresses models that couple compressible fluid dynamics with contact point boundary phenomena. 

\subsection{Strategy of the proof}
We now give a brief technical overview of our methods,  in a rough sketch that highlights the main ideas and suppresses certain technical complications.  


For compressible water waves, 
a common approach  is to perform tangential-normal decomposition in the energy estimates, and use structures to express the normal part. However, this approach becomes infeasible in contact point problems, since the tangential direction is not well-defined near the corner, and its orientation may vary with the angle. Moreover, the contact points bring singularities such that the boundary of the moving domain is not regular enough.



To address these issues, we search for the geometric structure which is introduced by Shatah and Zeng \cite{Shatah2008,Shatah20082,Shatah2011} which they used to deal with incompressible Euler equations. The main idea is that they introduced the geometric structure based on the kinematic boundary condition on $\gt$, namely
\begin{equation}\label{E1.13}
     \dt\kappa=-\Delta_\gt v\cdot\nt-2\nat \nt\cdot\nat v\quad \text{ on }\gt.
\end{equation}
Motivated by their ideas, the authors of \cite{MingWang SIAM2020,MingWang CPAM2021,MingWang ARMA2024} exploited this geometric structure to investigate contact point problems. They established a priori estimates and local well-posedness for incompressible capillary water waves within two-dimensional corner domains, starting from small contact angles and subsequently extending the theory to acute contact angles. 
The geometric structure therefore remains valid for contact point problems.


 
 To be more specific, viewing from the incompressible water waves with surface tension~\cite{Shatah2008,MingWang SIAM2020}, the total pressure is decomposed as $P=P_{v,v}+\sigma\kappa_\mathcal{H}$, where $P_{v,v}$ is the Lagrangian multiplier, and $\kappa_\mathcal{H}$ is the harmonic extension of $\kappa$ into $\ot$, which demonstrates certain structure
\begin{equation*}\label{L1.12}
    \dt^2 \nabla\kappa_{\mathcal{H}}+\mathcal{A}_{\text{in} }\left(\nabla \kappa_{\mathcal{H}}\right)=R\quad \text{ with }\quad \mathcal{A}_{\text{in}}(f):=-\nabla\mathcal{H}\left(\Delta_\gt \left(f|_{\gt}\right)^\perp\right).
\end{equation*}
 Accordingly, we first establish such geometric structure for the compressible Euler system. Analogously to the incompressible setting, we may also derive a curvature equation by invoking
\eqref{E1.13}
and  the momentum equation,\begin{equation}\label{L1.6}
\frac{\rho}{\sigma}    \dt^2\rho+\mathcal{A}_{\text{com}}(\rho)=2(\nat v\cdot\nt)\nat \left(\nat v\cdot\at\right)+R\quad \text{with}\quad \mathcal{A}_{\text{com}}(f):=-\Delta_\gt \nnt f,
\end{equation}here, $R$ denotes the remainder terms. 
The main difference between $\mathcal{A}_{\text{com}}(\cdot)$ and $\mathcal{A}_{\text{in}}(\cdot)$ lies in the Dirichlet–Neumann operator:  the former acts on variables governed by a hyperbolic system in the interior, while the latter corresponds to the elliptic system. Owing to this discrepancy, the coupling of acoustic waves with corner singularities constitutes a fundamentally new difficulty, which is not a trivial generalization of  \cite{Shatah2008,MingWang SIAM2020}.

We next outline the regularity structure of $(\rho, v)$ over the corner domain. When the boundary is Lipschitz (piecewise smooth), singularities arise in elliptic estimates when higher-order regularity is involved. The solution can be decomposed into a regular part and  a singular part near the corner; the singular part,  also known as asymptotics,  depends on the shape of the corner and the left-hand side  and boundary operators of the elliptic system. 

For the density satisfying a hyperbolic system, 
spatial regularity follows by treating the material‑derivative term as the source in the Neumann-boundary elliptic system,
\[\dt^2\rho\in H^1(\ot)\xrightarrow[\eqref{E3.2}]{H^3(\ot)\text{ elliptic estimates}}\rho\in H^3(\ot),  \] for $\theta_i\in (0,\pi/2)$. 
We decompose the system of $v$ into the sum of Dirichlet-boundary elliptic system and Neumann-boundary elliptic system, for both yield exactly the same Sobolev regularity upper bound  since the dominant singularities are both of the form  $r^\frac{\pi}{\theta_i}$.   To avoid the increase of singularities,  we restrict our attention to $H^{3+}(\ot)$ estimates:  
\begin{align*}
    \dt^3\rho\in L^2 (\ot)\xrightarrow[\eqref{E4.1}]{H^2(\ot)\text{ elliptic estimates}}&\dt\rho\in H^2(\ot)\xrightarrow[\eqref{E3.156}-\eqref{E3.168}]{H^{3+\epsilon}(\ot)\text{ elliptic estimates} }v\in H^{2+\epsilon}(\ot),
\end{align*} for $\theta_i\in (0,\pi/2)$, and $\epsilon\in (0,\min\{1, \pi/\theta_i-2\})$ small enough. 
 Compared with the regularity of density,  $v$ loses some regularity due to the existence of corners. 

Moreover, 
the presence of surface tension allows us to  explore the higher regularity on the boundary. In fact, the variables $v$ and $\nabla\rho$ exhibit higher regularity in the normal direction,
 \begin{equation*}
     v\in H^{2+\epsilon}(\ot)\xrightarrow[\text{Lemma \ref{L2.5}}]{\text{Trace theorem}}\nat v\cdot\at\in H^{\frac{1}{2}+\epsilon}(\gt)\xrightarrow[\eqref{E1.13} \text{ and }\eqref{L1.6}]{\text{Curvature equations}}\begin{aligned} \nat v\cdot\nt\in H^{\frac{3}{2}+\epsilon}(\gt)\\ \nnt\rho\in H^{\frac{3}{2}+\epsilon}(\gt)\end{aligned}.
 \end{equation*}
 Essentially, the  singularities indirectly cause
 both  $v\cdot\nt$ and $\nnt\rho$ to lose some regularity. See Lemma \ref{L4.2} and Lemma \ref{L4.3} for more details.



In addition, compared with well-separated moving domain, 
   the a priori estimate exhibits a significant difference: 
dissipation occurs at the corner points, where both friction and compressibility must be taken into account. In fact, 
  each time a tangential integration by parts on the boundary integrals is performed,  corner terms appear. 
In the higher-order  energy estimates,
 certain types of singularities appear. These singularities are present in the control of boundary integrals. The subtle aspect in this context is that  $\Delta_\gt\rho|_{X_i}$ is unknown within the geometric structure at the contact points (see Lemma \ref{Lemma3.1}),  we therefore leave $\nat\Delta_\gt\rho$ in the principal of the  curvature equation,
 \begin{equation*}
     \dt^3\kappa-\nat\left(\nat\dt\nnt\rho-2(\nat v\cdot\nt)\Delta_\gt\rho\right)=R.
 \end{equation*}
 This directly leads to the second dissipation term defined  in \eqref{E1.5} at the contact points, while the first one is similar to that in the  incompressible case \cite{MingWang SIAM2020}.
  The third one arises from the fact that $\dt^2\rho|_{X_i}$ is unknown, consequently, a certain a priori assumption \eqref{E4.24} must be added to render it dissipative.






For  the   free-boundary problems,  local existence does not follow directly from the a priori estimate. Roughly speaking,  several structures are lost if one tries to construct approximate solutions, which means that in the a priori estimates, terms that rely on the symmetric structure, or the geometric structures  need to be carefully scrutinized. 
For  instance, solutions cannot be constructed via the standard iteration of Euler equation, due to unavoidable derivative loss.
 Although the  regularity loss for $v$ is only of order $\epsilon$ (with $\epsilon$  arbitrarily small),  it severely hinders the uniform estimates, owing to the fact that $\nabla v\notin L^\infty(\ot)$. 

To get the approximate solutions, we design a new strategy to address this issue. In the iteration process,  the iterates can be constructed successively in the following order,\[
\begin{aligned}
    \nk{\rho}, \nk{\Gamma_t}, \nk{v}\xrightarrow[\text{Proposition \ref{Prop6.1} }]{\text{Evolution equations \eqref{E6-14}}} &\nkj{\rho}\xrightarrow[\text{ Lemma \ref{l6.1}}]{\text{Non-degenerate elliptic system \eqref{E6.191}}}\nkj{\Gamma_t}\xrightarrow[\text{ Lemma \ref{l6.2} } ]{\text{Div-curl system \eqref{E8.16}}}\nkj{v}.\\ \text{\scriptsize Lemma \ref{l6.3}}& \uparrow\ \downarrow \text{\scriptsize Euler equation \eqref{E} }  \\ & \rv 
\end{aligned}
\]
Among these  steps, the construction of $\nkj{\rho}$ plays a central role and necessitates the introduction of an auxiliary variable   $\rv$ governed by linearized Euler equation in a given domain. 
The second step provides a geometric interpretation of the free surface,  while the last step  implies that $\nkj{v}$ is uniquely determined by its boundary behaviors, divergence and curl. 
Subsequently, we prove that the sequence $\{(\nk{\rho},\nk{v},\nk{\Omega_t})\}_{k=1}^{\infty}$ is a contraction in a short time interval, and converges to the solution of \eqref{E2}. Notably,  no derivative loss occurs in our construction.

Note that our construction differs substantially from other approaches previously used in inviscid compressible liquid with free boundary,  several works rely on parabolic regularizations (e.g., \cite{Coutand2013}) or Nash–Moser iteration (e.g., \cite{Lindbladcompress2005}).
In contrast, our method provides a framework in low-regularity Sobolev spaces, as it exploits the geometric formulation of the problem,  which is also better tailored to the general free boundary problems.

\subsection{Notations}
\begin{itemize}[itemsep=0pt, parsep=0pt, topsep=0pt, label=-]
    \item $n_j(j=t,b)$ are the unit outward normal vectors on $\Gamma_j$, and $\tau_j$ are the corresponding unit tangential vectors obeying the right-hand rule with $n_j$.
    \item $X_i,\, i=l,r$ are the coordinates of the left and right contact points at time $t$.
    \item $f|_{X_i}=\chi_l(f|_{X_l})+\chi_r(f|_{X_r})$ stands for taking values of $f$ at the contact points, where $\chi_i(i=l,r)$ are cut-off functions near the contact points $X_i$:\[\chi_i(X)=\begin{cases}
        1,&\text{when }|X-X_i|\leq \epsilon,\ X\in \ot,\\0,&\text{otherwise},
    \end{cases}\]for some small $\epsilon>0$.
    \item The notation $\mp f|_{X_i}$ is denoted by \[\mp f|_{X_i}=\begin{cases}
        -f|_{X_l}, &\text{ if }i=l,\\ f|_{X_r},&\text{ if }i=r,
    \end{cases}\]and  $\pm f|_{X_i}$ defined similarly.
    \item $D_t=\pa_t+\nabla_v$ is the material derivative.
    \item $M^*$ denotes the transport of a matrix $M$. 
    \item $\varpi ^\top=(\varpi\cdot\at)\at$
     for a vector-valued function $\varpi$ defined on $\gt$, and $\varpi^\perp=\varpi\cdot n_t$ on $\gt$.
    \item $\Pi$: the second fundamental form where $\Pi(\varpi)=\nabla_\varpi \nt\in T_X\gt$ (space of tangential vectors) for $\varpi\in T_X\gt$.
    \item $\kappa=tr\Pi=\nabla_\at\nt\cdot\at$ is the mean curvature of the surface $\gt$.
    \item We define on $\gt$ that $\mathcal{D}_XY=(\nabla_XY)^\top$ for $X,Y\in T_X\gt$, and $\mathcal{D} \varpi =\mathcal{D}_\at \varpi=(\nabla_\at \varpi\cdot\at)\at$ for a vector $\varpi\in T_X\gt$.
    \item $\Delta_\gt$ is the Beltrami–Laplace operator on $\gt$, \[\Delta_\gt f=\mathcal{D}^2 f(\at,\at)=\nabla_\at\nabla_\at f-\nabla_{\mathcal{D}_\at\at }f.\]
    \item  $\mathcal{H}(f)$ or $f_\mathcal{H}$ is the harmonic extension of a  function $f$ on $\gt$, which is defined by the elliptic system \[\begin{cases}
        \Delta \mathcal{H}(f)=0&\text{in }\ot,\\
        \mathcal{H}(f)|_\gt=f,&\nabla_\nb \mathcal{H}(f)|_\gb=0.
    \end{cases}\]
    \item $\tilde{H}^\frac{1}{2}(\Gamma_j)(j=t,b)$ is a subspace of $H^\frac{1}{2}(\Gamma_j)$ related to corner domain \[\tilde{H}^\frac{1}{2}(\Gamma_j)=\left\{u\in\dot{H}^\frac{1}{2}(\Gamma_j)\mid \lambda_i^{-\frac{1}{2}}u\in L^2 (\Gamma_j) , i=l,r\right\},\] where $\dot{H}^\frac{1}{2}(\Gamma_j)$ is the closure of $\mathscr{D}(\Gamma_j)$ in $H^\frac{1}{2}(\Gamma_j)$, and $\lambda_i=\lambda_i(X)$ for $i=l,r$ is the distance (arc length) between the point $X\in\Gamma_j$ and the end $X_i$. We define the norm \[\mode{u}_{\tilde{H}^\frac{1}{2}(\Gamma_j)}^2=\mode{u}_{H^\frac{1}{2}(\Gamma_j)}^2+\sum_{i=l,r}\int_{\Gamma_j}\lambda_i^{-1}|u|^2\id S.\]
    \item $\tilde{H}^{-\frac{1}{2}}(\Gamma_j)$ stands for the dual space of $\tilde{H}^{\frac{1}{2}}(\Gamma_j)$, see \cite{Grisvard} for more details.
     \item $C(\cdot)$ stands for a positive coefficient $C$ depending on its variables, and 
      $P(\mathcal{E}(t))$ stands for some polynomial of the energy $\mathcal{E}(t)$ with positive constant coefficients.
\end{itemize}

\section{Preliminaries}\label{Section2}

In this paper, we establish the geometric structure for the compressible Euler equations. Before that, we recall some geometric calculations from the work \cite{Shatah2008} by Shatah and Zeng. First, we have
\begin{equation}\label{E2.1}
    \dt\nt=-\left((\nabla v)^*\nt\right)^\top=-(\nabla_{\at} v\cdot \nt)\at\quad \text{ on }\gt.
\end{equation}
Moreover, since our domain is two-dimensional in this paper, we know directly that $\at$ is the parallel-transporting tangent basis satisfying 
\begin{equation}\label{E2.2}
    \dt\at=(\nabla_{\at}v\cdot\nt)\nt \ \text{ and }\ \mathcal{D}_{\at}\at=\left(\nat\at\right)^\top=0\ \text{ on }\ \Gamma_t.
\end{equation}
 The following commutators are essential, which related to the evolution of the geometry.
\begin{lemma}\label{L2.1}
    Let $f$ be a smooth function, and $i,j\in \{1,2\}$. Then, the following hold
    \begin{equation*}
        \begin{aligned}
         &  [\dt,\nabla ]f=-(\nabla v)^*\nabla f,\quad [\dt,\partial_i]f=-\partial_iv_j\partial_j f &\text{in }\ot,\\
           & [\dt,\Delta]f=-2\nabla v\cdot\nabla^2f -\Delta 
   v\cdot\nabla f&\text{in }\ot,\\
    &[\dt,\nabla_\at]f=-(\nabla_\at v)^\top\nabla f=-(\nabla_\at v\cdot\at)\nabla_\at f &\text{on }\gt,
    \\& [\dt,\nabla_\nt]f=-(\nabla_\at v\cdot\nt+\nnt v\cdot\at)\nabla_\at f-\nabla_\nt v\cdot\nt\nnt f
     &\text{on }\gt,\\
 &[\dt,\Delta_\gt]f=-2\nabla_{(\nabla_\at v)^\top}\nabla_\at f-\left(\Delta_\gt  v\cdot\at+\nabla_\at v\cdot\nabla_\at\at\right) \nabla_\at f&\text{on }\gt.\\
        \end{aligned}
    \end{equation*}
\end{lemma}
\begin{proof}
    Most of the above formulas can be found in \cite{Shatah2008} Section 3.1. The rest follow by direct computation.
\end{proof}

We recall the following Reynolds transport theorem (see, e.g., \cite{Shatah2008}). 
\begin{lemma}\label{L2.3} Assume that the time-dependent $\ot\subset \mathbb{R}^2$  is transported by  Lipschitz velocity $v$. Then, for some  functions $f$ and $g$ defined on the moving domain $\ot$,   the following hold:
   \begin{align*}
       \ddt \int_\ot f \,\mathrm{d}X=&\int_\ot \dt f\,\mathrm{d}X+\int_\ot f\divv v\id X,\\
       \ddt \int_\gt f\id S=&\int_\gt \dt f\id S+\int_{\gt} f\nabla_\at v\cdot\at\id S,\\
       \ddt \int_\gt f\nabla_\at g \id S=&\int_\gt\dt f\nabla_\at g\id S+\int_\gt f\nat\dt g\id S.
   \end{align*}
\end{lemma}

Next, we give the lemma which deals with traces on each piece of boundary. 

\begin{lemma}[Traces on $\gt$ or $\gb$,  Theorem 5.3 of \cite{MingWang SIAM2020}]\label{L2.5} The maps \[u\mapsto \{u, \nabla_{n_j}u\}|_{\Gamma_j},\ \text{ for }j=t,b\] have unique continuous extensions as operators from $H^s(\ot)$ onto $\Pi_{i=0}^1 H^{s-i-\frac{1}{2}}(\Gamma_j)$ for $s>\frac{3}{2}$.  Moreover, one has the estimate:\[\mode{u}_{H^{s-\frac{1}{2}}(\Gamma_j)}+\mode{\nabla_{n_j}u}_{H^{s-\frac{3}{2}}(\Gamma_j)}\leq C\left(\mode{\Gamma_j}_{H^{s-\frac{1}{2}}}\right)\mode{u}_{H^s(\ot)}.\]
\end{lemma}
    

Moreover, we present some special trace theorems on corner domains  involving $\tilde{H}^\frac{1}{2}(\Gamma_j)$ and $\tilde{H}^{-\frac{1}{2}}(\Gamma_j)(j=t,b)$ by applying Hardy inequality with some straightening localizations.
\begin{lemma}[Lemma 2.5 of \cite{MingWang ARMA2024}]\label{L2.6}
    Assume that $u|_\gt=f, $ with $f|_{X_i}=0\,(i=l,r)$ for a function $u\in H^1(\ot)$. Then one has $f\in \tilde{H}^\frac{1}{2}(\gt)$ and \[\mode{f}_{\tilde{H}^\frac{1}{2}(\gt)}\leq C\left(\mode{\gt}_{H^\frac{5}{2}}\right)\mode{u}_{H^1(\ot)}.\] The case on $\gb$ holds similarly.
\end{lemma}
\begin{lemma}[Lemma 5.6 of \cite{MingWang SIAM2020}]\label{L2.7}
    Let $u\in H^\frac{1}{2}(\Gamma_j)(j=t,b)$, then $\nabla_{\tau_j}u$ belongs to $\tilde{H}^{-\frac{1}{2}}(\Gamma_j)$ and satisfies the estimate\[\mode{\nabla_{\tau_j}u}_{\tilde{H}^{-\frac{1}{2}}(\Gamma_j)}\leq C\left(\mode{\Gamma_j}_{H^\frac{5}{2}}\right)\mode{u}_{H^\frac{1}{2}(\Gamma_j)}.\]
\end{lemma}


\medskip

Consider 
the Neumann-boundary system \begin{equation}
    \begin{cases}\label{L2.2}
        \Delta u=h&\text{in }\ot,\\
        \nnt u|_{\gt}=f,& \nabla_{n_b}u|_{\gb}=g,
    \end{cases}
\end{equation}satisfying the compatibility condition\[\int_\ot h\id X=\int_\gt f\id S+\int_\gb g\id S.\]

 An $H^1$ solution always exists by using variational methods,
 although it has to be performed modulo constants since they are always solution of the homogeneous problem.
 One would expect from the case of a regular boundary that a solution should be $H^{s+1},s>0$, singularities arise when the higher-order regularity is related.

\begin{lemma}[Proposition 2.2 and Remark 2.3 of \cite{MingWang ARMA2024}]\label{L2.8}
    \begin{enumerate}
        \item[(1)] Let $h\in L^2(\ot),\ f\in L^2(\gt),\ g\in L^2(\gb)$ in \eqref{L2.2}. Then there exists a unique (up to an additive constant) variational solution $u\in H^1(\ot)$ to \eqref{L2.2} satisfying\[\mode{u}_{H^1(\ot)}\leq C\left(\mode{\gt}_{H^\frac{5}{2}}\right)\left(\mode{h}_{L^2(\ot)}+\mode{f}_{L^2(\gt)}+\mode{g}_{L^2(\gb)}+\mode{u}_{L^2(\ot)}\right);\]
        \item[(2)] Let $h\in H^s(\ot),\ f\in H^{\frac{1}{2}+s}(\gt),\ g\in H^{\frac{1}{2}+s}(\gb)$ in \eqref{L2.2}. The contact angle $\theta_i\in (0,\pi/2)$. Then there exists a unique (up to an additive constant) solution $u\in H^3(\ot)$ to \eqref{L2.2} satisfying \[\mode{u}_{H^{2+s}(\ot)}\leq C\left(\mode{\gt}_{H^3}\right)\left(\mode{h}_{H^{s}(\ot)}+\mode{f}_{H^{\frac{1}{2}+s}(\gt)}+\mode{g}_{H^{\frac{1}{2}+s}(\gb)}+\mode{u}_{L^{2}(\ot)}\right),\] for any constant $s\in[0,1]$;  
        \item[(3)] Let $\theta_i\in (0,\pi/2)$, one can have a more delicate and also natural estimate for \eqref{L2.2} with $h\in H^{1+\epsilon}(\ot)$, $f\in H^{\frac{3}{2}+\epsilon}(\gt),\ g\in H^{\frac{3}{2}+\epsilon}(\gb)$:
         \[\mode{u}_{H^{3+\epsilon}(\ot)}\leq C\left(\mode{\gt}_{H^{3+\epsilon}}\right)\left(\mode{h}_{H^{1+\epsilon}(\ot)}+\mode{f}_{H^{\frac{3}{2}+\epsilon}(\gt)}+\mode{g}_{H^{\frac{3}{2}+\epsilon}(\gb)}+\mode{u}_{L^2(\ot)}\right),\]where $\epsilon\in (0,\min\{1, \pi/\theta_i-2\})$.
    \end{enumerate}
\end{lemma}

One of the central questions in the theory of elliptic boundary value problems for domains with conic points is the question on the asymptotic behavior of the solutions near the singular boundary points.
  To be more specific,
  the solution can be decomposed into a regular  part and a singular part near the corners, where the singular part is indeed
 an explicit sum of singularities of the form 
 $r^\lambda\log^q r \varphi(\theta)$,
  where $r$ is the distance to the points $X_i$, and the $\lambda$ are a discrete set of real numbers and $\varphi(\theta)$ is a bounded trigonometric function (see, e.g., \cite{Poyferre2019,MingWangAA2017}). These singular functions depend on the shape of the corner as well as the left-side operators of the elliptic system including the boundary operators instead of the right side of the system. Moreover, there will be more singular functions when a higher regularity of the solution is considered. The smaller the contact angle becomes, the smoother the solution would become.
 \begin{remark}
    For $\theta_i\in (0,\pi/2)$, the $H^{3+\epsilon}$ regularity for the Dirichlet elliptic system can be achieved whenever $\epsilon\in (0,\min\{1, \pi/\theta_i-2\})$. The dominant singularity near the corner behaves like $r^{\pi/\theta_i}$ for both pure Dirichlet and pure Neumann boundary conditions.
\end{remark}

\section{Reformulation of the problem}\label{Section3}
\noindent
{\bf The wave equation for density.}
Applying  operator $\dt$ to the second equation of $\eqref{E2}$ and taking operator $\divv$ to the first equation of $\eqref{E2}$, we get  the following  wave equation for the   density $\rho$:\begin{equation}
    \begin{cases}\label{E3.2}
         \dt^2 \rho-\rho\Delta\rho= \rho\tr\left((\nabla v)^2 \right)+\frac{1}{\rho}(\dt\rho)^2\qquad  \text{in }\ot,\\
         \frac{1}{2}\rho^2|_\gt=\sigma\kappa+\mathrm{B}, \quad  \nabla_{\nb}\rho|_\gb=v\cdot\nabla_v\nb-\g e_2\cdot \nb,
    \end{cases}
 \end{equation} where $\tr ((\nabla v)^2)=\pa_i v_j\pa_j v_i$.

\noindent {\bf The curvature equation.} 
To handle the boundary energy, it is necessary to track the evolution of the density on the free boundary and isolate the error terms.
First we recall the equation for the curvature $\kappa$ from \cite{Shatah2008}, 
\begin{equation}
    \dt \kappa=-\Delta_\gt v\cdot\nt-2\Pi(\at)\cdot\nat v \quad \text{ on }\gt, \label{E3.1}
\end{equation}or equivalently, \begin{equation}\label{E3.4}
    \dt\kappa=-\nat (\nat v\cdot\nt)-(\nat v\cdot\at)\kappa\quad \text{ on }\gt.
\end{equation}

Upon applying material derivative $\dt$ to equation \eqref{E3.1}, the following records some essential terms for it,
\begin{equation*}
    \begin{aligned}
        \dt^2\kappa=&-\Delta_\gt \dt v\cdot\nt-2\Pi(\at)\cdot \nat \dt v\\&-\left([\dt,\Delta_\gt]v\right)\cdot\nt-\Delta_\gt v\cdot\dt\nt- 2\dt\left(\Pi(\at)\right)\cdot \nat v-2\Pi(\at)\cdot [\dt,\nat]v.
    \end{aligned}  
\end{equation*}
From the Euler equation $\dt v=-\nabla\rho-\g e_2$, we can find the equation for $\dt^2 \kappa$,
    \begin{equation}\label{E3.3}
        \dt^2\kappa-\Delta_\gt\left(\nabla_\nt \rho\right)= \mathfrak{R} \quad \text{ on }\gt,
    \end{equation}
    with the error term 
    \begin{equation}\label{E3.5}
        \begin{aligned}
            \mathfrak{R}= & 4\left(\nat v\cdot\at\right)\Delta_\gt v\cdot\nt+2(\nat v\cdot\nt)\Delta_\gt v\cdot \at-\Delta_\gt\nt\cdot\nabla\rho\\& +6\left(\nat v\cdot\at\right)\nat\nt\cdot\nat v+3(\nat v\cdot\nt)\nat\at\cdot\nat v\\=& 4(\nat v\cdot\at)\nat (\nat v\cdot\nt)+2(\nat v\cdot\nt)\nat (\nat v\cdot\at)-\Delta_\gt\nt\cdot\nabla\rho\\&+\kappa\left(2(\nat v\cdot\at)^2-(\nat v\cdot\nt)^2\right),
        \end{aligned}
    \end{equation}
here we have used $\nat \nt =\kappa \at  $ and $\nat \at=-\kappa\nt$.

\noindent {\bf The dynamics of contact points.}
We  introduce some geometric information about the contact angle $\theta_i$ at the contact points:\begin{equation}\label{E3.7}
    \at\cdot\ab=-\cos\theta_i,\quad \at\cdot\nb=\pm \sin \theta_i,\quad  \nt\cdot\ab=\mp\sin\theta_i,\quad \nt\cdot\nb=-\cos\theta_i\quad \text{ at }X_i(i=l,r).
\end{equation}
In addition, 
recall that the boundary condition  $v\cdot\nb=0$ on $\gb$ and the assumption that the bottom $\gb$ is a line segment near the contact points, one gets \begin{equation*}
    \nabla_\ab v\cdot\nb=-v\cdot\nabla_\ab\nb=0\quad \text{ at }X_i(i=l,r).
\end{equation*}
Moreover, 
the following lemma records the dynamics of contact points.
\begin{lemma}\label{Lemma3.1}
    The following identities hold at the  contact points $X_i(i=l,r)$: 
    \begin{equation}\label{E3.42}
        \begin{aligned}
            \pm \dt\nnt\rho=&-\frac{\sigma}{\beta_c}\nat\nnt\rho(\sin\theta_i)^2-\frac{2\sigma}{\beta_c}(\nat v\cdot\at)(\nat v\cdot\nt)(\sin\theta_i)^2\\&+\frac{\sigma}{\beta_c}\nat\rho\kappa(\sin\theta_i)^2-\frac{2\sigma}{\beta_c}(\nat v\cdot\nt)^2(\nt\cdot\ab)(\at\cdot\ab)\\&\pm \g e_2\cdot\at(\nat v\cdot\nt)\pm \dt\left(v\cdot\nabla_v \nb(\nb\cdot\nt)\right), 
        \end{aligned}
    \end{equation}
    and 
    \begin{equation}\label{E3.43}
        \begin{aligned}
            \pm\dt\nat\rho=&-\frac{\sigma}{\beta_c}\nat\nnt\rho(\nt\cdot\ab)(\at\cdot\ab)+\frac{\sigma}{\beta_c}\nat\rho\kappa(\nt\cdot\ab)(\at\cdot\ab)\\&-\frac{2\sigma}{\beta_c}(\nat v\cdot\at)(\nat v\cdot\nt)(\nt\cdot\ab)(\at\cdot\ab)+\frac{\sigma}{\beta_c}(\nat v\cdot\nt)^2(\sin^2\theta_i-\cos^2\theta_i)\\&\mp \g e_2\cdot\nt (\nat v\cdot\nt)\pm \dt (v\cdot\nabla_v \nb(\nb\cdot\at)).
        \end{aligned}
    \end{equation}
\end{lemma}
\begin{proof}
     Recall that the behavior of the contact points $X_i$ is described by \[\mp\beta_c v\cdot\ab=[\varsigma]-\sigma\cos \theta_i=[\varsigma]+\sigma\at\cdot\ab\quad \text{ at }X_i(i=l,r).\] Upon taking material derivative $\dt$ on both sides of the equation and substituting the Euler equation from \eqref{E2}, noting that $\dt\at$ is given by \eqref{E2.2}, it is easy to see that
    \begin{equation}\label{E3.8} \pm (\nabla_\ab \rho+\g e_2\cdot\ab)=\frac{\sigma}{\beta_c}(\nat v\cdot\nt)\nt\cdot\ab\quad \text{ at }X_i(i=l,r),\end{equation} where one notices again that $\ab$ is a  constant vector near $X_i$. Moreover, combining \eqref{E3.7}, \eqref{E3.8} and the boundary condition $\nabla_\nb\rho|_\gb=v\cdot\nabla_v \nb-\g e_2\cdot\nb$, one obtains
    \begin{align}
        \pm\nnt\rho=\frac{\sigma}{\beta_c}(\nat v\cdot\nt)(\nt\cdot\ab)^2\mp \g e_2\cdot \nt\pm v\cdot\nabla_v\nb(\nb\cdot\nt)\quad \text{ at }X_i(i=l,r),\label{E3.40}\\
        \pm\nat\rho=\frac{\sigma}{\beta_c}(\nat v\cdot\nt)(\nt\cdot\ab)(\at\cdot\ab)\mp\g e_2\cdot\at\pm v\cdot\nabla_v \nb(\nb\cdot\at)\quad \text{ at }X_i(i=l,r).\label{E3.41}
    \end{align}
    To derive $\dt\nnt\rho$ and $\dt\nat\rho$, we need to compute the following identities. First, from \eqref{E2.1} and \eqref{E3.7}, one can write 
    \begin{equation}
        \dt(\sin\theta_i)^2=\dt(\nt\cdot\ab)^2 =-2(\nat v\cdot\nt)(\at\cdot\ab)(\nt\cdot\ab).\label{E3.44}
    \end{equation}Second, from the Euler equation $\dt v=-\nabla\rho-\g e_2$, one can find
    \begin{equation}
        \begin{aligned}
            \dt(\nat v\cdot\nt)&=\nat(\dt v\cdot\nt)-(\dt v\cdot \at)\kappa-2(\nat v\cdot\at)(\nat v\cdot\nt)\\&=-\nat\nnt\rho+\nat\rho\kappa-2(\nat v\cdot\at)(\nat v\cdot\nt).\label{E3.45}
        \end{aligned}
    \end{equation}In addition,\begin{align}
        \dt(\nat v\cdot\at)&=\nat(\dt v\cdot\at)+(\dt v\cdot\nt)\kappa-(\nat v\cdot\at)^2+(\nat v\cdot\nt)^2.\label{E3.46}
    \end{align}    
    Taking $\dt$ to  both sides of  equations \eqref{E3.40}-\eqref{E3.41} and applying 
   \eqref{E3.44}-\eqref{E3.45}, we directly obtain the desired results.
\end{proof}
   

\noindent{\bf The system for velocity field.} 
        One can use Hodge decomposition in the corner domain, which exists since $\gt$ is piecewise regular, writing 
        \begin{equation}\label{E3.156}
            v=\nabla \mathcal{V}+\nabla^\perp \mathcal{U},
        \end{equation}
      where $\nabla^\perp =(-\pa_2,\pa_1)^T$ is the perpendicular gradient, the first component is a gradient and the second depends only on $\nabla^\perp \cdot v$. More precisely,  
        we can choose $\mathcal{U}$ to
         satisfy the elliptic problem with Dirichlet data on both sides,\begin{equation}\label{E3.169}
            \begin{cases}
                \Delta\mathcal{U}=\nabla^\perp \cdot v&\text{ in }\ot,\\
                \mathcal{U}=0&\text{ on }\gt,\\
               \mathcal{U}=0&\text{ on }\gb.
            \end{cases}
        \end{equation}
        Denote the vorticity  by $\omega=\nabla^\perp\cdot v$, which satisfies the following transport equation:
        \begin{equation}\label{E3.16}
            \dt\omega=-(\divv v)\omega\quad \text{ in }\ot.
        \end{equation}
        Meanwhile, we can again use elliptic system  for $\mathcal{V}$ with Neumann data at both sides, 
        \begin{equation}\label{E3.168}
            \begin{cases}
                \Delta \mathcal{V}=\divv v=-\frac{1}{\rho}\dt\rho&\text{ in }\ot,\\
                \nnt\mathcal{V}=v\cdot\nt&\text{ on }\gt,\\
                \nabla_\nb \mathcal{V}=0&\text{ on }\gb,
            \end{cases}
        \end{equation}
        here we have used $\nabla^\perp \mathcal{U}\cdot n_j=\nabla_{\tau_j} \mathcal{U}=0$ on $\Gamma_j$ for $j=t,b$, and $v\cdot\nb=0$ on $\gb$. Note that such div-curl decomposition for $v$ is orthogonal in $L^2(\ot)$ since \begin{align*}
            \int_\ot \nabla^\perp \mathcal{U}\cdot\nabla \mathcal{V}\id X= \int_\gt\mathcal{U} \nat \mathcal{V}\id S+\int_\gb \mathcal{U}\nabla_\ab\mathcal{V}\id S -\int_\ot \mathcal{U}\nabla^\perp \cdot\nabla\mathcal{V}\id X=0.
        \end{align*}  

        In addition, assume that \[\int_\ot \mathcal{V}\id X=0.\]The norm $\mode{\mathcal{V}}_{L^2(\ot)}$ and $\mode{\mathcal{U}}_{L^2(\ot)}$ can  both  be absorbed by the higher-order norm during the proof.

\section{Lower-order energy estimates}\label{Section4}

In this section, we establish the lower-order energy estimates that control the low-frequency components of $(\rho, v, \Gamma_t)$. Meanwhile, these estimates provide the fundamental framework for the approximate solutions constructed in Section \ref{Section6}.
 
The lower-order energy $E_\ell(t)$ is defined as 
\begin{equation}\label{E11}
    \begin{aligned}
        E_\ell(t)=
        & \mode{\dt^2\rho}^2_{L^2(\ot)}+\mode{\dt\rho}_{H^1(\ot)}^2+\sigma\mode{\nat\nnt\rho}^2_{L^2(\gt)}\\
&+\mode{\rho}_{H^1(\ot)}^2+\mode{\omega}_{H^1(\ot)}^2+\mode{v}^2_{H^\frac{3}{2}(\ot)}+\mode{\gt}_{H^1}^2,
    \end{aligned}
\end{equation} and the dissipation 
\begin{equation}\label{E4.34}
F_\ell(t)=\sum_{i=l,r}\llx{(\sin\theta_i)\nat\nnt\rho}^2. 
\end{equation}

 Assume that there exists a positive constant $c_0>0$ such that the density $\rho$ is uniformly positive, i.e., \begin{equation}\label{E4.22}
    \rho\geq c_0>0\quad  \text{ in }\ot.
\end{equation}
   Let $\theta_i\in(0,\pi/2), i=l,r$, assume that there exists a positive constant $c_1>0$, such that
\begin{equation}\label{E4.23}
    \sin\theta_i(t)\geq c_1>0,
\end{equation} which can be verified via a continuity argument.
The main result of this section is: 
\begin{theorem}\label{T1}
    Let the contact angles $\theta_i\in (0,\pi/2), i=l,r$,
    assume that  \eqref{E4.22}-\eqref{E4.23} hold for $t\in [0,T]$. Then, the following holds
   \[
       \begin{aligned}
        \sup_{t\in [0,T]} E_\ell&(t)+\int_0^T F_\ell(t)\id t \leq P(E_\ell(0))\\&+\int_0^T P\left(\mode{v}_{H^{2+\epsilon}(\ot)},\mode{\rho}_{H^{\frac{5}{2}}(\ot)},\mode{\dt\rho}_{H^2(\ot)},\mode{\dt^2\rho}_{H^1(\ot)},\mode{\nat\dt\nnt\rho}_{L^2(\gt)},E_\ell(t)\right) \id t,
       \end{aligned}
        \] for $\epsilon\in (0,1)$ small enough.
\end{theorem}
To begin with, the following lemma establishes the regularity of $\rho, v$ and $\gt$.

\begin{lemma}\label{L4.2}
    Assume that $E_\ell(t)\in L^\infty[0,T]$ for some $T>0$, if $\theta_i\in (0,\pi/2)$, one has the following estimates:
    \begin{equation*}
        \begin{aligned}
            \mode{\gt}_{H^{\frac{7}{2}}}+\mode{\nt}_{H^\frac{5}{2}(\gt)}\leq& P(E_\ell(t)),\\
            \mode{\rho}_{H^2(\ot)}+\mode{\nnt\rho}_{H^1(\gt)}\leq &P\left(E_\ell(t)\right), \\
            \mode{v}_{H^2(\ot)}+\mode{\nabla_\at v\cdot\nt}_{H^1(\gt)}\leq& P\left(E_\ell(t)\right).\\
        \end{aligned}
    \end{equation*}
\end{lemma}
\begin{proof}
       

  The regularity of hypersurface $\gt$ can be measured by its curvature $\kappa$ \cite{Poyferre2019},
\begin{equation}\label{E4.41}
    \begin{aligned}
        \mode{\gt}_{H^\frac{5}{2}}+\mode{\nt}_{H^\frac{3}{2}(\gt)}\leq C\left(\mode{\kappa}_{H^\frac{1}{2}(\gt)}+\lgt{\nt}\right)\leq C\left(\mode{\gt}_{H^1}\right)\left(\mode{\rho}^2_{H^1(\ot)}+1\right)\leq P(E_\ell(t)).
    \end{aligned}
\end{equation}

Using  interpolation as in smooth domains, we have for some constant $\delta>0$ that the following holds \[\lgt{\nnt\rho}\leq \frac{\delta}{C\left(\mode{\gt}_{H^\frac{5}{2}}\right)}\mode{\rho}_{H^2(\ot)}+\frac{C\left(\mode{\gt}_{H^\frac{5}{2}}\right)}{\delta}\mode{\rho}_{H^1(\ot)}.\]
        It's straightforward to find
        \[\mode{\nnt\rho}_{H^\frac{1}{2}(\gt)}\leq C\left(\mode{\gt}_{H^\frac{5}{2}}\right)\left(\mode{\nat\nnt\rho}_{L^2(\gt)}+\mode{
        \nnt\rho
    }_{L^2(\gt)}\right). \]
Combining system \eqref{E3.2} and Lemma \ref{L2.8}, choosing a suitable $\delta\in(0,1)$,  one deduces that\[
    \begin{aligned}
         \mode{\rho}_{H^2(\ot)}&\leq C\left(\mode{\gt}_{H^\frac{5}{2}}\right)\left(\mode{\frac{1}{\rho}\dt^2\rho}_{L^2(\ot)}+\mode{\nabla v}_{L^4(\ot)}^2+\mode{\frac{1}{\rho}\dt\rho}^2_{L^4(\ot)}\right)\\&\quad +C\left(\mode{\gt}_{H^\frac{5}{2}}\right)\left(\mode{\nnt\rho}_{H^\frac{1}{2}(\gt)}+\mode{\nabla_\nb \rho}_{H^\frac{1}{2}(\gb)}+\mode{\rho}_{L^2(\ot)}\right)\\&\leq C\left(\mode{\gt}_{H^\frac{5}{2}}\right)\left(P(E_\ell(t))+\mode{v}_{H^\frac{3}{2}(\ot)}^2\right)+\delta \mode{\rho}_{H^2(\ot)}\\& \leq P(E_\ell(t)), 
    \end{aligned}    
   \] here we have used the Sobolev embedding $H^\frac{1}{2}(\ot)\hookrightarrow L^4(\ot)$. Then, we directly obtain     \begin{equation*}\label{E4.44}
       \mode{\nnt\rho}_{H^1(\gt)}\leq P(E_\ell(t)).
   \end{equation*}

  Similar to \eqref{E4.41}, we derive from Lemma \ref{L2.5} that 
   \begin{equation*}\label{E4.43}
       \mode{\gt}_{H^\frac{7}{2}}+\mode{\nt}_{H^\frac{5}{2}(\gt)}\leq C\left(\mode{\kappa}_{H^\frac{3}{2}(\gt)}+\mode{\nt}_{L^2(\gt)}\right)\leq C\left(\mode{\gt}_{H^\frac{3}{2}}\right)\left(\mode{\rho}_{H^2(\ot)}^2+1\right)\leq P(E_\ell(t)).
   \end{equation*} 

   Next we turn to gain the regularity of $v$. 
    Using the boundary condition $\sigma\dt\kappa=\rho\dt\rho$ on $\gt$ and equation \eqref{E3.4}, yields 
     \begin{equation}\label{E4.7}
        \begin{aligned}
        \mode{\nat v\cdot\nt}_{H^1(\gt)}\leq C\left(\mode{\gt}_{H^3}\right)\left(\mode{\dt\kappa}_{L^2(\gt)}+\mode{v}_{H^\frac{3}{2}(\ot)}\right)\leq P(E_\ell(t)).
    \end{aligned}
    \end{equation}
Applying Lemma \ref{L2.8} to  the elliptic systems  \eqref{E3.169} and \eqref{E3.168}, one has 
\begin{align*}
    \mode{v}_{H^2(\ot)}\leq &\mode{\mathcal{V}}_{H^3(\ot)}+\mode{\mathcal{U}}_{H^3(\ot)}\\\leq & C\left(\mode{\gt}_{H^3}\right)\left(\mode{\frac{1}{\rho}\dt\rho}_{H^1(\ot)}+\mode{v\cdot\nt}_{H^\frac{3}{2}(\gt)}+\mode{\omega}_{H^1(\ot)}+\mode{\mathcal{V}}_{L^2(\ot)}+\mode{\mathcal{U}}_{L^2(\ot)}\right)\\ \leq &P(E_\ell(t)).
\end{align*}
This completes the proof.

\end{proof}

We now deal with the dissipation at the contact points, which will appear later in the estimates and play a key role. 
\begin{lemma}\label{L4.10}
   We have the following equality at the contact point $X_i$:
   \begin{equation*}\label{E4.11}
       (\dt\nnt\rho)\nat\nnt\rho\big|^{X_l}_{X_r}=-\frac{\sigma}{\beta_c}F_\ell(t)+\sum_{i=l,r}r_i,
   \end{equation*}
here dissipation $F_\ell(t)$ is defined by \eqref{E4.34} and remainder terms $r_i$ are given by \eqref{E4.111}.  
  We have the estimates
  \begin{equation*}
      \llx{r_i}\leq P\left(\mode{ v}_{H^{2+\epsilon}(\ot)},\mode{ \rho}_{H^{2+\epsilon}(\ot)},E_\ell(t)\right)F_\ell^\frac{1}{2}(t), \ i=l,r,
  \end{equation*}
   and
   \begin{equation*}
       \llx{\dt\nnt\rho}+\llx{\dt\nat\rho}\leq \frac{\sigma}{\beta_c}F^\frac{1}{2}_\ell(t)+P\left(\mode{ v}_{H^{2+\epsilon}(\ot)},\mode{ \rho}_{H^{2+\epsilon}(\ot)},E_\ell(t)\right).
   \end{equation*}
\end{lemma}
\begin{proof}
    To begin with,  multiplying   both sides of  \eqref{E3.42} by $\nat\nnt\rho$ gives \[\pm\dt\nnt \rho\nat\nnt\rho|_{X_i}=-\frac{\sigma}{\beta_c}\llx{(\sin\theta_i)\nat\nnt\rho}^2+r_i\]with 
    \begin{equation}\label{E4.111}
        \begin{aligned}
            r_i=&-\frac{2\sigma}{\beta_c}(\nat v\cdot\at)(\nat v\cdot\nt)(\sin\theta_i)^2\nat\nnt\rho+\frac{\sigma}{\beta_c}\nat\rho\kappa(\sin\theta_i)^2\nat\nnt\rho\\&-\frac{2\sigma}{\beta_c}(\nat v\cdot\nt)^2(\nt\cdot\ab)(\at\cdot\ab)\nat\nnt\rho\\&\pm \left(\g e_2\cdot \at(\nat v\cdot \nt)+\dt\left(v\cdot\nabla_v \nb(\nb\cdot\nt)\right)\right)\nat\nnt\rho, \quad \text{ at }X_i(i=l,r).
        \end{aligned}
    \end{equation}
  Therefore, applying Sobolev embedding  to \eqref{E3.42}-\eqref{E3.43} and \eqref{E4.111}, we directly obtain the desired results.
\end{proof}
\begin{proof}[Proof of Theorem \ref{T1}]
    We divide the proof into several steps.

    {\bf Step1: Basic energy estimates.}
    Taking the $L^2(\ot)$-inner product of \eqref{E3.2} with $\dt \rho$, using Lemmas \ref{L2.3} and \ref{L4.2},  then integrating by parts, we get that
 \begin{equation}\label{E4.3}
    \ddt \left(\int_\ot |\dt\rho|^2\id X+\int_\ot\rho|\nabla \rho|^2\id X\right)
    \leq \pe,
 \end{equation}here we have used Sobolev embeddings 
$H^\frac{3}{2}(\ot)\hookrightarrow L^\infty(\ot)$ and $H^\frac{1}{2}(\ot)\hookrightarrow L^4(\ot)$.
    Taking $L^2(\ot)$-inner product of $\eqref{E2}_2$ with $\rho$, we find that \begin{equation}\label{E4.46}
     \ddt \int_\ot |\rho|^2\id X=-\int_\ot \rho^2\divv v\id X\leq \pe.
    \end{equation} 
    Similarly, the equation \eqref{E3.16} implies that, 
    \begin{equation}\label{E3.17}
       \ddt \mode{\omega}^2_{H^1(\ot)}
        \leq P\left(\mode{v}_{H^{2+\epsilon}(\ot)},E_\ell(t)\right).
    \end{equation}

    {\bf Step2:  Material derivative energy estimates.}
    Upon applying operator $\dt$ to \eqref{E3.2}, we arrive at 
    \begin{equation}\label{E4.1}
        \begin{cases}
            \dt^3\rho-\rho\nabla\cdot\dt\nabla\rho=\dt\rhs+\dt\rho\Delta\rho+\rho [\dt,\nabla\cdot]\nabla\rho&\text{in }\ot,\\
            \rho\dt^2\rho|_\gt=\sigma\dt^2\kappa-(\dt\rho)^2, \\\nabla_\nb\dt\rho|_{\gb}=(\nabla_\nb v-\nabla_v \nb)\cdot\nabla\rho+\dt\left(v\cdot\nabla_v\nb-\g e_2\cdot\nb\right).
        \end{cases}
    \end{equation}
    Then  taking the $L^2$-inner product of \eqref{E4.1} with $\dt^2\rho$,  integrating by parts to get that 
    \begin{equation*}
        \begin{aligned}
            &\frac{1}{2}\ddt \left(\int_\ot |\dt^2\rho|^2\id X+\int_\ot \rho|\nabla \dt\rho|^2 \id X\right)\\=&\int_\ot \frac{1}{2}\dt\rho|\nabla\dt\rho|^2\id X+\int_\gt\nt\cdot\dt \nabla\rho(\rho\dt^2\rho)\id S+\int_\gb \nb\cdot\dt\nabla\rho (\rho\dt^2\rho)\id S\\&-\int_\ot\rho \nabla\dt\rho\cdot(\nabla v)^*\nabla\dt\rho\id X-\int_\ot\nabla\rho\cdot\dt\nabla\rho\dt^2\rho\id X-\int_\ot \rho[\dt,\nabla]\rho\nabla\dt^2\rho\id X\\&+\int_\ot \left(\dt\rhs+\dt\rho\Delta\rho+\rho[\dt,\nabla\cdot]\nabla\rho\right)\dt^2\rho\id X\\&+\frac{1}{2}\int_\ot \left(|\dt^2\rho|^2 +\rho|\nabla\dt\rho|^2\right)\divv v\id X\\\leq& P\left(\mode{v}_{H^{2+\epsilon}(\ot)},\mode{\rho}_{H^{2+\epsilon}(\ot)},\mode{\dt^2\rho}_{H^1(\ot)},E_\ell(t)\right)
                \\&+ \int_\gt\rho\nt\cdot\dt \nabla\rho\dt^2\rho\id S+
                \int_\gb\rho \nb\cdot\dt\nabla\rho \dt^2\rho\id S.
        \end{aligned} 
    \end{equation*}
    We divide the boundary integral on $\gt$ into four components \begin{align*}
        \int_\gt\nt\cdot\dt \nabla\rho(\rho\dt^2\rho)\id S=&\int_\gt (-\dt\nt\cdot\nabla\rho)\rho\dt^2\rho\id S-
       \int_\gt \dt\nnt\rho(\dt\rho)^2 \id S\\&+\sigma\int_\gt\dt\nnt\rho\Delta_\gt(\nnt\rho)\id S+\sigma\int_\gt\dt\nnt\rho\mathfrak{R}\id S\\=&: I_1+I_2+I_3+I_4,
    \end{align*}here we have used the curvature equation \eqref{E3.3}.
    Now, it remains to deal with the right-hand side of equality above. In fact, 
    \begin{equation*}
        \begin{aligned}
            I_1+I_2\leq\pe\left(\mode{\dt\rho}_{H^2(\ot)}+\mode{\dt^2\rho}_{H^1(\ot)}+1\right). \\
        \end{aligned} 
    \end{equation*}
    Secondly, a direct computation using Lemma \ref{L4.10} leads to
    \begin{equation*}
        \begin{aligned}
            I_3
            =& \left(\sigma\dt\nnt\rho\nat\nnt\rho\right)\big|_{X_r}^{X_l} -\sigma\int_{\gt}\nat\dt\nnt\rho\nat\nnt\rho \id S\\=&-\frac{\sigma^2}{\beta_c}F_\ell(t)+\sum_{i=l,r}\sigma r_i|_{X_i} -\frac{\sigma}{2}\ddt \int_\gt|\nat\nnt\rho|^2\id S-\frac{\sigma}{2}\int_\gt|\nat\nnt\rho|^2(\nat v\cdot\at)\id S\\
            \leq &-\frac{\sigma^2}{2\beta_c}F_\ell(t)-\frac{\sigma}{2}\ddt \int_\gt|\nat\nnt\rho|^2\id S+P\left(\mode{v}_{H^{2+\epsilon}(\ot)},\mode{\rho}_{H^{2+\epsilon}(\ot)},E_\ell(t)\right).
        \end{aligned}
    \end{equation*}
   Recall that the term $2(\nat v\cdot\nt)\nat (\nat v\cdot\at)\notin L^2(\gt)$ in the definition of $\mathfrak{R}$, using integration by parts  and  Lemma \ref{L4.10} shows that
    \begin{align*}
        I_4\leq P\left(\mode{v}_{H^{2+\epsilon}(\ot)},\mode{\rho}_{H^{2+\epsilon}(\ot)},\mode{\dt\rho}_{H^2(\ot)},\mode{\nat\dt\nnt\rho}_{L^2(\gt)},E_\ell(t)\right) \left(1+\frac{\sigma^2}{\beta_c}F_\ell^\frac{1}{2}(t)\right).
    \end{align*}
Notice that the boundary condition in \eqref{E4.1}, yields
    \begin{equation*}
        \begin{aligned}
            \int_\gb\rho \nb\cdot\dt\nabla\rho \dt^2\rho\id S\leq \pe\mode{\dt^2\rho}_{H^1(\ot)}.
        \end{aligned}
    \end{equation*}
    Hence, from the above estimates, we deduce that
    \begin{equation}
        \begin{aligned}
            & \frac{1}{2}\ddt \left(\int_\ot|\dt^2\rho|^2\id X+\int_{\ot}\rho|\nabla\dt\rho|^2\id X+\sigma\int_\gt|\nat\nnt\rho|^2\id S\right)+\frac{\sigma^2}{4\beta_c}F_\ell(t)\\\leq &P\left(\mode{v}_{H^{2+\epsilon}(\ot)},\mode{\rho}_{H^{2+\epsilon}(\ot)},\mode{\dt\rho}_{H^2(\ot)},\mode{\dt^2\rho}_{H^1(\ot)},\mode{\nat\dt\nnt\rho}_{L^2(\gt)},E_\ell(t)\right) .
        \end{aligned}
       \label{E4.5}
    \end{equation}

{\bf Step3: Velocity and interface regularity.}
Direct calculation via the  Euler equation $\dt v=-\nabla\rho-\g e_2$, combined with Sobolev interpolation shows that 
\begin{equation}\label{E4.10}
    \mode{v}^2_{H^\frac{3}{2}(\ot)}\leq C(T) \mode{v_0}^2_{H^\frac{3}{2}(\Omega_0)}+C(T)\int_0^T P\left(E_\ell(t),\mode{v}_{H^{2+\epsilon}(\ot)},\mode{\rho}_{H^\frac{5}{2}(\ot)}\right)\id t.
\end{equation}

Second, for the term $\mode{\gt}_{H^1}$, we can parameterize $\gt$ under Eulerian coordinates $(x,z)$ by \[\gt=\{(x,z)|z=\eta(t,x), t>0,c_l(t)\leq x\leq c_r(t)\},\]where $c_i(t)$ is the $x$ coordinate for the contact point $X_i(i=l,r)$. 
 Note that the kinematic condition on $\gt$ can be written in the form of $\eta, $ $v$ as \[\pa_t\eta+v_1\pa_x\eta=v_2\quad \text{ on }[c_l(t),c_r(t)],\]
where $v=(v_1,v_2)^T$. Moreover, a direct computation shows that under this parameterization, the material derivative $\dt$ on $\gt$ is simply $\dt=\pa_t+v_1\pa_x $, so the equation above for $\eta$ can be rewritten as $\dt\eta=v_2$. As a result, it is straightforward to show that
\begin{equation}\label{E4.12}
    \pa_t\mode{\gt}^2_{H^1}=\pa_t\mode{\eta}^2_{H^1(c_l(t),c_r(t))}\leq P(E_\ell(t)). 
\end{equation}

Combining \eqref{E4.3}-\eqref{E3.17} and \eqref{E4.5}-\eqref{E4.12}, then  integrating on both sides with respect to time $t$ on $[0,T]$, we get the desired results.
\end{proof}
\section{Higher-order  energy estimates}\label{Section5}
The main result of this section gives   the a priori estimate for solutions to \eqref{E2}, which  shows that  solutions  have bounded energy  in  a sufficiently small  time interval $[0,T]$.
\begin{theorem}\label{Theorem1}
    Assume that 
    the initial contact angle $\theta_i(0)\in(0,\pi/2)$.
    Let $(\gt,v,\rho)\in H^4\times C(H^{2+}(\Omega_t))\times C(H^3(\Omega_t))$ be a solution of \eqref{E2}. Suppose that $\mathcal{E}(t)$ and $F(t)$ are defined  by \eqref{E1.4} and \eqref{E1.5}, respectively. 
    There exists a time $T>0$ such that $\forall t\in [0,T]$,  
 the following a priori estimate holds:\[\sup_{t\in[0,T]}\mathcal{E}(t)+\int_0^{T} F(t)\id t\leq P(\mathcal{E}(0))+\int_0^{T} P(\mathcal{E}(t))\id t,\]where $P(\cdot)$ is some polynomial with positive constant coefficients depending on $\sigma\, ,\beta_c\, \mode{\gb}_{H^5}$ and $\g$. The time  $T$ depends on the same objects and on the norm of the initial data. 
\end{theorem}
 To start with, we define the higher-order energy functional \begin{equation}\label{E12}
    \begin{aligned}
         E_h(t)=&\mode{\dt^3\rho}^2_{L^2(\ot)}+\mode{\nabla\dt^2\rho}_{L^2(\ot)}^2+\sigma\mode{\nat\dt\nnt\rho}^2_{L^2(\gt)}+\mode{\omega}_{H^2(\ot)}^2,
    \end{aligned}
\end{equation}
and higher-order dissipation term 
\begin{equation}\label{E5.47}
    F_h(t)=\sum_{i=l,r}\left|(\sin\theta_i)(\nat\dt\nnt\rho-2(\nat v\cdot\nt)\Delta_\gt\rho)\big|_{X_i}\right|^2.
\end{equation}
Assume that there exist  positive constants $c_2>0$ and $c_3>0$, such that
    \begin{equation}\label{E4.24}
        \left(-\nabla_\ab\rho-\g e_2\cdot\ab\right)|_{X_l}\geq c_2\quad \text{ and }\quad  (\nabla_\ab\rho+\g e_2\cdot\ab)|_{X_r}\geq c_3,
    \end{equation}which can be verified due to the initial settings and a short time interval (see the end of this section). Also assume that \eqref{E4.22}-\eqref{E4.23} hold for $t\in [0,T]$. The total energy and dissipation can be rewritten  as 
    \[\mathcal{E}(t)= E_\ell(t)+E_h(t)\quad \text{ and }\quad F(t)= F_\ell(t)+F_h(t)+\left.\left(\frac{1}{\sin\theta_i}\left(-\nabla_\ab\rho-\g e_2\cdot\ab\right)\rho|\dt^2\rho|^2\right)\right|_{X_r}^{X_l}.\]
Under those assumptions, the dissipation $F(t)$ is well-defined.



  The following lemma records the higher-order regularity of $\rho$ and $v$.
\begin{lemma}\label{L4.3} 
    Assume that $\theta_i\in (0,\pi/2)$,  then for some constants $\delta'_0\in(0,\frac{1}{2}]$ and  $\epsilon\in (0,\frac{\pi}{\theta_i}-2)$ small enough,  one has 
    \begin{equation*}
        \begin{aligned}
            \mode{\nabla v}_{L^\infty(\ot)}+\mode{\nabla \rho}_{L^\infty (\ot)}+\mode{v}_{H^{2+\epsilon}(\ot)}+  \mode{\nat v\cdot\nt}_{H^{\frac{3}{2}+\epsilon}(\gt)}&\leq \pe\left(E^{\epsilon\delta_0'}_h(t)+1\right),\\
            \mode{\rho}_{H^3(\ot)} +  \mode{\dt\rho}_{H^2(\ot)}+\mode{\gt}_{H^4}+\mode{n_t}_{H^3(\gt)}&\leq  \pe\left(E_h^{\delta'_0}(t)+1\right),\\
           \mode{\nnt\rho}_{H^{\frac{3}{2}+\epsilon}(\gt)}
       &\leq \pe\left(E^{\frac{1}{2}\delta_0'}_h(t)+1\right).
        \end{aligned}
    \end{equation*}
\end{lemma}
\begin{proof}
    First, from Sobolev embedding and Sobolev interpolation, one has
    \begin{equation}\label{E5.3}
    \begin{aligned}
        \mode{\nabla v}_{L^\infty(\ot)}\leq& \pe \mode{v}_{H^{2+\epsilon}(\ot)},\\
        \mode{\nabla\rho}_{L^\infty(\ot)}\leq& \pe \mode{\rho}_{H^{2+\epsilon}(\ot)}\leq \pe \mode{\rho}^\epsilon_{H^3(\ot)},
    \end{aligned}
    \end{equation}
 for $\epsilon>0$ small enough. 
      Similar as in Lemma \ref{L4.2}, 
     one obtains
    \begin{equation}\label{E5.62}
        \begin{aligned}
            \mode{\gt}_{H^4}+\mode{\nt}_{H^3(\gt)}\leq 
            & \pe\left(\mode{\rho}^2_{H^\frac{5}{2}(\ot)}+1\right),
        \end{aligned}
    \end{equation} as well as 
     \begin{equation}\label{E5.4}
        \begin{aligned}
            \sigma\mode{\nat v\cdot\nt}_{H^{\frac{3}{2}+\epsilon}(\gt)}\leq& \pe \left(\sigma\mode{\dt\kappa}_{H^{\frac{1}{2}+\epsilon}(\gt)}+\sigma\mode{(\nat v\cdot\at)\kappa}_{H^{\frac{1}{2}+\epsilon}(\gt)}+1\right)\\\leq &\pe\left(\mode{\dt\rho}^\epsilon_{H^2(\ot)}+\mode{v}_{H^{2+\epsilon}(\ot)}+1 \right).
        \end{aligned}
    \end{equation}
   
The boundary condition   $\sigma\dt^2\kappa=\rho\dt^2\rho+(\dt\rho)^2$ on $\gt$ implies that 
\begin{equation}\label{E5.46}
    \sigma\lgt{\dt^2\kappa}\leq \mode{\rho\dt^2\rho}_{L^2(\gt)}+\mode{\dt\rho}^2_{L^4(\gt)}\leq  \pe\left(E^\frac{1}{4}_h(t)+1\right).
\end{equation}
It follows from the equation \eqref{E3.3} that \begin{align*}
    \nat\left(\nat\nnt\rho+2(\nat v\cdot\nt)(\nat v\cdot\at)\right)=f\quad \text{ on }\gt,
\end{align*}with \begin{align*}
    f:=\dt^2\kappa-2(\nat v\cdot\at)\nat(\nat v\cdot\nt)+\Delta_\gt \nt\cdot\nabla\rho-\kappa\left(2(\nat v\cdot\at)^2-(\nat v\cdot \nt)^2\right).
\end{align*}
Therefore, applying \eqref{E5.46} and Lemma \ref{L4.2} to  the definition of $f$,  we obtain  \begin{align*}
    \lgt{f}\leq \pe\left(E^\frac{1}{4}_h(t)+\mode{v}_{H^{2+\epsilon}(\ot)}+1\right),
\end{align*}which in turn implies that 
\begin{equation}\label{E5.63}
    \begin{aligned}
        \mode{\nnt\rho}_{H^{\frac{3}{2}+\epsilon}(\gt)}\leq \pe\left(E^\frac{1}{4}_h(t)+\mode{v}_{H^{2+\epsilon}(\ot)}+1\right).
    \end{aligned}
\end{equation}

Upon applying the elliptic estimate Lemma \ref{L2.8} to \eqref{E3.2}, together with Sobolev interpolation, yields for $\theta_i\in (0,\frac{\pi}{2})$, $\delta_0'\in (0,\frac{1}{2}],$
\begin{equation}\label{E5.60}
    \begin{aligned}
         \mode{\rho}_{H^3(\ot)}\leq & C\left(\mode{\gt}_{H^3}\right)\left(\mode{\frac{1}{\rho}\dt^2\rho}_{H^1(\ot)}+\mode{\tr(\nabla v)^2}_{H^1(\ot)}+\mode{\frac{1}{\rho^2}(\dt\rho)^2}_{H^1(\ot)}\right)\\&+C\left(\mode{\gt}_{H^3}\right)\left(\mode{\nnt\rho}_{H^\frac{3}{2}(\gt)}+\mode{\nabla_\nb\rho}_{H^\frac{3}{2}(\gb)}+\mode{\rho}_{L^2(\ot)}\right)\\\leq& \pe\left(E_h^{\delta'_0}(t)+\mode{v}_{H^{2+\epsilon}(\ot)}+1\right).
    \end{aligned}
\end{equation}

On the other hand, we search for the regularity of $\dt\rho$,
  we use the commutator $[\nnt,\dt]$  from Lemma \ref{L2.1} to find that 
\begin{align*}
   \mode{\nnt\dt\rho}_{H^\frac{1}{2}(\gt)}\leq & \mode{[\nnt,\dt]\rho}_{H^\frac{1}{2}(\gt)}+\mode{\dt\nnt\rho}_{H^\frac{1}{2}(\gt)}\\ \leq & \pe \left(\mode{v}_{H^{2+\epsilon}(\ot)}+E^\frac{1}{2}_h(t)+\mode{\dt\rho}^\frac{1}{2}_{H^2(\ot)}+1\right)\\ \leq & \frac{\delta}{C\left(\mode{\gt}_{H^3}\right)}\mode{\dt\rho}_{H^2(\ot)} +\pe \left(E^{\delta_0'}_h(t)+\mode{v}_{H^{2+\epsilon}(\ot)}+1\right),
\end{align*}for some constant $\delta\in (0,1)$. 
Here we have used  Sobolev interpolation and Young's inequality.   By  the elliptic estimate Lemma \ref{L2.8} as well as \eqref{E4.1}$_3$, for  $\theta_i\in (0,\pi/2)$, we have 
    \begin{equation*}
        \begin{aligned}
             \mode{\dt\rho}_{H^2(\ot)}\leq& C\left(\mode{\gt}_{H^3}\right)\left(\mode{\frac{1}{\rho}\dt^3\rho-\frac{1}{\rho}\dt\left(\rho\tr(\nabla v)^2+\frac{1}{\rho}(\dt\rho)^2\right)-\frac{1}{\rho}\dt\rho\Delta\rho+[\dt,\Delta]\rho}_{L^2(\ot)}\right)\\&+C\left(\mode{\gt}_{H^3}\right)\left(\mode{\nabla_\nt\dt\rho}_{H^\frac{1}{2}(\gt)}+\mode{\nabla_\nb\dt\rho}_{H^\frac{1}{2}(\gb)}+\mode{\dt\rho}_{L^2(\ot)}\right)\\ 
             \leq &\pe\left(E_h^{\delta'_0}(t)+\mode{v}_{H^{2+\epsilon}(\ot)}+\mode{\rho}_{H^{2+\epsilon}(\ot)}+1\right)+\delta\mode{\dt\rho}_{H^2(\ot)},
        \end{aligned}
    \end{equation*}
    choosing suitable $\delta$, and invoking \eqref{E5.60}, we get \begin{equation}\label{E5.57}
        \mode{\dt\rho}_{H^2(\ot)}\leq \pe \left(E_h^{\delta'_0}(t)+\mode{v}_{H^{2+\epsilon}(\ot)}+1\right).
    \end{equation}

Proceeding as in Lemma \ref{L4.2} to investigate the regularity of $v$, utilizing Lemma \ref{L2.8}, \eqref{E4.7} and  \eqref{E5.60}-\eqref{E5.57}, for $\epsilon\in (0,\pi/\theta_i-2)$ small enough,  one has 
\begin{equation*}
    \begin{aligned}
        \mode{v}_{H^{2+\epsilon}(\ot)}\leq &C\left(\mode{\gt}_{H^{3+\epsilon}}\right) \left(\mode{\frac{1}{\rho}\dt\rho}_{H^{1+\epsilon}(\ot)}+\mode{\omega}_{H^{1+\epsilon}(\ot)}+\mode{v\cdot\nt}_{H^{\frac{3}{2}+\epsilon}(\gt)}+\mode{\mathcal{V}}_{L^2(\ot)}+\mode{\mathcal{U}}_{L^2(\ot)}\right)\\ \leq &\pe\left(E_h^{\epsilon\delta'_0}(t)+\mode{v}_{H^{2+\epsilon}(\ot)}^{\epsilon}+1\right),
    \end{aligned}
\end{equation*}
here we have used $\nat (v\cdot\nt)=\nat v\cdot\nt-v\cdot\nat\nt$.
We obtain the desired estimate by Young's inequality, \begin{equation}\label{E5.11}
    \mode{v}_{H^{2+\epsilon}(\ot)}\leq \pe\left(E^{\epsilon\delta_0'}_h(t)+1\right).
\end{equation}
Substituting \eqref{E5.11} into \eqref{E5.3}-\eqref{E5.4} and  \eqref{E5.63}-\eqref{E5.57}   completes the proof.
\end{proof}
\begin{remark}\label{R5.1}
    In view of the estimates for $\mode{\nat v\cdot\nt}_{H^{\frac{3}{2}+\epsilon}(\gt)}$ and $\mode{\nnt\rho}_{H^{\frac{3}{2}+\epsilon}(\gt)}$ provided in  Lemma \ref{L4.3},  together with the Sobolev embedding, the following hold at the contact points $X_i(i=l,r)$, \begin{equation*}
        \llx{\nat\nnt\rho}\leq \pall, \text{ and } \llx{\nat(\nat v\cdot\nt)}\leq \pall.
    \end{equation*}
\end{remark}
With the preparations above, we are ready to derive equation for $\dt^3\kappa$. 
\begin{lemma}\label{L5.2}
    The following equation holds on $\gt$, 
    \begin{equation}\label{E7.14}
        \dt^3\kappa-\nat\left(\nat\dt\nnt\rho-2(\nat v\cdot\nt)\Delta_\gt\rho\right)=\mathfrak{R}_3
    \end{equation}
    with the remainder term defined as \begin{equation}\label{E5.33}
        \begin{aligned}
            \mathfrak{R}_3=&-6(\nat v\cdot\at)\Delta_\gt\nnt\rho+\nat\rho\Delta_\gt (\nat v\cdot\nt)-2\Delta_\gt\rho\nat (\nat v\cdot\nt)\\&+(-18(\nat v\cdot \nt)(\nat v\cdot\at)-3\nat\nnt\rho+3\nat\rho\kappa)\nat\left(\nat v\cdot\at\right) \\&+|\kappa|^2\dt\nnt\rho-(\nat \kappa)\nat\dt\rho+\mathfrak{R}_1\left(\nat\nnt\rho+\dt\kappa+1\right),
        \end{aligned}
    \end{equation}here $\mathfrak{R}_1:= R(\nabla\rho,\kappa,\nat v\cdot\nt,\nabla_\at v\cdot\at)$ denote some generic at most one-order  terms and adhere to a similar structure pattern in each context.
    
\end{lemma}
\begin{proof}
    
    In fact, taking $\dt$ on the curvature equation \eqref{E3.3} leads to
    \begin{equation}\label{E5.30}
        \dt^3\kappa-\dt\Delta_\gt\nnt\rho=\dt\mathfrak{R}.
    \end{equation}
   Special attention should be paid to the terms $\nat v\cdot\at$ and $\Delta_\gt v\cdot\at$, which  arise  in the commutator $[\dt,\Delta_\gt]$.
   Using Lemma \ref{L2.1} to get that\begin{equation}\label{E5.31}
        \begin{aligned}
            -\dt\Delta_\gt\nnt\rho=-\Delta_\gt\dt\nnt\rho+2(\nat v\cdot\at)\Delta_\gt\nnt\rho+\nat\left(\nat v\cdot\at\right)\nat\nnt\rho.
        \end{aligned}
    \end{equation}

    Before proceeding, we simplify the following variables using  $\nat \nt=\kappa\at$ and $\nat \at=-\kappa\nt$,
    \[\Delta_\gt\nt=-|\kappa|^2\nt+(\nat\kappa)\at=-|\kappa|^2\nt+\frac{\rho}{\sigma}\nat\rho\at=\mathfrak{R}_1,\]and \[\Delta_\gt \at=-|\kappa|^2\at-(\nat\kappa)\nt=\mathfrak{R}_1,\]which directly leads to 
    \begin{align*}
        \Delta_\gt\nabla\rho\cdot\at&=\Delta_\gt\nat \rho-\Delta_\gt\at\cdot\nabla\rho+2\kappa \left(\nat\nnt\rho-\kappa\nat\rho\right)=\Delta_\gt\nat\rho+\mathfrak{R}_1\left(\nat\nnt\rho+1\right),\\
        \Delta_\gt\nabla\rho\cdot\nt&=\Delta_\gt \nnt\rho-\Delta_\gt\nt\cdot\nabla\rho-2\kappa(\Delta_\gt \rho+\kappa\nnt\rho)=\Delta_\gt \nnt\rho-2\kappa\Delta_\gt\rho+\mathfrak{R}_1.
    \end{align*}
    Also, $\Delta_\gt v\cdot\nt$ could be transformed into $-\dt\kappa+\mathfrak{R}_1$ through \eqref{E3.1}. Therefore,  using the calculations  \eqref{E3.45}-\eqref{E3.46}, together with the above identities, the following records some basic computations of $\dt\mathfrak{R}$:
    \begin{equation}\label{E5.12}
        \begin{aligned}
             4\dt \left((\nat v\cdot\at)\Delta_\gt  v\cdot\nt\right)=&-4(\nat v\cdot\at)\Delta_\gt\nnt\rho-8(\nat v\cdot\nt)(\nat v\cdot\at)\nat\left(\nat v\cdot\at\right)\\&-4\Delta_\gt\rho \nat \left(\nat v\cdot\nt\right)
             +12\kappa (\nat v\cdot\at)\Delta_\gt\rho+\mathfrak{R}_1\left(1+\dt\kappa\right),
        \end{aligned}
    \end{equation}
    \begin{equation}\label{E5.167}
        \begin{aligned}
             2\dt\left((\nat v\cdot\nt)\Delta_\gt v\cdot \at\right)=&-2(\nat v\cdot\nt)\Delta_\gt \nat\rho+\mathfrak{R}_1\left(1+\dt\kappa+\nat\nnt\rho\right)\\&+\left(-2\nat\nnt\rho+2\nat\rho\kappa-10(\nat v\cdot\at)(\nat v\cdot\nt)\right)\nat \left(\nat v\cdot\at\right),
        \end{aligned}
    \end{equation}
    as well as 
    \begin{equation}
        \begin{aligned}
             -\dt\left(\Delta_\gt\nt\cdot\nabla\rho\right)=&\nat\rho\Delta_\gt (\nat v\cdot\nt)+\kappa\nat\rho\nat \left(\nat v\cdot\at\right)\\&+|\kappa|^2\dt\nnt\rho-(\nat\kappa)\nat\dt\rho+\mathfrak{R}_1\left(1+\dt\kappa\right).
        \end{aligned}
    \end{equation}
    The other two terms can be separated similarly, 
    \begin{align}
        \dt(6(\nat v\cdot\at)\nat\nt\cdot\nat v)=&6\dt\left((\nat v\cdot\at)^2\kappa\right)=-12\kappa(\nat v\cdot\at)\Delta_\gt\rho+\mathfrak{R}_1(1+\dt\kappa),\\
        \dt\left(3(\nat v\cdot\nt)\nat\at\cdot\nat v\right)=&-3\dt\left((\nat v\cdot\nt)^2\kappa\right)=\mathfrak{R}_1(1+\dt\kappa+\nat\nnt\rho).\label{E5.19}
    \end{align}

    As a result, $\dt\mathfrak{R}$ can be obtained by combining \eqref{E5.12}-\eqref{E5.19}. Note that the first term on the RHS of \eqref{E5.167} can be written as \[-2(\nat v\cdot\nt)\Delta_\gt\nat\rho=-\nat(2(\nat v\cdot\nt)\Delta_\gt\rho)+2\Delta_\gt\rho\nat(\nat v\cdot\nt).\]The higher-order term $\nat (2(\nat v\cdot\nt)\Delta_\gt\rho)$ and $-\Delta_\gt \dt\nnt\rho$ (in \eqref{E5.31})  are  retained in the principal part of the curvature equation, while the remaining terms  are grouped together as the remainder $\mathfrak{R}_3$ and placed on the RHS of the higher-order curvature equation \eqref{E5.30}. This completes the proof.

\end{proof}

 The following lemma 
 proceed to analyze the dissipation effects generated at the corner points $X_i$ and to estimate the remaining terms. 
\begin{lemma}\label{L5.3}
    We have at the contact points $X_i(i=l,r)$ the equations 
    \begin{equation}\label{E5.48}
        \pm\dt^2\nnt\rho=-\frac{\sigma}{\beta_c}(\nat \dt\nnt\rho-2(\nat v\cdot\nt)\Delta_\gt \rho)(\sin\theta_i)^2+\widetilde{r}_i\quad \text{ at }X_i(i=l,r),
    \end{equation}
     where the remainder terms \[\widetilde{r}_i:=\mathfrak{R}_1\left(\nat\nnt\rho+\dt\nat\rho+\dt\nnt\rho+\dt\kappa+1\right)\] satisfy the following estimates \[\llx{\widetilde{r}_i}\leq \pall .\]
    Moreover, with the high-order dissipation $F_h(t)$ defined in \eqref{E5.47}, it holds that 
    \begin{align}
         \llx{\dt^2\nnt\rho}\leq &\frac{\sigma}{\beta_c}F_h^\frac{1}{2}(t)+\pall,\label{E5.49}\\\dt^2\nnt\rho\left(\nat\dt\nnt\rho-2(\nat  v\cdot\nt)\Delta_\gt\rho\right)\big|_{X_r}^{X_l}\leq& -\frac{\sigma}{2\beta_c}F_h(t)+\pall .\label{E5.50}
    \end{align}
\end{lemma}
\begin{proof}
 Before proceeding, suppose  the bottom $\gb$ is sufficiently smooth  that the regularity of $\nb$ can be ignored. The Euler equation gives \begin{align*}
    \dt^2\left(v\cdot\nabla_v\nb\left(\nb\cdot\nt\right)\right)=\mathfrak{R}_1\left(1+\nat\nnt\rho+\dt\nabla\rho\right).
\end{align*}
Then, 
    applying the material derivative $\dt$ to  \eqref{E3.42} and employing arguments analogous to those in Lemma \ref{L5.2}, together with \eqref{E3.44}-\eqref{E3.46}, directly yields \eqref{E5.48}.
    Recall that $\dt\nnt\rho|_{X_i}$ and  $\dt\nat\rho|_{X_i}$  can be   controlled through Lemma \ref{L4.10}, applying Sobolev embedding and  Lemma \ref{L4.3}, we deduce that
    \begin{equation*}
        \begin{aligned}
            \llx{\widetilde{r}_i}\leq& \pall \left(\mode{\nnt\rho}_{H^{\frac{3}{2}+\epsilon}(\gt)}+\mode{v}_{H^{2+\epsilon}(\ot)}+\mode{\rho}_{H^{2+\epsilon}(\ot)}+\mode{\dt\rho}_{H^{1+\epsilon}(\ot)}+1\right) \leq \pall.
        \end{aligned}
    \end{equation*}
    A direct computation shows that  \eqref{E5.49}-\eqref{E5.50} hold. 
\end{proof}

The higher-order curvature equation allows us to control certain classes of boundary integrals on $\gt$ via energy and dissipation.

\begin{lemma}\label{L4.7}  Assume that  $\theta_i\in (0,\pi/2)$,  it holds that 
    \begin{align*}
        \int_\gt\rho\dt^3\rho\psi\id S\leq &\mode{\psi}_{H^1(\gt)}\left( C_{c_1,\sigma}F^\frac{1}{2}_h(t)+\pall\right),
    \end{align*}
 for some function $\psi\in H^1(\gt)$ and some constant $C_{c_1,\sigma}$ depends only on $c_1$ and $\sigma$.
\end{lemma}
\begin{proof}
For any $\psi\in H^1(\gt)$, one can derive from Lemma \ref{L5.2} that  \begin{equation*}\label{E4.15}
   \begin{aligned}
       \rho\dt^3\rho=&\sigma\dt^3\kappa-3\dt\rho\dt^2\rho\\=&\sigma\nat\left(\nat\dt\nnt\rho-2(\nat v\cdot\nt)\Delta_\gt\rho\right)+\sigma\mathfrak{R}_3-3\dt\rho\dt^2\rho \quad \text{on }\gt,
   \end{aligned}
\end{equation*}
 Integration by parts shows that
 \begin{equation}\label{E5.90}
     \begin{aligned}
         \int_\gt \rho\dt^3\rho\psi\id S=&\sigma(\nat\dt\nnt\rho-2(\nat v\cdot\nt)\Delta_\gt\rho)\psi\big|_{X_r}^{X_l}+\int_\gt \left(\sigma\mathfrak{R}_3-3\dt\rho\dt^2\rho\right)\psi \id S\\&-\sigma\int_\gt (\nat\dt\nnt\rho-2(\nat v\cdot\nt)\Delta_\gt\rho)\nat\psi\id S.
     \end{aligned}
 \end{equation}

The treatment of the corner term 
 follows directly from the definition of $F_h$, yielding
 \begin{equation}\label{E5.28}
    \sigma(\nat\dt\nnt\rho-2(\nat v\cdot\nt)\Delta_\gt\rho)\psi\big|_{X_r}^{X_l}\leq  C_{c_1,\sigma} F^\frac{1}{2}_h(t)\mode{\psi}_{H^1(\gt)},
 \end{equation}
 where $c_1$ is given by the  assumption \eqref{E4.23}, and  the constant $C_{c_1,\sigma}$ depends only on the coefficients $c_1$ and  $\sigma$.  
 Meanwhile, the estimate of the last integral term  is standard,
 \begin{equation}\label{E5.29}
     \begin{aligned}
     &   -\int_\gt (\nat\dt\nnt\rho-2(\nat v\cdot\nt)\Delta_\gt\rho)\nat\psi\id S\\\leq & \left(\mode{\nat \dt\nnt\rho}_{L^2(\gt)}+2\mode{\nat v\cdot\nt}_{L^\infty(\gt)}\mode{\Delta_\gt\rho}_{L^2(\gt)}\right)\mode{\nat\psi}_{L^2(\gt)}\\\leq&\pall\mode{\psi}_{H^1(\gt)}.
     \end{aligned}
 \end{equation}

We  concentrate on the higher-order components of  the remainder $\mathfrak{R}_3$ (defined in \eqref{E5.33}), as these demand careful handling,
\begin{equation}\label{E5.36}
    \begin{aligned}
        \int_\gt \left(\sigma\mathfrak{R}_3-3\dt\rho\dt^2\rho\right)\psi \id S\leq& -6\int_\gt(\nat v\cdot\at)\Delta_\gt\nnt\rho \psi\id S+\int_\gt \nat\rho\Delta_\gt(\nat v\cdot\nt)\psi\id S \\&+\int_\gt \left(-18(\nat v\cdot\nt)(\nat v\cdot\at)-3\nat\nnt\rho+3\nat\rho\kappa\right)\nat(\nat v\cdot\at) \psi\id S   \\&+\pall \mode{\psi}_{L^2(\gt)}.
    \end{aligned}
\end{equation}
Inspired by Remark \ref{R5.1}, applying Lemma \ref{L2.6} and Lemma \ref{L2.7}, yields  \begin{equation}\label{E5.340}
    \begin{aligned}
        -6\int_\gt(\nat v\cdot\at)\Delta_\gt\nnt\rho \psi\id S=&6\int_\gt \xc{\nat\nnt\rho}\nat(\psi\nat v\cdot\at)\id S\\\leq & \tha{\xc{\nat\nnt\rho}}\thb{\nat(\psi\nat v\cdot\at)}
        \\\leq &\pall \mode{\psi}_{H^{\frac{1}{2}}(\gt)}.
    \end{aligned}
\end{equation}
Using integration by parts, we can achieve 
\begin{equation}
    \begin{aligned}
        \int_\gt \nat\rho\Delta_\gt(\nat v\cdot\nt)\psi\id S=&\left.\nat\rho\nat(\nat v\cdot\nt)\psi\right|_{X_r}^{X_l}-\int_\gt\Delta_\gt\rho \nat(\nat v\cdot\nt)\psi\id S\\&-\int_\gt\nat\rho\nat(\nat v\cdot\nt)\nat\psi\id S\\\leq&\pall \mode{\psi}_{H^1(\gt)}. 
    \end{aligned}
\end{equation}

Due to the regularity of $v$, we cannot directly apply the Sobolev trace theorem here, instead using the similar argument as \eqref{E5.340}, denote $f:=-18(\nat v \cdot\nt)(\nat v\cdot\at)-3\nat\nnt\rho+3\nat\rho\kappa$, so that \begin{equation}\label{E5.35}
    \begin{aligned}
        \int_\gt f\nat(\nat v\cdot\at)\psi\id S&=\int_\gt f\nat\xc{\nat v\cdot\at}\psi\id S\\&\leq \tha{\xc{\nat v\cdot\at}}\thb{\nat(f\psi)}\\&\leq \pall\mode{\psi}_{H^1(\gt)}.
    \end{aligned}
\end{equation}


We can sum up \eqref{E5.90}-\eqref{E5.35} to get the desired result.  
\end{proof}
Establishing the  higher-order energy estimate is nevertheless more challenging, as it involves integration by parts with respect to the material derivative $\dt$. To facilitate a clear presentation, we  state the following proposition, which offers a consolidated approach for estimating the associated remainder $\mathfrak{R}_3$ (defined in \eqref{E5.33}).
  Due to the limited spatial regularity and  the unknown contact point values, these estimates are fairly delicate and require a good deal of care.
\begin{proposition}\label{Prop5.1}
   Assume that  $\theta_i\in(0,\pi/2)$, then the following estimate holds:
    \begin{equation*}
        \begin{aligned}
          \int_0^T  \int_\gt \dt^2\nnt\rho\mathfrak{R}_3\id S\id t\leq &\eo+\sup_{t\in [0,T]}\pe\left(E^{\delta_0'+\epsilon}_h(t)+1\right)\\&+\int_0^T\pall\left(F^\frac{1}{2}_h(t)+1\right)\id t,
        \end{aligned}
    \end{equation*} 
    for some constants $\delta'_0\in(0,\frac{1}{2}]$ and  $\epsilon\in (0,\pi/\theta_i-2)$ small enough.
\end{proposition}
\begin{proof}
    For the sake of brevity, throughout the proof we will repeatedly  use  five essential tools without explicitly referring to them: the Reynolds transport theorem in Lemma \ref{L2.3}, the standard Sobolev embedding,   the catalogs of lower-order regularity in Lemma \ref{L4.2} and higher-order regularity in Lemma \ref{L4.3},  the contact point value estimates in Remark \ref{R5.1} and Lemma \ref{L5.3}. We decompose the remainder integral into several terms based on distinct regularity:
    \begin{equation*}
        \begin{aligned}
             \int_\gt\dt^2\nnt\rho\mathfrak{R}_3\id S=& \int_\gt\dt^2\nnt\rho\left(-6(\nat v\cdot\at)\Delta_\gt\nnt\rho-3\nat\nnt\rho\nat(\nat v\cdot\at)\right)\id S\\&+\int_\gt\dt^2\nnt\rho\left(\nat\rho\Delta_\gt (\nat v\cdot\nt)-2\nat (\nat v\cdot\nt)\Delta_\gt\rho\right)\id S
             \\&+\int_\gt\dt^2\nnt\rho\left(-18(\nat v\cdot \nt)(\nat v\cdot\at)+3\nat\rho\kappa\right)\nat(\nat v\cdot\at)\id S
             \\&+\int_\gt\dt^2\nnt\rho(-\nat \kappa)\nat\dt\rho\id S\\&+ \int_\gt\dt^2\nnt\rho|\kappa|^2\dt\nnt\rho\id S\\&+ \int_\gt\dt^2\nnt\rho\mathfrak{R}_1\left(\nat\nnt\rho+\dt\kappa+1\right)\id S \\=:&\mathcal{J}_{1}+\mathcal{J}_{2}+\mathcal{J}_{3}+\mathcal{J}_{4}+\mathcal{J}_{5}+\mathcal{J}_{6}.
        \end{aligned}
    \end{equation*}
    In the subsequent estimation of the lower-order terms in $\mathfrak{R}_3$, we will proceed term by term, following a progression from simplicity to complexity  and from  lower to higher order.

\textbf{(1) Term $\mathcal{J}_6$.}
We now focus on the simplest one $\mathcal{J}_6$. An integration by parts in time shows that 
\begin{align*}
   & \int_\gt\dt^2\nnt\rho\mathfrak{R}_1\left(\nat\nnt\rho+\dt\kappa+1\right)\id S\\=&\ddt \int_\gt\dt\nnt\rho\mathfrak{R}_1\left(\nat\nnt\rho+\dt\kappa+1\right)\id S-\int_\gt\dt\nnt\rho(\dt\mathfrak{R}_1+\mathfrak{R}_1\nat v\cdot\at)\left(\nat\nnt\rho+\dt\kappa+1\right)\id S\\&-\int_\gt\dt\nnt\rho\mathfrak{R}_1\left(\nat\dt\nnt\rho+\dt^2\kappa\right)\id S.
\end{align*} With the help of $\mode{\dt\nnt\rho}_{L^2(\gt)}\leq \pe \left(E_h^{\frac{1}{2}\delta_0'}(t)+1\right)$, one arrives at \begin{align*}
   & \int_\gt\dt\nnt\rho\mathfrak{R}_1\left(\nat\nnt\rho+\dt\kappa+1\right)\id S\\\leq& \lgt{\dt\nnt\rho}\mode{\mathfrak{R}_1}_{L^\infty(\gt)}\left(\lgt{\nat\nnt\rho}+\lgt{\dt\kappa}+\lgt{\mathfrak{R}_1}\right)\\\leq& \pe \left(E_h^{\frac{1}{2}\delta_0'+\epsilon}(t)+1\right).
\end{align*}
A direct calculation shows that 
\begin{equation}\label{E10.10}
    \begin{aligned}
         \int_0^T \mathcal{J}_6\id t
           \leq &\eo+\sup_{t\in[0,T]}\pe\left(E_h^{\frac{1}{2}\delta_0'+\epsilon}(t)+1\right)+\int_0^T\pall \lgt{\dt\mathfrak{R}_1}\id t\\&+\int_0^T\pall \left(\lgt{\nat\dt\nnt\rho}+\lgt{\nat\nnt\rho}+\lgt{\dt^2\kappa}+1\right) \id t\\
         \leq &\eo+\sup_{t\in[0,T]}\pe\left(E_h^{\frac{1}{2}\delta_0'+\epsilon}(t)+1\right)+\int_0^T\pall\id t.
    \end{aligned}
\end{equation}

\textbf{(2) Term $\mathcal{J}_5$.} For this term we find\[ \mathcal{J}_5=\int_\gt |\kappa|^2\dt^2\nnt\rho \dt\nnt\rho\id S\leq \frac{1}{2}\ddt\int_\gt |\kappa|^2 |\dt\nnt\rho|^2\id S+\pall,\] which leads to \begin{equation*}
    \begin{aligned}
        \int_0^T\mathcal{J}_5\id t\leq \eo+\sup_{t\in[0,T]}\pe\left(E_h^{\delta_0'}(t)+1\right)+\int_0^T\pall\id t.
    \end{aligned}
\end{equation*}

\textbf{(3) Term $\mathcal{J}_4$.}
Applying the boundary condition $\rho\nat\rho=\sigma\nat\kappa$ on $\gt$, one obtains
\begin{align*}
    \mathcal{J}_4=&-\int_\gt \dt^2\nnt\rho\nat\kappa\nat\dt\rho\id S\\=&-\frac{1}{\sigma}\int_\gt \rho\nat\rho\dt^2\nnt\rho\nat\dt\rho\id S.
\end{align*}
For the sake of clarity, set $(\phi,\psi)=\left(-\frac{\rho\nat\rho}{\sigma},\dt\rho\right)$, note that 
\begin{equation*}
    \llx{\dt\phi}\leq \pall \text{ and } \llx{\dt\psi} \text{ may not exist}.
\end{equation*} In this case, $\mathcal{J}_4$ can be expanded as follows, \begin{equation*}
    \begin{aligned}
        \int_\gt \phi\dt^2\nnt\rho\nat\psi\id S=&\phi\dt^2\nnt\rho\psi\big|_{X_r}^{X_l}+\dt\phi\dt\nnt\rho\psi\big|_{X_r}^{X_l}-\ddt\int_\gt \left(\phi\nat\dt\nnt\rho\psi+\nat\phi\dt\nnt\rho \psi\right)\id S\\&+\int_\gt \phi\nat\dt\nnt\rho\dt\psi\id S+\int_\gt \dt\nnt\rho(\nat\phi\dt\psi-\dt\phi\nat\psi)\id S,
    \end{aligned}
\end{equation*}
which implies that 
\begin{equation}\label{E10.14}
    \begin{aligned}
        \int_0^T\mathcal{J}_4\id t \leq & \eo+\sup_{t\in [0,T]}\pe\left(E_h^{\delta_0'+\epsilon}(t)+1\right)+\int_0^T \pall\left(F^\frac{1}{2}_h(t)+1\right)\id t,
    \end{aligned}
\end{equation}
where $\delta_0'+\epsilon\in(0,1)$.

\textbf{(4) Term $\mathcal{J}_3$.} With the above preparations, $\mathcal{J}_3 $ can be rewritten as  \begin{align*}
    \mathcal{J}_3=&\int_\gt\dt^2\nnt\rho\left(-18(\nat v\cdot \nt)(\nat v\cdot\at)+3\nat\rho\kappa\right)\nat(\nat v\cdot\at)\id S\\=&\int_\gt (-9\nat v\cdot\nt)\dt^2\nnt\rho\nat \left((\nat v\cdot\at)^2\right)\id S+\int_\gt (3\nat \rho\kappa)\dt^2\nnt\rho\nat(\nat v\cdot\at)\id S.
\end{align*}
Choosing \[(\phi_1,\psi_1)=\left(-9\nat v\cdot\nt,(\nat v\cdot\at)^2\right),\quad (\phi_2,\psi_2)=(3\nat\rho\kappa,\nat v\cdot\at)\] respectively. Under this setting, for $k=1,2$, $\dt\phi_k|_{X_i}(i=l,r)$ exist and $\nat \psi_k\notin L^2(\gt)$.
We proceed by expanding $\mathcal{J}_3$ into the form of 
\begin{equation}\label{E5.440}
    \begin{aligned}
         \int_\gt \phi \dt^2\nnt\rho\nat \psi\id S=&\phi\dt^2\nnt\rho\psi|_{X_r}^{X_l}-\ddt\int_\gt\left(\nat\phi\dt\nnt\rho\psi+\phi\nat\dt\nnt\rho\psi\right)\id S\\&+\int_\gt \nat \dt\nnt\rho\dt(\phi\psi)\id S+\int_\gt \nat\phi\dt\nnt\rho\dt\psi\id S\\&-\int_\gt\xc{\dt\phi}\nat(\dt\nnt\rho\psi)\id S.
    \end{aligned}
\end{equation}The last integral in \eqref{E5.440} requires a different treatment \[-\int_\gt\xc{\dt\phi}\nat(\dt\nnt\rho\psi)\id S\leq \pall\mode{\dt\phi}_{H^1(\ot)}\mode{\psi}_{H^\frac{1}{2}(\gt)},\]whereas the other terms follow straightforward arguments. Therefore, 
we obtain 
\begin{equation}\label{E10.15}
    \begin{aligned}
        \int_0^T\mathcal{J}_3\id t\leq \eo+\sup_{t\in [0,T]}\pe\left(E_h^{\delta_0'+\epsilon}(t)+1\right)+\int_0^T \pall\left(F^\frac{1}{2}_h(t)+1\right)\id t,
    \end{aligned}
\end{equation}for coefficient $\delta_0'+\epsilon\in (0,1)$.

\textbf{(5) Term $\mathcal{J}_2$. }
In contrast to the previous estimates, $\mathcal{J}_2$ involves terms  of higher order, which requires more delicate treatment,
\begin{align*}
    \mathcal{J}_2=&\int_\gt\dt^2\nnt\rho\nat\rho\Delta_\gt (\nat v\cdot\nt)\id S+\int_\gt\dt^2\nnt\rho\nat (\nat v\cdot\nt)(-2\Delta_\gt\rho)\id S\\=&: \mathcal{J}_{2,1}+\mathcal{J}_{2,2}.
\end{align*}

We can now use \eqref{E3.4} to write  \[-\Delta_\gt (\nat v\cdot\nt)=\frac{1}{\sigma}(\rho\nat\dt\rho+\nat\rho\dt\rho)+\nat(\nat v\cdot\at)\kappa+\frac{1}{\sigma}\rho\nat\rho\nat v\cdot\at,\]from which we deduce
\begin{align*}
    \mathcal{J}_{2,1}=\underbrace{\int_\gt \dt^2\nnt\rho\left(-\frac{1}{\sigma}\rho\nat\rho\right)\nat\dt\rho \id S}_{\Leftrightarrow \mathcal{J}_4}+\underbrace{\int_\gt \dt^2\nnt\rho (-\nat\rho\kappa)\nat(\nat v\cdot\at)\id S}_{\Leftrightarrow \mathcal{J}_3}+\underbrace{\int_\gt \dt^2\nnt\rho\mathfrak{R}_1\id S}_{\Leftrightarrow \mathcal{J}_6},
\end{align*}
here, the symbol $\Leftrightarrow$ indicates that the corresponding integrals has already been estimated in the corresponding $\mathcal{J}_i(i=3,4,6)$. Therefore, \eqref{E10.10}, \eqref{E10.14} and \eqref{E10.15} can be used to obtain
\begin{equation}\label{E5.554}
    \begin{aligned}
        \int_0^T \mathcal{J}_{2,1}\id t\leq \eo+\sup_{t\in[0,T]}\pe\left(E_h^{\delta_0'+\epsilon}(t)+1\right)+\int_0^T\pall\left(F^\frac{1}{2}_h(t)+1\right)\id t.
    \end{aligned}
\end{equation}

For $\mathcal{J}_{2,2}$, 
a direct calculation shows that
\begin{align*}
    \mathcal{J}_{2,2}\leq& -2\ddt\int_\gt \dt\nnt\rho\nat(\nat v\cdot\nt)\Delta_\gt\rho\id S+\pall\\&+2\int_\gt \dt\nnt\rho\nat\xc{\dt(\nat v\cdot\nt)}\Delta_\gt\rho\id S\\&+2\int_\gt \dt\nnt\rho\nat(\nat v\cdot\nt)\nat\xc{\dt\nat\rho}\id S.
\end{align*}
Here we have employed the properties at the contact points:  $\llx{\dt(\nat v\cdot\nt)}\leq\pall$ and $\llx{\dt\nat\rho}\leq \pall.$ Next, applying Lemma \ref{L2.6} and Lemma \ref{L2.7}, one obtains 
\begin{align*}
    &2\int_\gt \dt\nnt\rho\nat\xc{\dt(\nat v\cdot\nt)}\Delta_\gt\rho\id S\\ \leq &\tha{\xc{\dt(\nat v\cdot\nt)}}\thb{\nat(\dt\nnt\rho\Delta_\gt\rho)}\\ \leq& \pall,
\end{align*} the other one follows  the same argument, leading to
\begin{equation}\label{E5.544}
    \int_0^T \mathcal{J}_{2,2}\id t\leq \eo+\sup_{t\in[0,T]}\pe\left(E_h^{\delta_0'+\epsilon\delta_0'}(t)+1\right)+\int_0^T\pall\id t.
\end{equation}

Consequently, combining \eqref{E5.554} and  \eqref{E5.544}, we can conclude that
\begin{equation*}
    \begin{aligned}
        \int_0^T\mathcal{J}_2\id t\leq &\eo+\sup_{t\in[0,T]}\pe\left(E_h^{\delta_0'+\epsilon}(t)+1\right)+\int_0^T\pall\left(F^\frac{1}{2}_h(t)+1\right)\id t.
    \end{aligned}
\end{equation*}

\textbf{(6) Term $\mathcal{J}_1$.} The final integral $\mathcal{J}_1$ can be rewritten as 
\begin{equation}\label{E5.588}
    \begin{aligned}
         \mathcal{J}_1=&\int_\gt\dt^2\nnt\rho\left(-6(\nat v\cdot\at)\Delta_\gt\nnt\rho-3\nat\nnt\rho\nat(\nat v\cdot\at)\right)\id S\\=&-3\nat v\cdot\at\dt^2\nnt\rho\nat\nnt\rho\big|_{X_r}^{X_l}+3\ddt\int_\gt\nat\dt\nnt\rho(\nat v\cdot\at)\nat\nnt\rho\id S\\&-3\int_\gt \nat \dt\nnt\rho\dt\left(\nat v\cdot\at\nat\nnt\rho\right)\id S-3\int_\gt \dt^2\nnt\rho (\nat v\cdot\at)\Delta_\gt\nnt\rho\id S.
    \end{aligned}
\end{equation}
We now focus solely on the final integral term. The first three terms do not pose a difficulty, as they can be handled in the same manner as discussed above. Since the last term cannot be directly controlled in terms of $\Delta_\gt\nnt\rho\notin L^2(\gt)$ and $\nat(\nat v\cdot\at)\notin L^2(\gt)$, utilizing  the fact that \eqref{E3.3} and \eqref{E4.1}$_2$, we obtain
 \[\Delta_\gt\nnt\rho=\frac{1}{\sigma}\rho\dt^2\rho+\frac{1}{\sigma}(\dt\rho)^2-\mathfrak{R}\quad \text{ on }\gt,\]with $\mathfrak{R}$ defined in \eqref{E3.5}, which  directly leads to the following decomposition 
\begin{align*}
    -3\int_\gt\dt^2\nnt\rho(\nat v\cdot\at)\Delta_\gt\nnt\rho\id S=& -\frac{3}{\sigma}\int_\gt\nt\cdot\dt^2\nabla\rho(\nat v\cdot\at)\rho\dt^2\rho\id S\\&+\frac{3}{\sigma}\int_\gt [\nt,\dt^2]\cdot\nabla\rho(\nat v\cdot\at)\rho\dt^2\rho\id S\\&+3\int_\gt \dt^2\nnt\rho(\nat v\cdot\at)\left(\mathfrak{R}-\frac{1}{\sigma}(\dt\rho)^2\right)\id S\\=&:\mathcal{J}_{1,1}+\mathcal{J}_{1,2}+\mathcal{J}_{1,3}.
\end{align*}

For $\mathcal{J}_{1,1}$, let $\mathcal{H}\left(\nat v\cdot\at\right)$ denote the harmonic extension of $\nat v\cdot\at$, we use Green’s formula to find 
\begin{equation*}
    \begin{aligned}
        \mathcal{J}_{1,1}=&-\frac{3}{\sigma}\int_\ot\nabla\cdot\dt^2\nabla\rho\mathcal{H}\left(\nat v\cdot\at\right)\rho\dt^2\rho\id X-\frac{3}{\sigma}\int_\ot\dt^2\nabla\rho\cdot\nabla\left(\mathcal{H}\left(\nat v\cdot\at\right)\rho\dt^2\rho\right)\id X\\&+\frac{3}{\sigma}\int_\gb\nb\cdot\dt^2\nabla\rho  \mathcal{H}\left(\nat v\cdot\at\right)\rho\dt^2\rho\id S\\=&:\mathcal{J}_{1,1a}+\mathcal{J}_{1,1b}+\mathcal{J}_{1,1c}.
    \end{aligned}
\end{equation*}Due to the equations \eqref{E3.2}, 
 we see that \begin{equation*}
    \begin{aligned}
           \rho\nabla\cdot\dt^2\nabla\rho=&\dt^4\rho-\dt^2\left(\tr\left(\rho(\nabla v)^2\right)+\frac{1}{\rho}(\dt\rho)^2\right)-[\dt^2,\rho\nabla\cdot]\nabla\rho,
    \end{aligned}
\end{equation*}which leads to \begin{equation*}
    \begin{aligned}
        \mathcal{J}_{1,1a}
        \leq& 
        -\frac{3}{\sigma}\ddt\int_\ot\dt^3\rho \mathcal{H}\left(\nat v\cdot\at\right)\dt^2\rho\id X+\pall.
    \end{aligned}
\end{equation*}
For $p=\frac{2}{1-\epsilon},q=\frac{2}{\epsilon}$ satisfy $\frac{1}{p}+\frac{1}{q}=\frac{1}{2}$, yields   \begin{equation*}
    \begin{aligned} 
        \mathcal{J}_{1,1b}\leq &C_\sigma\mode{\dt^2\nabla\rho}_{L^2(\ot)}\mode{\nabla(\rho\mathcal{H}\left(\nat v\cdot\at\right))}_{L^p(\ot)}\mode{\dt^2\rho}_{L^q(\ot)}\\&+C_\sigma\mode{\dt^2\nabla\rho}_{L^2(\ot)}\mode{\rho\mathcal{H}\left(\nat v\cdot\at\right)}_{L^\infty(\ot)}\mode{\nabla\dt^2\rho}_{L^2(\ot)}\\\leq &\pall, 
    \end{aligned}
\end{equation*}which is slightly more delicate than other arguments due to the Sobolev embedding $H^{\epsilon}(\ot) \hookrightarrow L^p(\ot)$.

For $\mathcal{J}_{1,1c}$, thanks to $\dt^2v=-\dt\nabla\rho$, as well as  the boundary conditions $\nabla_\nb\rho|_\gb=v\cdot\nabla_v \nb-\g e_2\cdot \nb$, one can find
\begin{equation}\label{E5.622}
    \begin{aligned}
          \nb\cdot \dt^2\nabla\rho=&\dt^2\nabla_\nb\rho-\dt^2\nb\cdot\nabla\rho-2\dt\nb\cdot\dt\nabla\rho\\=& \dt^2(v\cdot\nabla_v \nb-\g e_2\cdot \nb)-\dt(\nabla_v\nb)\cdot\nabla\rho-2\nabla_v\nb\cdot\dt\nabla\rho\\=&(-3\nabla_v\nb\cdot\ab-\nabla_\ab\nb\cdot v)\nabla_\ab\dt\rho+R_1,
    \end{aligned}
\end{equation}
where $R_1=R_{\nb}(\nabla v,\nabla \rho)$ denotes a generic remainder term that collects all one-order terms.
Here we have assumed that the bottom $\gb$ is smooth enough, so we ignore the  regularity of $\nb$.
     Thus, \begin{equation*}
        \mathcal{J}_{1,1c}\leq \pall.
    \end{equation*}
    Combining $\mathcal{J}_{1,1a}-\mathcal{J}_{1,1c}$, we conclude that\begin{equation}\label{E5.633}
        \int_0^T \mathcal{J}_{1,1}\id t\leq  \eo+\sup_{t\in[0,T]}\pe\left(E_h^{\delta_0'+\epsilon}(t)+1\right)+\int_0^T\pall\id t.
    \end{equation}

    For $\mathcal{J}_{1,2}$,
     a direct calculation shows that
    \begin{equation}\label{E2.65}
        \begin{aligned}
             \left[\nt,\dt^2\right]\cdot\nabla\rho=&-\dt^2\nt\cdot\nabla\rho-2\dt\nt\cdot\dt\nabla\rho\\
        =& 2(\nat v\cdot\nt)\nat \dt\rho+\mathfrak{R}_1(\nat\nnt\rho+1),
        \end{aligned}
    \end{equation}
    which leads to 
      \begin{equation}\label{E5.666}
         \mathcal{J}_{1,2}\leq\pall.
     \end{equation}

     For $\mathcal{J}_{1,3}$, the following expansion due to \eqref{E3.5}, 
     \begin{align*}
         \mathcal{J}_{1,3}=& \underbrace{3\int_\gt \dt^2\nnt\rho(\nat v\cdot\nt)\nat(\nat v\cdot\at)^2\id S}_{\Leftrightarrow \mathcal{J}_3}+\underbrace{\int_\gt\dt^2\nnt\rho\mathfrak{R}_1\left(1+\dt\kappa\right)\id S}_{\Leftrightarrow \mathcal{J}_6},
     \end{align*} 
which leads to 
\begin{equation}\label{E5.67}
    \int_0^T \mathcal{J}_{1,3}\id t  \leq \eo+\sup_{t\in [0,T]}\pe\left(E_h^{\delta_0'+\epsilon}(t)+1\right)+\int_0^T \pall\left(F^\frac{1}{2}_h(t)+1\right)\id t.
\end{equation}



Combining \eqref{E5.588}, \eqref{E5.633}, \eqref{E5.666} and \eqref{E5.67}, we can conclude that
\begin{equation*}
    \begin{aligned}
        \int_0^T\mathcal{J}_1\id t\leq \eo+\sup_{t\in [0,T]}\pe\left(E_h^{\delta_0'+\epsilon}(t)+1\right)+\int_0^T \pall\left(F^\frac{1}{2}_h(t)+1\right)\id t.
    \end{aligned}
\end{equation*}

Combining  the estimates from 
 $\mathcal{J}_1$ to $\mathcal{J}_6$, shows that the term $ \int_0^T\int_\gt\dt^2\nnt\rho\mathfrak{R}_3\id S\id t$ can be estimated as stated.
\end{proof}

\begin{proposition}\label{L4.11}
    Suppose that the  a priori assumptions \eqref{E4.22}-\eqref{E4.23} and \eqref{E4.24} hold.
        Let the contact angle $\theta_i\in (0,\pi/2)$. Then the following higher-order estimate holds
     \begin{align*}
        & \sup_{t\in [0,T]}E_h(t)+\int_0^T F_h(t)\id t+\int_0^T\left.\left(\frac{1}{\sin\theta_i}\left(-\nabla_\ab\rho-\g e_2\cdot\ab\right)\rho|\dt^2\rho|^2\right)\right|_{X_r}^{X_l}\id t\\\leq& \eo+ \sup_{t\in[0,T]}\pe+\int_0^T \pall\id t,
     \end{align*}  with $E_h(t)$ and $F_h(t)$ denoted by \eqref{E12} and \eqref{E5.47}, respectively.
\end{proposition}
\begin{proof}
   Taking $\dt^2 $  on both sides of \eqref{E3.2},
\begin{equation}\label{E7.21}
    \begin{cases}
      \dt^4\rho-\rho\nabla\cdot(\dt^2\nabla\rho)=\dt^2\rhs+[\dt^2,\rho\nabla\cdot]\nabla\rho&\text{ in }\ot,\\ \rho\dt^3\rho|_\gt=\sigma\dt^3\kappa-3\dt\rho\dt^2\rho,\quad \nabla_\nb\dt^2\rho|_\gb=\dt\nabla_\nb\dt\rho+[\nabla_\nb,\dt]\dt\rho,
    \end{cases}
\end{equation} 
and applying the $L^2(\ot)$ inner product in \eqref{E7.21} with $\dt^3\rho$, we obtain
\begin{equation*}
    \begin{aligned}
        &\frac{1}{2}\ddt \left(\int_\ot|\dt^3\rho|^2\id X+\int_\ot\rho|\nabla\dt^2\rho|^2\id X\right)\\=&\frac{1}{2}\int_\ot \dt\rho|\nabla\dt^2\rho|^2\id X+\int_\gt \rho\nt\cdot\dt^2\nabla\rho\dt^3\rho\id S+\int_\gb \rho\nb\cdot\dt^2\nabla\rho\dt^3\rho \id S-\int_\ot\rho\dt^2\nabla\rho\cdot[\nabla,\dt]\dt^2\rho\id X\\&-\int_\ot\nabla\rho\cdot\dt^2\nabla\rho\dt^3\rho\id X-\int_\ot \rho[\dt^2,\nabla]\rho\dt\nabla\dt^2\rho\id X +\int_\ot \dt^2\rhs\dt^3\rho\id X\\&+\int_\ot[\dt^2,\rho\nabla\cdot]\nabla\rho\dt^3\rho\id X+\frac{1}{2}\int_\ot\left(|\dt^3\rho|^2+\rho|\nabla\dt^2\rho|^2\right)\divv v\id X\\\leq& \pall\left(1+\mode{\dt[\dt^2,\nabla]\rho}_{L^2(\ot)}+\mode{\dt^2\nabla v}_{L^2(\ot)}+\mode{[\dt^2,\rho\nabla\cdot]\nabla\rho}_{L^2(\ot)}\right)+\int_\gt\rho \nt\cdot\dt^2\nabla\rho\dt^3\rho\id S\\&+\int_\gb\rho \nb\cdot\dt^2\nabla\rho\dt^3\rho \id S+\ddt\int_\ot\rho\left(\dt((\nabla v)^*\nabla\rho)+(\nabla v)^*\nabla\dt\rho\right)\nabla\dt^2\rho\id X.
    \end{aligned}
\end{equation*}
Moreover, a direct computation shows that \[\begin{aligned}
   \dt^2 \nabla v=&[\dt^2,\nabla ]v+\nabla \dt^2 v\\=&\dt[\dt,\nabla]v+[\dt,\nabla]\dt v-\nabla\dt\nabla\rho\\=&-\dt \left((\nabla v)^* \nabla v\right)-(\nabla v)^*\nabla\dt v-\nabla\left(-(\nabla v)^*\nabla \rho+\nabla\dt\rho\right),
\end{aligned}\] as well as
\begin{align*}
    [\dt^2,\rho\nabla\cdot]\nabla\rho=\dt^2\rho\Delta\rho+2\dt\rho\dt\Delta\rho-\rho\dt(\nabla v\cdot\nabla^2\rho)-\rho\nabla v\cdot\nabla\dt\nabla\rho.
\end{align*}The commutator $\dt[\dt^2,\nabla]\rho$ can be  derived similarly to the  above two identities. 
This implies that\[\mode{\dt[\dt^2,\nabla]\rho}_{L^2(\ot)}+\mode{\dt^2\nabla v}_{L^2(\ot)}+\mode{[\dt^2,\rho\nabla\cdot]\nabla\rho}_{L^2(\ot)}
\leq \pall.\]
As a result, we arrive at
\begin{equation}\label{E4.18}
    \begin{aligned}
       &\sup_{t\in[0,T]}  \frac{1}{2}\left(\int_\ot|\dt^3\rho|^2\id X+\int_\ot
\rho|\nabla\dt^2\rho|^2\id X\right)\\\leq&\eo+\sup_{t\in [0,T]}\pe\left(E_h^{\delta_0'+\epsilon}(t)+1\right)+\int_0^T\pall\id t\\&+\int_0^T\int_\gt\rho \nt\cdot\dt^2\nabla\rho\dt^3\rho\id S\id t +\int_0^T\int_\gb \rho\nb\cdot\dt^2\nabla\rho\dt^3\rho \id S\id t,
    \end{aligned}
\end{equation}
for some constants $\delta'_0\in(0,\frac{1}{2}]$ and  $\epsilon\in (0,\pi/\theta_i-2)$ small enough.

For the integral term on the bottom, recall from \eqref{E5.622} that $  \nb\cdot \dt^2\nabla\rho=(-3\nabla_v\nb\cdot\ab-\nabla_\ab\nb\cdot v)\nabla_\ab\dt\rho+R_1$ on $\gb$.  Denote $f_1:=-3\nabla_v\nb\cdot\ab-\nabla_\ab\nb\cdot v$ for the sake of brevity. Note that the bottom $\gb$ is a line segment near the contact points, which implies that $f_1$ and $R_1$ vanish near the contact points.
Therefore, we can decompose the boundary integral into two parts:
\begin{align*}\label{E5.71}
    \int_\gb \rho\nb\cdot\dt^2\nabla\rho\dt^3\rho \id S=&\int_\gb f_1\nabla_\ab\dt\rho(\rho\dt^3\rho)\id S
    +\int_\gb R_1(\rho\dt^3\rho)\id S.
\end{align*}
Using the transport identity \[\ddt \int_\gb gh \id S=\int_\gb h \dt g \id S+\int_\gb g\dt h\id S-\int_\gb v\cdot\ab\nabla_\ab(gh)\id S. \]
Then, integrating by parts with respect to time,  
 together with Lemma \ref{L4.3},  the remainder term arrives at 
\begin{equation}\label{E5.662}
    \begin{aligned}
        \int_0^T \int_\gb R_1\rho\dt^3\rho\id S\id t\leq \eo+\sup_{t\in[0,T]}\pe\left(E_h^{\frac{1}{2}\delta_0'+\epsilon}(t)+1\right)+\int_0^T\pall\id t.
    \end{aligned}
\end{equation}
On the other hand, we can expand
\begin{align*}
    \ddt \int_\gb f_1\rho\nabla_\ab\dt\rho\dt^2\rho\id S
    =& \int_\gb f_1\nabla_\ab\dt\rho(\rho\dt^3\rho)\id S+\int_\gb ([\dt,f_1\rho\nabla_\ab]\dt\rho)\dt^2\rho\id S\\&-\frac{1}{2}\int_\gb\nabla_\ab(\rho f_1) |\dt^2\rho|^2\id S+\int_\gb \nabla_\ab(v\cdot\ab)\left(f_1\rho\nabla_\ab\dt\rho\dt^2\rho\right)\id S,
\end{align*}
 here  we have used integration by parts.
Upon integrating in time from $0$ to $T$, one has
\begin{equation}\label{E5.663}
    \begin{aligned}
         & \int_0^T\int_\gb f_1\nabla_\ab\dt\rho(\rho\dt^3\rho)\id S\id t\\\leq& \eo+\sup_{t\in[0,T]}\pe\left(E_h^{\delta_0'}(t)+1\right) +\int_0^T\pall\id t.
    \end{aligned}
\end{equation}
Combining \eqref{E5.662} and \eqref{E5.663},
 we obtain that the bottom integral $ \int_0^T \int_\gb \rho\nb\cdot\dt^2\nabla\rho\dt^3\rho\id S$ shares the same bound as \eqref{E5.663}.

Now, we address the most challenging boundary integral in \eqref{E4.18} by decomposing it into three parts,
\begin{equation*}
    \begin{aligned}
        \int_\gt \rho\nt\cdot \dt^2\nabla\rho\dt^3\rho\id S=&\int_\gt\rho[\nt,\dt^2]\cdot\nabla\rho\dt^3\rho\id S+\sigma\int_\gt \dt^2\nnt\rho \dt^3\kappa\id S-3\int_\gt \dt^2\nnt\rho\dt\rho\dt^2\rho\id S\\\triangleq&\mathcal{I}_1+\mathcal{I}_2+\mathcal{I}_3.
    \end{aligned}
\end{equation*}

\textbf{- Term $\mathcal{I}_1$.}
 Recall from \eqref{E2.65} that 
 \begin{align*}
     \mathcal{I}_1=2\int_\gt \rho\dt^3\rho (\nat v\cdot\nt)\nat\dt\rho\id S+\int_\gt\rho\dt^3\rho\mathfrak{R}_1(\nat \nnt\rho+1)\id S.
 \end{align*}
 Following  arguments analogous to those used in the bottom integral, the following expansions hold, \begin{equation}\label{E5.65}
     \begin{aligned}
        2\int_\gt \rho\dt^3\rho (\nat v\cdot\nt)\nat\dt\rho\id S  =&2\ddt \int_\gt \rho\dt^2\rho(\nat v\cdot\nt)\nat\dt\rho\id S\\&-2\int_\gt \dt\left(\rho\nat v\cdot\nt\right)\dt^2\rho\nat\dt\rho\id S-\int_\gt \rho(\nat v\cdot\nt)\nat |\dt^2\rho|^2\id S,
     \end{aligned}
 \end{equation} and \begin{equation}\label{E5.66}
     \begin{aligned}
        \int_\gt\rho\dt^3\rho\mathfrak{R}_1(\nat \nnt\rho+1)\id S=&\ddt\int_\gt\rho\dt^2\rho\mathfrak{R}_1(\nat \nnt\rho+1)\id S\\&-\int_\gt \dt^2\rho\left(\mathfrak{R}_1\nat\dt\nnt\rho+\dt\mathfrak{R}_1(1+\nat\nnt\rho)\right)\id S.
     \end{aligned}
 \end{equation}
Integrating  \eqref{E5.65} and \eqref{E5.66} over time, we arrive at \begin{equation*}\label{E4.28}
     \begin{aligned}
        \int_0^T\mathcal{I}_1\id t\leq &\eo+\sup_{t\in[0,T]}\pe\left(E_h^{\delta_0'}(t)+1\right)+\int_0^T\pall\id t\\&+\int_0^T\left(-(\nat v\cdot \nt)\rho|\dt^2\rho|^2\right)\big|_{X_r}^{X_l}\id t.
     \end{aligned}
 \end{equation*}

\textbf{- Term $\mathcal{I}_3$.}
 Next we  turn to the  simpler one:  \begin{equation}\label{E10.2}
    \begin{aligned}
       \mathcal{I}_3=&  -3\int_\gt \dt^2\nnt\rho\dt\rho\dt^2\rho\id S\\=&-3\ddt\int_\gt\dt\nnt\rho\dt\rho\dt^2\rho \id S+3\int_\gt \dt\nnt\rho\dt\rho\dt^2\rho(\nat v\cdot\at)\id S\\&+3\int_\gt \dt\nnt\rho(\dt^2\rho)^2\id S+3\int_\gt \dt\nnt\rho\dt\rho\dt^3\rho\id S.
    \end{aligned}
 \end{equation}
 The last integral in the RHS of  \eqref{E10.2} follows directly from   Lemma \ref{L4.7},  \[\begin{aligned}
     3\int_\gt\dt\nnt\rho\dt\rho\dt^3\rho\id S
     \leq\pall\left(1+F_h^\frac{1}{2}(t)\right),
 \end{aligned} 
     \]
 which leads to 
\begin{equation*}\label{E4.19}
     \begin{aligned}
         \int_0^T \mathcal{I}_3\id t\leq  \eo+\sup_{t\in[0,T]}\pe\left(E_h^{\delta'_0+\epsilon}(t)+1\right)+\int_0^T \pall\left(1+F_h^\frac{1}{2}(t)\right)\id t,
     \end{aligned}
 \end{equation*}for coefficient $\delta_0'+\epsilon\in (0,1)$. Note that $\mathcal{I}_3$ has a similar structure to $\mathcal{J}_{1,1}+\mathcal{J}_{1,2}$. However the fact that  $\dt\rho\in H^1(\gt)$ whereas $(\nat v\cdot\at)\notin H^1(\gt)$ leads to a different treatment.
 
 \textbf{- Term $\mathcal{I}_2$.}
 It remains to address the final and most challenging term we encounter. 
It follows from Lemma \ref{L5.2}   that \begin{equation}\label{E10.18}
    \begin{aligned}
        \mathcal{I}_2=&\sigma\int_\gt \dt^2\nnt\rho \dt^3\kappa\id S\\=& \sigma\int_\gt \dt^2\nnt\rho\nat\left(\nat\dt\nnt\rho-2(\nat v\cdot\nt)\Delta_\gt\rho\right)\id S+\sigma\int_\gt \dt^2\nnt\rho\mathfrak{R}_3\id S.
    \end{aligned}
\end{equation}
For the first term in $\mathcal{I}_2$, integrating $\nat$ by parts, invoking Lemma \ref{L5.3}, it contributes to the energy term and dissipation term
\begin{equation}\label{E10.19}
    \begin{aligned}
          &\sigma\int_\gt \dt^2\nnt\rho\nat\left(\nat\dt\nnt\rho-2(\nat v\cdot\nt)\Delta_\gt\rho\right)\id S\\=&\sigma\dt^2\nnt\rho\left(\nat\dt\nnt\rho-2(\nat v\cdot\nt)\Delta_\gt\rho\right)\big|_{X_r}^{X_l}\\& -\sigma\int_\gt \nat\dt^2\nnt\rho(\nat\dt\nnt\rho-2(\nat v\cdot\nt)\Delta_\gt\rho)\id S\\\leq& -\frac{\sigma^2}{2\beta_c}F_h(t) -\frac{\sigma}{2}\ddt\int_\gt|\nat\dt\nnt\rho|^2\id S+\underbrace{2\sigma\int_\gt \nat\dt^2\nnt\rho(\nat v\cdot\nt)\Delta_\gt\rho\id S}_{=:\mathcal{I}_{2,1}}+\pall.
    \end{aligned}
\end{equation}
Here, 
the main challenge lies in transforming the higher-order term $\nat\dt^2\nnt\rho$ into  lower-order ones. In fact, one finds that  
\begin{equation}\label{E5.38}
    \begin{aligned}
        \mathcal{I}_{2,1}=&2\sigma\ddt\int_\gt\nat\dt\nnt\rho(\nat v\cdot\nt)\Delta_\gt\rho\id S
        -2\sigma\int_\gt \nat\dt\nnt\rho \mathfrak{R}_1(1+\nat\nnt\rho)\Delta_\gt\rho\id S
         \\&-2\sigma\int_\gt \nat\dt\nnt\rho(\nat v\cdot\nt)\nat\dt\nat\rho\id S.
    \end{aligned}
\end{equation}
We now focus exclusively on the last term in \eqref{E5.38} while the other terms are standard, their details are thus omitted. Recall  from Lemma \ref{L4.10} that $\llx{\dt\nat\rho}$ is a bounded function with respect to time, by using integrating by parts and Lemma \ref{L4.3}, we obtain
 \begin{equation}\label{E5.39}
    \begin{aligned}
       & -2\sigma\int_\gt \nat\dt\nnt\rho(\nat v\cdot\nt)\nat\dt\nat\rho\id S\\\leq &2\sigma\int_\gt \Delta_\gt\dt\nnt\rho(\nat v\cdot\nt)\xc{\dt\nat\rho}\id S+\pall.
    \end{aligned}
\end{equation}
Thanks to Lemma \ref{L5.2} and the boundary condition $\sigma\dt^3\kappa=\rho\dt^3\rho+3\dt\rho\dt^2\rho$ on $\gt$, we have \[\sigma\Delta_\gt\dt\nnt\rho=\rho\dt^3\rho+3\dt\rho\dt^2\rho+2\sigma\nat((\nat v\cdot\nt)\Delta_\gt\rho)-\sigma\mathfrak{R}_3.\]
Thus, the first term in the RHS of \eqref{E5.39}  is decomposed into five parts: \begin{equation}
    \begin{aligned}
        &2\sigma\int_\gt \Delta_\gt\dt\nnt\rho(\nat v\cdot\nt)\xc{\dt\nat\rho}\id S\\= &
        2\int_\gt \rho\dt^3\rho(\nat v\cdot\nt)\dt\nat\rho\id S+4\sigma\int_\gt \nat\left((\nat v\cdot\nt)^2\Delta_\gt\rho\right)\xc{\dt\nat\rho}\id S\\& -2\dt\nat\rho\big|_{X_i}\int_\gt \rho \dt^3\rho(\nat v\cdot\nt)\id S\\& +\int_\gt\left(6\dt\rho\dt^2\rho-4\sigma\Delta_\gt\rho\nat(\nat v\cdot\nt)\right)(\nat v\cdot\nt)\xc{\dt\nat\rho}\id S\\& -2\sigma\int_\gt\mathfrak{R}_3(\nat v\cdot\nt)\xc{\dt\nat\rho}\id S\\=&:J_1+J_2+J_3+J_4+J_5,
    \end{aligned}
\end{equation}which will be handled differently.

The first term $J_1$ has the following expansion 
\[J_1=2\int_\gt \rho\dt^3\rho(\nat v\cdot\nt)\nat\dt\rho\id S+\int_\gt\rho\dt^3\rho\mathfrak{R}_1\id S,\]
which
has an analogous structure as $\mathcal{I}_{1}$,  a similar argument leads to
\begin{equation}
    \begin{aligned}
        \int_0^T J_1\id t\leq &\eo+\sup_{t\in[0,T]}\pe\left(E_h^{\delta_0'}(t)+1\right)+\int_0^T\pall\id t\\&+\int_0^T\left(-(\nat v\cdot \nt)\rho|\dt^2\rho|^2\right)\big|_{X_r}^{X_l}\id t.
    \end{aligned}
\end{equation} 

For the second term $J_2$,   Lemmas \ref{L2.6} and \ref{L2.7} allow us to control 
\begin{equation}
    \begin{aligned}
        J_2\leq C \thb{\nat\left((\nat v\cdot\nt)^2\Delta_\gt\rho\right)}\tha{\xc{\dt\nat\rho}}\leq\pall.
    \end{aligned}
\end{equation}

The third term $J_3$  follows from using Lemmas \ref{L4.10} and  \ref{L4.7}, \begin{equation}
    \begin{aligned}
        J_3\leq \pall\left(1+F^\frac{1}{2}_h(t)\right).
    \end{aligned}
\end{equation}

For  $J_4$, owing to Lemma \ref{L4.3}, we obtain
\begin{equation}\label{E5.55}
    J_4\leq \pall.
\end{equation}

 The control of $J_5$ is almost the same as that of Proposition \ref{Prop5.1}; hence, we only point out the necessary modifications.
For instance, \begin{align*}
   & \int_\gt \nat v\cdot\nt\nat v\cdot\at\Delta_\gt\nnt\rho\xc{\dt\nat\rho}\id S\\\leq &\tha{\nat v\cdot\nt\nat v\cdot\at\xc{\dt\nat\rho}}\thb{\Delta_\gt\nnt\rho}\\\leq& \pall.
\end{align*}
The control of other integrals in $J_5$ follows from a straightforward adaptation of the procedure described in  Proposition \ref{Prop5.1}, 
 thus we omit the details. We conclude that \begin{equation}\label{E5.79}
    J_5\leq\pall.
\end{equation}

Combining \eqref{E5.38}-\eqref{E5.79}, we conclude that
\begin{equation}\label{E5.80}
    \begin{aligned}
        \int_0^T\mathcal{I}_{2,1}\id t\leq& \eo+\sup_{t\in[0,T]}\pe\left(E_h^{\frac{3}{2}\delta_0'}(t)+1\right)\\&+\int_0^T\pall\left(1+F^\frac{1}{2}_h(t)\right)\id t+\int_0^T(-(\nat v\cdot\nt)\rho|\dt^2\rho|^2)\big|_{X_r}^{X_l}\id t.
    \end{aligned}
\end{equation}

As a result, combining Proposition \ref{Prop5.1}, \eqref{E5.80}  together with \eqref{E10.18}-\eqref{E10.19}, we obtain 
\begin{equation*}\label{E10.21}
    \begin{aligned}
   & \frac{\sigma}{2}\sup_{t\in [0,T]}\int_\gt |\nat\dt\nnt\rho|^2 +\frac{\sigma^2}{2\beta_c}\int_0^T F_h(t)\id t+    \int_0^T \mathcal{I}_2\id t\\\leq &\eo+\sup_{t\in[0,T]}\pe\left(E_h^{\frac{3}{2}\delta_0'}(t)+1\right)\\&+\int_0^T\pall\left(1+F^\frac{1}{2}_h(t)\right)\id t+\int_0^T(-(\nat v\cdot\nt)\rho|\dt^2\rho|^2)\big|_{X_r}^{X_l}\id t,
    \end{aligned}
\end{equation*}for coefficient $\frac{3}{2}\delta_0'\in (0,1)$.

\textbf{- Synthesize.} Consequently, combining $\mathcal{I}_1-\mathcal{I}_3$ and bottom integral estimate with \eqref{E4.18},
 we derive for $\delta_0=\frac{3}{2}\delta_0'\in (0,1)$,
\begin{equation}\label{E4.27}
    \begin{aligned}
        & \frac{1}{2}\sup_{t\in[0,T]} \left(\int_\ot|\dt^3\rho|^2\id X+\int_\ot \rho|\nabla\dt^2\rho|^2\id X+\sigma\int_\gt |\nat\dt\nnt\rho|^2\id S\right)\\& +\frac{\sigma^2}{4\beta_c}\int_0^T F_h(t)\id t+2\int_0^T\left.\left((\nat v\cdot\nt)\rho|\dt^2\rho|^2\right)\right|_{X_r}^{X_l}\id t \\\leq &\eo+\sup_{t\in[0,T]}\pe\left(E_h^{\delta_0}(t)+1\right)+\int_0^T \pall\id t.
    \end{aligned}
\end{equation}

   Moreover, applying  $\nabla^2$ to $\eqref{E3.16}$ and taking the  $L^2(\ot)$-inner product   with $\nabla^2\omega$, yields
   \begin{equation}\label{E4.281}
       \frac{1}{2}\ddt\int_\ot |\nabla^2\omega|^2\id X\leq \pall\left(1+\mode{v}_{H^{2+\epsilon}(\ot)}\right)\leq \pall.
   \end{equation}
    Then, combining \eqref{E4.27} and \eqref{E4.281},    it follows from  the  Young's inequality and   $\delta_0\in (0,1)$ that
    \begin{equation}\label{E5.681}
        \begin{aligned}
        &\sup_{t\in [0,T]} E_h(t)+\int_0^TF_h(t)\id t+\int_0^T\left.\left(\frac{1}{\sin\theta_i}\left(-\nabla_\ab\rho-\g e_2\cdot\ab\right)\rho|\dt^2\rho|^2\right)\right|_{X_r}^{X_l}\id t
        \\ \leq&\eo+C_{\delta_0}\sup_{t\in [0,T]}\pe +\int_0^T\pall\id t+\delta_0\sup_{t\in [0,T]}E_h(t),
    \end{aligned}  
    \end{equation}
    here we have employed the fact that 
\begin{align}\label{E5.69}
    \nat v\cdot\nt=-\frac{\beta_c}{\sigma\sin\theta_i}(\nabla_\ab\rho+\g e_2\cdot\ab)\quad \text{ at }X_i(i=l,r),
\end{align}which is derived from \eqref{E3.8}.
    The last term in the RHS of \eqref{E5.681} can be absorbed into the left-hand side. 
   This completes the proof.
\end{proof}

    Here we are in the position to prove the main result.
    \begin{proof}[Proof of Theorem \ref{Theorem1}]
        It follows from
       Theorem \ref{T1}, Lemma \ref{L4.3} and Proposition \ref{L4.11} that    for $\theta_i\in (0,\pi/2)$ and $T<1$(to be fixed later) small enough, 
        \begin{align*}
            \sup_{t\in[0,T]} \mathcal{E}(t)+\int_0^T F(t)\id t&\leq P(\mathcal{E}(0))+\sup_{t\in[0,T]}\pe+\int_0^TP(\mathcal{E}(t))\id t
            \\ &\leq P(\mathcal{E}(0))+\int_0^TP(\mathcal{E}(t))\id t.
        \end{align*}

        In the end, it remains to verify the assumptions \eqref{E4.22}-\eqref{E4.23} and \eqref{E4.24} through the continuity arguments.    
        In fact, the uniform positivity of the initial density implies its persistence over the interval $[0,T]$ since \[|\rho(t,x)-\rho_0(x)|\leq T \sup_{t\in [0,T]}|\dt\rho|\leq T\sup_{t\in [0,T]}P(\mathcal{E}(t)).\]

        Secondly, 
        since the initial contact angle satisfies $\theta_i(0)\in (0,\pi/2)$, we have $\sin\theta_i(0)=\mp\nt(0)\cdot \ab(0)|_{X_i(0)}>0, (i=l,r),$ where $\ab(0)$ is a constant vector since we started in the beginning that $\gb$ becomes a straight line segment near the contact  point. By continuity, we can choose $T$ small enough such that for any $t\in [0,T]$, $\ab(t)|_{X_i}=\ab(0)|_{X_i(0)}$, this guarantees 
        \begin{align*}
            \left|\sin\theta_i(t)-\sin\theta_i(0)\right|&=\big|\nt(t)|_{X_i}-\nt(0)|_{X_i(0)}\big||\ab(0)|\\&\leq T\sup_{t\in [0,T]}\left|\dt\nt|_{X_i}\right||\ab(0)|\\&\leq T\sup_{t\in [0,T]}P(\mathcal{E}(t)).
        \end{align*}
        Consequently, if $T$ is small enough, then
         \[\sin\theta_i(t)\geq c_1\quad\text{for } t\in [0,T]\]with some constant $c_1>0$. Meanwhile, we can also have $\theta_i(t)\in (0,\pi/2)$ for $t\in [0,T]$. 

        For the last assumption \eqref{E4.24}, since the fixed bottom is inclined at the contact points, one has the fact that for constant $\g>0$,
        \begin{equation}\label{E4.61}
            -\g e_2\cdot \ab(0)|_{X_l(0)}>0\text{ and } \g e_2\cdot \ab(0)|_{X_r(0)}>0.
        \end{equation}
        Taking the initial density to be near a constant in a neighborhood of the contact points. Then, for  $t\in [0,T]$,  combining \eqref{E4.61} and the following argument
        \begin{align*}
            \left|\left.(-\nabla_\ab\rho-\g e_2\cdot\ab)\right|_{X_i}-\left.(-\nabla_\ab\rho-\g e_2\cdot\ab)(0)\right|_{X_i(0)}\right|\leq& \frac{\sigma}{\beta_c}T\sup_{t\in [0,T]}\llx{\dt(\sin\theta_i\nat v\cdot\nt)}\\\leq& T\sup_{t\in [0,T]}P(\mathcal{E}(t)),
        \end{align*}we  can conclude that \eqref{E4.24} is valid up to a time $T$. 
        These verify the necessary a  prior assumptions and establish our main theorem.
    \end{proof}

    \section{Well-posedness of the nonlinear system}\label{Section6}

    In this section, we give the proof of local well-posedness. Unlike the prior works on the free boundary problems of compressible Euler flows, which mainly focus on   parabolic regularization methods in Lagrangian coordinates, here we propose a  different approach that is fully implemented  within the setting of Eulerian coordinates.


\subsection{Reformulation of the problem}\label{S6.1}
    In this part, we need to choose a good formulation to construct approximate solutions. The basic idea we refer to \cite{MingWang CPAM2021,Shatah2011}, with some necessary modifications are needed and presented here. We start by introducing a universal coordinate map $\Phi_\st$ on the boundary $S_t=\gt\cup\gb$ that transforms the moving domain $\ot$ into a fixed domain $\Omega_*$. Choosing a compact,  reference domain $\Omega_*$ with the upper surface $\Gamma_{t*}$ and the  bottom $\Gamma_{b*}$. Without loss of generality,   $\Omega_*$ can be taken as the initial domain $\Omega_0$. 
      The corresponding contact points are denoted by $X_{i*}$ for $i=l,r$ with contact angles $\theta_{i*}\in (0,\pi/2)$. The variable $\dgt$ stands for the  “distance” between $\Gamma_{t*}$ and $\gt$, the boundary conditions for $\dgt$ (i.e. the evolution  for the contact points) take the form of an ODE.
    
    \subsubsection{Surface coordinates}
    

      First, we establish a coordinate system in a tubular neighborhood of $\Gamma_{t*}$. More precisely,  define a unit upward vector field $\mu\in H^s(\Gamma_{t*},\mathcal{S}^1)$ for some large $s$ satisfying
      \begin{equation}
          \mu\cdot\nt_*\geq c_0\text{ on }\gt_*\text{ and }\mu|_{X_{l*}}=-\tau_{b*}|_{X_{l*}},\ \mu|_{X_{r*}}=\tau_{b*}|_{X_{r*}}\label{E6.11}
      \end{equation}
      for some fixed constant $c_0\in (0,1)$. Such  vector field can be obtained by slightly regularizing the unit normal vector of  surface $\gt_*$.
    From the implicit function theorem, there exists a small constant $d_0>0$  such that
    \begin{equation*}
        \Phi: \gt_*\times [-d_0,d_0]\rightarrow \mathbb{R}^2\quad \text{ where }\Phi(X,d)\triangleq X+d\mu(X)
    \end{equation*}is an $H^s$ diffeomorphism from its domain to a neighborhood of $\gt_*$. This coordinate system associates each hypersurface $\gt$ close to $\gt_*$ with a unique function \[d_\gt: \gt_*\rightarrow \mathbb{R}.\] Plugging $\dgt$ into $\Phi$, 
     we can define the following universal coordinate map \[\Phi_\gt: \gt_*\rightarrow \gt\subset \mathbb{R}^2\quad \text{ with }\Phi_\gt (X)=X+d_\gt (X)\mu(X),\] which means that the free  interface can be constructed from a diffeomorphism defined on the reference domain in a tubular neighborhood  coordinate system. Sometimes, the  function $d_\gt(X)$ is also used as the expression of the upper surface $\gt$. 
    Moreover, let $\Phi_\st : S_*=\gt_*\cup \gb_* \rightarrow\st=\gt\cup\gb$, then $\Phi_\gt$ can be extended to the entire boundary $S_*$.
     
     We construct harmonic coordinates  on $\Omega_*$ in the following manner: \[\mathcal{T}_\st :\Omega_*\rightarrow \Omega_t\ \text{ with }\ \mathcal{T}_\st =\mathcal{H}_*\left(\Phi_\st -\mathrm{Id}_{S_*}\right)+\mathrm{Id}, \]where  $\mathcal{H}_*(\Phi_\st -\mathrm{Id}_{S_*})$ is the harmonic extension of $\Phi_\st -\mathrm{Id}_{S_*}$ satisfying \begin{equation*}
        \begin{cases}
            \Delta\mathcal{H}_*\left(\Phi_\st -\mathrm{Id}_{S_*}\right)=0&\text{ in }\Omega_*,\\ \left.\mathcal{H}_* \left(\Phi_\st -\mathrm{Id}_{S_*}\right)\right|_{\gt_*}=d_\gt\mu,& \left.\mathcal{H}_*\left(\Phi_\st -\mathrm{Id}_{S_*}\right)\right|_{\gb_*}=\Phi_\st|_{\gb_*}-\mathrm{Id}_{\gb_*}.
        \end{cases}
    \end{equation*}
    The corresponding compatibility conditions for the Dirichlet boundary conditions are \[\dgt\mu(X_{i*})=\left(\Phi_\st|_{\Gamma_{b*}}-\mathrm{Id}_{\gb_*}\right)(X_{i*}),\ i=l,r.\]
    So long as $\gtx$ does not self-intersect, for each curve $\gt$ in the neighborhood of $\Gamma_{t*}$: \[\lam\triangleq \left\{\gt \mid s\in(1.5,4.5],\  \mode{\dgt}_{H^s(\Gamma_{t*})}<\delta,\ \mode{\Phi_\st-\mathrm{Id}_{S_*}}_{H^s(\Gamma_{b*})}\leq\delta,\ \theta_i\in (0,\pi/2), i=l,r\right\},\] provided that $\delta \ll 1$, $\Phi_\st$ is a diffeomorphism both in $H^s(\gtx,\gt)$ and $H^s(\Gamma_{b*},\gb)$, and $\mathcal{T}_\st$ is a diffeomorphism from $\Omega_*$ to $\mathcal{T}_\st(\Omega_*)$ since \[\mode{\nabla\mathcal{T}_\st-\mathrm{I}}_{H^{s-\frac{1}{2}}(\Omega_*)}\leq C\left(\mode{\dgt}_{H^s(\Gamma_{t*})}+\mode{\Phi_\st-\mathrm{Id}_{S_*}}_{H^s(\Gamma_{b*})}\right).\]From the definition of $\Phi_\st$, the norm $\mode{\Phi_\st-\mathrm{Id}_{S_*}}_{H^s(\Gamma_{b*})}$ on the bottom can be controlled by $\dgt$ as well.
    \begin{remark}
        One can see from \cite{Shatah2011} that $s>\frac{n+1}{2}$ (here  $n=2$ in our case) ensures that the mapping $\gt\in \lam$, or equivalently $\dgt$, to its mean curvature $\kappa\circ \Phi_\st\in H^{s-2}(\Gamma_{t*})$ is smooth, and $s\leq 4.5$ follows from \eqref{E7.28}.
    \end{remark}

    \subsubsection{Pull-back Formulas}
When the domain $\ot$ is defined by $\mathcal{T}_\st$, 
the evolution of the free surface function $d_\gt$ can be defined by the normal component of $v$. In fact, the kinematic condition on $\gt$ can be rewritten as
\begin{equation*}\label{E8.2} \pa_t\Phi_\st\cdot (\nt \circ\Phi_\st)=(v\cdot\nt)\circ \Phi_\st,\text{ where }\pa_t\Phi_\st=(\pa_t d_\gt)\mu \text{ on }\gt_*,\end{equation*}which leads to 
\begin{equation}\label{E8.7}
    v\cdot\nt=(\pa_t d_\gt\mu)\circ \Phi_\st^{-1}\cdot\nt. 
\end{equation}

In order to pull back the system \eqref{E2} into $\Omega_*$, one has to handle the pull-back  of derivatives. For  notational simplicity, we denote by 
\begin{equation*}\label{E6.31}
    \gth{f}=f\circ \Phi_\st,\ \text{ and }\  \oth{g}=g\circ\mathcal{T}_\st,
\end{equation*}for any function $f$ defined on $\gt$, and $g$ defined on $\ot$, which hints the fact that \[(\pa_i f)\circ\Phi_\st =\left(\pa_i\left(\Phi_\st^{-1}\right)_j\circ \Phi_\st\right)\pa_j \gth{f},\text{ and }(\pa_i g)\circ\mathcal{T}_\st=\left(\pa_i\left(\mathcal{T}_\st^{-1}\right)_j\circ \mathcal{T}_\st\right)\pa_j\oth{g}.\]


Moreover, we construct the Lagrangian map from $v$. In fact, denoting by $\Omega_0$ and $\Gamma_0$ the initial domain and upper surface respectively, 
the velocity field $v$ induces a flow map $U(t,\cdot):S_0\rightarrow S_t$ by solving the ODE\[U(0,\cdot)=\mathrm{Id}_{S_0},\quad \pa_t U(t,\cdot)=v(t,U(t,\cdot)).\]Together with the diffeomorphism $\Phi_\st$, this induces a conjugate flow map $U_*: S_*\rightarrow S_*$ by \[U_*(t,\cdot)\triangleq \Phi_\st ^{-1}\circ U(t,\cdot)\circ \Phi_{s_0},\] and the corresponding $v_*$ is given by \[\pa_t U_*(t,\cdot)=v_* (t,U_*(t,\cdot)).\]Therefore, the velocity field induced by this family of transformations $U_*(t,\cdot)$ can be represented as 
\begin{equation*}\label{E8.1}
    v_*=D\Phi^{-1}_\st\left(v\circ \Phi_\st-\pa_t d_\gt \mu \right)=D\Phi^{-1}_\st \left(v^\top \circ \Phi_\st-\pa_t d_\gt\mu^\top \right)\text{ on }\gt_*.
\end{equation*}
The material derivative $D_{t*}$ associated to $v_*$ is denoted by  \[D_{t*}=\pa_t+\nabla_{v_*}.\] In this case, for any function $f$ defined on $\gt$, 
\begin{equation}\label{E8.5}
    (\dt f)\circ \Phi_\st=D_{t*}(f\circ \Phi_\st).
\end{equation}

Similar analysis is employed, denote
\begin{equation*}\label{E10.22}
v_\sharp   =D\mathcal{T}^{-1}_\st (v\circ\mathcal{T}_\st-\mathcal{H}_*(\pa_t\dgt\mu)),
\end{equation*} and the corresponding material derivative  is denoted by $\mathbb{D}_t=\pa_t+\nabla_{ v_\sharp}$, which implies that \begin{equation*}\label{E8.14}
    \dt g\circ \mathcal{T}_\st=\mathbb{D}_t(g\circ \mathcal{T}_\st),
\end{equation*}for any function $g$ defined in $\ot$.

\subsubsection{Constructing the free surface from its mean curvature}\label{E6.1.3}
 The interface cannot  be constructed directly from its mean curvature due to the degeneracy of the associated elliptic problem. 
 Indeed, the linearized operator (e.g., $L(S_*)d=(-\Delta_*-|\Pi_*|^2)d$) may not be invertible, see Section 2.3 of \cite{Shatah2011} for more details.  However, in our case, 
  the contact points make the problem invertible since they remove rigid-body translations and rotations, while the imposed properties  \eqref{E6.11} ensure the non-zero normal displacement, thereby 
  removing tangential reparametrization.  

Next, define\[d_{\gt}(X_{i*})=d_i,\ i=l,r.\]The evolution equations for $d_i=d_i(t) (i=l,r)$ are derived as follows.
In fact, for $X_i(i=l,r)$, utilizing \eqref{E8.7} and \eqref{E8.5}, we obtain immediately that
\begin{equation}\label{E6.7}
    \begin{aligned}
         \dt v\cdot\nt=&\left(\pa_t\dgt \mu\circ \Phi^{-1}_\st-v\right)\cdot\dt\nt+\left(\pa_t^2 \dgt+\nabla_{v_*}\pa_t\dgt\right)\circ\Phi_\st^{-1}\left(\mu \circ\Phi^{-1}_\st \cdot \nt\right)\\&+\pa_t \dgt \circ\Phi^{-1}_\st \left(\nabla_{v_*}\mu\circ\Phi_\st^{-1}\cdot\nt\right),
    \end{aligned}
\end{equation}
which implies that for $i=l,r$,
\begin{equation}\label{E6.16}
    \begin{aligned}
          \pa_t^2 d_i(t)=\mathfrak{B}_i\triangleq &-\frac{1}{\mu\cdot(\nt\circ\Phi_\st)}\left(\nabla_{v_*}\pa_t\dgt (\mu\cdot(\nt\circ\Phi_\st))+\pa_t\dgt \nabla_{v_*}\mu \cdot(\nt\circ \Phi_\st)\right)
   \\&\left.-\frac{1}{\mu\cdot(\nt\circ\Phi_\st)}\left((\pa_t\dgt\mu-v\circ\Phi_\st)(\dt\nt\circ\Phi_\st)+ \left(\nnt\rho+\g e_2\cdot\nt\right)\circ\Phi_\st\right)\right|_{X_{i*}}.
    \end{aligned}
\end{equation}

Define the set $B_{\delta_1}\triangleq \{(\kappa\circ \Phi_\st,d_l,d_r)\mid \mode{\kappa\circ \Phi_\st-\kappa_{*}}_{H^s(\Gamma_{t*})},|d_l|,|d_r|<\delta_1\}$, where $\kappa_{*}$ is the value of $\kappa$ taken at $\Gamma_{t*}$. Since the linearized curvature equation reduces
to a standard elliptic system  with boundary conditions, 
the operator 
\begin{align*}
   \mathcal{K}:
   \Lambda(\Gamma_{t*},s+2,\delta,\pi/2)\subset H^{s+2}(\Gamma_{t*}) &\rightarrow \mathcal{K}\left(\Lambda(\Gamma_{t*},s+2,\delta,\pi/2)\right)\subset B_{\delta_1}\\\dgt&\mapsto  (\kappa\circ \Phi_\st,d_l,d_r)
\end{align*}
is invertible, and by the inverse function theorem it is easy to conclude that the map $\mathcal{K}^{-1}$ is a local diffeomorphism and has a regularity gain of two derivatives.  
See Proposition 4.2 in \cite{MingWang CPAM2021} for more details.
Subsequently,  the  free surface $\gt$ (or equivalently $\dgt$) can be recovered from  $(\kappa\circ \Phi_\st-\kappa_*)$ and $d_i(i=l,r)$.

\subsection{The linearized system}\label{S6.2}

For the purpose of constructing solutions to the nonlinear system, assuming that the free surface $\gt\in \lam$ is known already, we consider the following linear system for $f(t,X)$ with $X\in \Gamma_{t*}$,
\begin{equation*}\label{E6.101}
    \begin{cases}
        \frac{\gth{\rho}}{\sigma} D_{t*}^2 \gth{f}+ \left((a-\Delta_\gt)\nnt(\gth{f}\circ\Phi_\st^{-1})\right)\circ\Phi_\st =\gth{F_2}&\text{ on }\Gamma_{t*},\\
       \pm D_{t*} (\nnt (\gth{f}\circ \Phi_\st ^{-1}))\circ \Phi_\st
          +\frac{\sigma}{\beta_c}\left((\sin\theta_i)^2\nat\nnt (\gth{f}\circ\Phi^{-1}_\st)\right)\circ\Phi_\st
         =\gth{F_{4,i}}&\text{ at }X_{i*}(i=l,r),
    \end{cases}
\end{equation*}with the given initial data, and  $ a>0$. We utilize  Galerkin's method to construct approximate solutions of  $\gth{f}$.  Since $H^1(\Gamma_{t*})$ is separable, it possesses a countable  basis $\{\mathfrak{e}_j\}_{j=1}^\infty$, which is  an orthogonal basis of $L^2(\Gamma_{t*})$. Denote \[(\phi,\psi)_{\Gamma_{t*}}=\int_{\Gamma_{t*}}\phi\psi|\det D\Phi_\st|\id S.\]
For a positive integer $m$, taking the approximation under the form 
\[\gth{f}^m(t,X):=\sum_{j=1}^m a_j^m(t)\mathfrak{e}_j(X),\quad \text{with }a_j^m:[0,T]\rightarrow \mathbb{R}.\]
For each $\ell=1,\dots, m$, the approximation leads to  a finite-dimensional linear system of ordinary differential equations:
\begin{equation*}\label{E6.111}
    \begin{aligned}
    & \frac{1}{\sigma}(\gth{\rho}\mathfrak{e}_j,\mathfrak{e}_\ell)_{\Gamma_{t*}} \frac{\id^2 }{\id t^2} a_j^m(t)+\left(\frac{2}{\sigma}(\gth{\rho}\nabla_{v_*}\mathfrak{e}_j,\mathfrak{e}_\ell)_{\Gamma_{t*}}+\sum_{i=l,r}\left.\frac{\beta_c}{\sigma (\sin\theta_i)^2\circ\Phi_\st}\left(\nnt (\mathfrak{e}_j\circ\Phi_\st^{-1})\right)\circ\Phi_\st\mathfrak{e}_\ell\right|_{X_{i*}}\right)\ddt a_j^m(t) \\&+\left(\frac{1}{\sigma}\left(\gth{\rho}(\nabla_{v_*}^2+\nabla_{\pa_t v_*})\mathfrak{e}_j,\mathfrak{e}_\ell\right)_{\Gamma_{t*}}+\int_\gt \nnt \left(\mathfrak{e}_j\circ\Phi_\st^{-1}\right)(a-\Delta_\gt)(\mathfrak{e}_\ell\circ\Phi_\st^{-1})\id S
       \right)a_j^m(t)
       \\& +\left(\sum_{i=l,r}\left.\frac{\beta_c}{\sigma (\sin\theta_i)^2\circ\Phi_\st}D_{t*}\left(\nnt (\mathfrak{e}_j\circ\Phi_\st^{-1})\right)\circ\Phi_\st\mathfrak{e}_\ell\right|_{X_{i*}}+ \left.\left(\nnt \left(\mathfrak{e}_j\circ\Phi_\st\right)\nat (\mathfrak{e}_\ell\circ\Phi_\st)\right)\circ\Phi_\st^{-1}\right|_{X_{r*}}^{X_{l*}}\right)a_j^m(t)\\&=(\gth{F_2},\mathfrak{e}_\ell)_{\Gamma_{t*}}+\left.\frac{\beta_c}{\sigma (\sin\theta_i)^2\circ\Phi_\st}\gth{F_{4,i}}\mathfrak{e}_\ell\right|_{X_{r*}}^{X_{l*}},
    \end{aligned}
\end{equation*}
here we have used \[D_{t*}^2=\pa_t^2+\nabla_{\pa_tv_*}+2\nabla_{v*}\pa_t+(\nabla_{v_*})^2.\] Owing to the positivity of $\gth{\rho}$ and the regularity of the mapping, the coefficient matrix $M_{m\times m}$ with $M_{j\ell}=(\gth{\rho}\mathfrak{e}_j,\mathfrak{e}_\ell)_{\Gamma_{t*}} $, is symmetric and uniformly positive definite on $[0,T]$. The standard existence and uniqueness theorem for linear ODEs with variable coefficients  guarantees a unique solution $a_j^m(t)$ on the whole interval $[0,T]$.
Then, 
 the existence of the following system
\begin{equation*}
    \begin{cases}
        \dt^2 f-\rho\Delta f=F_1&\text{ in }\ot,\\
        f=\gth{f}\circ \Phi_\st^{-1}&\text{ on }\gt,\\
           \nabla_\nb f=F_3&\text{ on }\gb,\\
    \end{cases}
\end{equation*}
can be established by standard arguments pertaining to the initial value problem for  hyperbolic system. 

\subsection{Uniform estimate}\label{S6.3}
We shall establish the existence of solutions via  Picard iterations.  To begin with, the initial boundary is set as  $S_0=S_*$. For each $k\in\mathbb{N}$, we denote the $k$-th iteration by $\left(\nk{\rho},\nk{v},\nk{\Omega_t}\right)$. For notational convenience, the following abbreviations are adopted: 
\[\nk{n_j}=n_j,\  \nk{\tau_j}=\tau_j, \text{ for }j=t,b,\]
 and
\[\dt=\pa_t+v^k\cdot\nabla,\ \Delta_\gt=\Delta_{\Gamma_t^k},\ \kappa=\nabla_{\nk{\tau_t}}\nk{n_t}\cdot \nk{\tau_t}.\]
The solution $\nkj{\rho}$ is constructed via the following  linear system  whose coefficients depend on $(\nk{\rho},\nk{v},\nk{\Omega_t})$, 
\begin{equation}\label{E6-14}
    \begin{cases}
        \dt^2\nkj{\rho}-\rp\Delta\nkj{\rho}=F_{1}^{k} &\text{ in }\Omega_t^k,\\ \frac{\rp}{\sigma}\dt^2 \nkj{\rho}+(a-\Delta_\gt)\nnt\nkj{\rho}=\nk{F_2}&\text{ on }\nk{\Gamma_t},\\
        \nabla_\nb \nkj{\rho}=F_{3}^{k}&\text{ on }\Gamma_b^k,\\
        \pm \dt\nnt \nkj{\rho}=-\frac{\sigma}{\beta_c}\nat\nnt\nkj{\rho}(\sin\theta_i^k)^2+\nk{F_{4,i}}&\text{ at }X_i^k(i=l,r),
    \end{cases}
\end{equation}
with 
\begin{equation*}
    \begin{aligned}
         F_{1}^{k}\triangleq &\frac{1}{\rp}\left(\dt\rp\right)^2+
\rp\mathrm{tr}\left((\nabla \nk{v})^2\right),\\\nk{F_2}\triangleq &  4\left(\nat \nk{v}\cdot\at\right)\nabla_\at\left(\nabla_\at\nk{v}\cdot\nt\right)+2\left(\nat\nk{v}\cdot\nt\right)\nat\left(\nat \rv\cdot \at\right)-\Delta_\gt n_t\cdot\nabla\rp\\&+\kappa\left(2(\nat \nk{v}\cdot\at)^2-(\nat \nk{v}\cdot\nt)^2\right) +a\nnt\rp-\frac{1}{\sigma}\left(\dt\rp\right)^2,\\ \nk{F_{3}}\triangleq &\nk{v}\cdot\nabla_{\nk{v}}\nb-\g e_2\cdot\nb ,\\
\nk{F_{4,i}} \triangleq & - \frac{2\sigma}{\beta_c}\left(\nat \rv\cdot\at\right)\left(\nat \nk{v}\cdot\nt\right)(\sin\theta^k_i)^2\\&+\frac{\sigma}{\beta_c}\left(\frac{\sigma}{\beta_c}(\nat\nk{v}\cdot\nt)\sin\nk{\theta_i}\cos\nk{\theta_i}-\g e_2\cdot\at+\nk{v}\cdot\nabla_{\nk{v}}\nb\left(n_b\cdot\at\right)\right)\kappa(\sin\theta^k_i)^2\\&-\frac{2\sigma}{\beta_c}(\nat \nk{v}\cdot\nt)^2(\nt\cdot\ab)(\at\cdot\ab)\pm \g e_2\cdot\at(\nat \nk{v}\cdot\nt)\pm\dt\left((\nb\cdot\nt)\nk{v}\cdot\nabla_{\nk{v}}\nb\right) .  
\end{aligned}
\end{equation*}
Note that the “hidden” unknown $\rv$ appears in $\nk{F_2}$ and $\nk{F_{4,i}}$ is pivotal, whereas in other quantities we use $\nk{v}$. The evolution of $\rv$ is governed by  the following version of Euler equation,
\begin{equation}\label{E}
    \dt \nkj{\mathring{v}}=-\nabla \nkj{\rho}-\g e_2,\quad \text{ in }\nk{\Omega_t}, \text{ with }\rv|_{t=0}=v_0,
\end{equation}
  which is introduced so that, upon applying the material derivative $\dt$ to \eqref{E6-14}$_2$, the principal terms match those of  the higher-order curvature equation \eqref{E7.14}, with the corresponding adjustment made to the principal terms in \eqref{E6-14}$_4$. Remark  that the notion $\nk{\rho}$ appears in \eqref{E6-14} is actually  $\nk{\rho}\circ \mathcal{T}_\st^{k-1}\circ\left(\mathcal{T}_\st^k\right)^{-1}$, we keep this notion $\nk{\rho}$ for simplicity as they do not affect the subsequent analysis.

  In parallel, a vorticity $\nkj{\omega}$   is introduced, subject to the linear equation given below, with coefficients that depend explicitly on the preceding  iterates,
\begin{equation}\label{E7.13}
    \dt \nkj{\omega}=-(\divv \nk{v})\nabla^\perp \cdot\nk{v}\quad \text{ in }\Omega_t^k.
\end{equation}

Subsequently, the moving interface $ \nkj{\Gamma_t}$, or equivalently $d_\gt^{k+1}$,  is determined via the following non-degenerate system 
\begin{equation}\label{E6.191}
\begin{cases}
      d_\gt^{k+1}= \mathcal{K}^{-1} \left(\frac{1}{\sigma}P(\nkj{\rho})\circ\Phi^k_\st ,\nkj{d_l}(t),\nkj{d}_r(t) \right),\\
       \frac{\id^2}{\id t^2}\nkj{d}_i(t)=\mathfrak{B}_i^k, \ i=l,r,\\
       d^{k+1}_i(0)=d_{i,0},\ \ddt d_i^{k+1}(0)=d_{i,1}.
\end{cases}
\end{equation}where  $\mathfrak{B}_i^k$ is given in  \eqref{E6.16}.

Based on the construction of the moving domain $\nkj{\Omega_t}$, one can recover the velocity fields $\nkj{v}$. More precisely, the velocity $\nkj{v}$ on $\nkj{\Omega_t}$ is obtained as the solution of   the following div-curl problem: 
\begin{equation}\label{E8.16}
    \begin{cases}
        \divv \nkj{v}=-\left(\frac{1}{\rp}\dt \nkj{\rho}\right)\circ \mathcal{T}_\st^{k} \circ\left(\nkj{\mathcal{T}_\st}\right)^{-1}&\text{ in }\nkj{\Omega_t},\\
        \nabla^\perp \cdot \nkj{v}=\nkj{\omega}\circ \mathcal{T}_\st^{k} \circ\left(\nkj{\mathcal{T}_\st}\right)^{-1}&\text{ in }\nkj{\Omega_t},\\
        \nkj{v}\cdot \nkj{n_t}=\left(\pa_t \nkj{d_\gt}\mu\right)\circ \left(\nkj{\Phi_\st}\right)^{-1}\cdot\nkj{n_t}&\text{ on }\nkj{\Gamma_t},\\
        \nkj{v}\cdot \nkj{n_b}=\varUpsilon &\text{ on }\nkj{\Gamma_b},
    \end{cases}
\end{equation}
where $\nkj{n_j}(j=t,b)$ are the unit outward normal vectors of $\nkj{\Gamma_j}$, and $\varUpsilon=\varUpsilon(t)$ is the balancing function to ensure the compatibility condition \[-\int_{\Omega_*}\left(\frac{1}{\nk{\rho}}\dt\nkj{\rho}\right)\circ \mathcal{T}_\st^k\left|\det D\mathcal{T}_\st^{k+1}\right|\id X=\int_{\Gamma_{t*}}\pa_t \nkj{d_\gt}\mu\cdot\left( n_{t}^{k+1}\circ\Phi_\st^{k+1}\right)\left|\det D\nkj{\Phi_\st}\right|\id S
+\varUpsilon |\nkj{\Gamma_b}|,\] for each fixed moment $t$.  We remark here that the simply connectedness of $\nk{\Omega_t}$ ensures the unique solvability of the div-curl system \eqref{E8.16}, see Section \ref{Section3} for more details. 
Now we are ready to give the uniform-in-$k$ estimates for $\left(\nkj{\rho},\nkj{v},\nkj{\Omega_t}\right)$. To  begin with, we define the energy functional for $k+1$-th iterates as
\begin{equation*}
    \begin{aligned}
        \mathcal{E}^{k+1}(t)=&a\sum_{j=0}^1 \mode{\dt^j\nnt\nkj{\rho}}^2_{L^2(\Gamma_t^k)}+\sum_{j=0}^3\mode{\dt^j\rho^{k+1}}^2_{L^2(\Omega_t^k)}+\sum_{j=0}^2\mode{\nabla\dt^j\nkj{\rho}}_{L^2(\Omega_t^k)}^2\\&+\sigma\sum_{j=0}^1 \mode{\nabla_{\tau_t}\dt^j\nabla_{n_t}\rho^{k+1}}^2_{L^2(\Gamma_t^k)}+\mode{\nkj{\omega}}_{H^2(\Omega_t^k)}^2\\&+\sum_{i=l,r}\left(\cmode{d_i^{k+1}(t)}^2+\cmode{\ddt d_i^{k+1}(t)}^2\right).
    \end{aligned}
\end{equation*}
Moreover, the dissipation $F^{k+1}(t)$ is defined as
\begin{equation*}
    \begin{aligned}
        F^{k+1}(t)=& \sum_{i=l,r}\left|(\sin\nk{\theta_i})\nat\nnt\nkj{\rho}\big|_{X_i^k}\right|^2 \\&+\sum_{i=l,r}\left|\left(\sin\theta^k_i\right)\left(\nat \dt\nnt\nkj{\rho}-2(\nat \nk{v}\cdot\nt)\Delta_\gt\nkj{\rho}\right)\Big|_{X^k_i}\right|^2\\&
        +\left.\left(\nat\nk{v}\cdot\nt\, \nk{\rho}\left|\dt^2\nkj{\rho}\right|^2\right)\right|_{X^k_r}^{X^k_l}.
    \end{aligned}
\end{equation*}

The following proposition is the main result for the uniform estimates:
\begin{proposition}\label{Prop6.1}
    Assume  that $\nk{\rho}\geq c_0$ for some $c_0>0$, and the contact angles $\theta_i^k$ satisfy \[\min_i \sin\theta^k_i\geq c_1 \text{ for some }c_1>0.\]  Moreover, suppose that at the left and right contact points,  \[ \left.\left(\nat\nk{v}\cdot\nt\right)\right|_{X_l^k}\geq c_2\ \text{ and } \ -\left.\left(\nat\nk{v}\cdot\nt\right)\right|_{X_r^k}\geq c_3\ \text{ for some }c_2,c_3>0.\]
    There exists constant $T>0$ small enough, the inequality below holds
    \begin{equation}\label{E6.252}
        \sup_{t\in [0,T]} \mathcal{E}^{k+1}(t)+\int_0^T F^{k+1}(t)\id t\leq P(\mathcal{E}(0)). 
    \end{equation}
\end{proposition}

The proof proceeds by induction. Under the induction hypothesis that the energies up to the $k$-th iteration are bounded, we assume the following given quantities are  bounded as well, on a time interval $[0,T]$, 
\[\pa_t^3\nk{d_\gt}\in H^1(\Gamma_{t*}), \ \pa_t^2\nk{d_\gt} \in H^{2.5}(\Gamma_{t*}) ,\ \pa_t\nk{d_\gt} \in H^{3.5}(\Gamma_{t*}),\ \nk{d_\gt}\in  H^{4.5}(\Gamma_{t*}) ,\]and $\nk{v}\in H^{2+\epsilon}(\nk{\Omega_t})$, $\dt \nk{v}\in H^2(\nk{\Omega_t})$, as well as $\dt^2\nk{v}\in H^1(\nk{\Omega_t})$ with uniform bounds $\po:=P(\mathcal{E}(0))$, all of which can be verified in the subsequent analysis.

Now, we consider the estimates for $\nkj{d_\gt}$.
\begin{lemma}\label{l6.1}
    Let $\mathcal{P}_0:=P(\mathcal{E}(0))$. For some $T>0$,  suppose that $\mathcal{E}^k(t)\leq \po$ for all $t\in [0,T]$. Let $\nk{\Gamma_t}\in\lam$ for $t\in[0,T]$. Then the following estimates hold:
   \begin{align*}
  \mode{\pa_t^3\nkj{d_\gt}}_{H^1(\Gamma_{t*})}  +\mode{\pa_t^2\nkj{d_\gt}}_{H^{2.5}(\Gamma_{t*})}+\mode{\pa_t\nkj{d_\gt}}_{H^{3.5}(\Gamma_{t*})}+\mode{\nkj{d_\gt}}_{H^{4.5}(\Gamma_{t*})}\leq\pallk.
   \end{align*}
\end{lemma}
\begin{proof}
    The diffeomorphism $\nk{\Phi_\st}(X)=X+\nk{d}_\gt(X)\mu(X)$ from $S_*$ to $S_t^k$, and its inverse satisfying \[\mode{\nk{\Phi_\st}}_{H^{4.5}(S_*)}+\mode{(\nk{\Phi_\st})^{-1}}_{H^{4.5}(S^k_t)}\leq \po.\]  In addition, 
    the variable\begin{equation}\label{E7.61}
        \nk{v_*}:=D\left(\nk{\Phi_\st}\right)^{-1}\left(\nk{v}\circ\Phi_\st^k-\pa_t\nk{d_\gt}\mu\right)
    \end{equation}
     is introduced precisely to ensure the equality \[(\dt f)\circ \Phi_\st^k=(\pa_t +\nk{v_*}\cdot\nabla)\left(f\circ\Phi_\st^k\right)=: D_{t*}\gth{f},\]  and the quantity 
     \begin{equation}\label{E7.62}
         v_\sharp^k:=D\left(\mathcal{T}_{\st}^{k}\right)^{-1}\left(v^k\circ\mathcal{T}_\st^{k}-\mathcal{H}_*(\pa_t\dgtk{}\mu)\right)
     \end{equation}
      is defined in an analogous manner. Then,  $\nk{v_*}$ and $\nk{v_\sharp}$ enjoy the following property, 
    \begin{equation}\label{E7.22}
        \mode{\nk{v}_*}_{H^{1.5+\epsilon}(\Gamma_{t*})}+\mode{\nk{v}_\sharp}_{H^{2+\epsilon}(\Omega_*)}\leq\po. 
    \end{equation}
      Note that the boundary condition \eqref{E8.16}$_3$ constrains the derivative $\nabla_{\nk{v}_*}$ to be tangential, namely, $\nabla_{\nk{v}_*}=(\nk{v_*}\cdot \tau_{t*})\pa_{s}$.    Applying $D_{t*}$ to  both sides of \eqref{E7.61}, yields\[D_{t*} \nk{v_*}=\left(D\nk{\Phi_\st}\right)^{-1}\left(-(\pa_t+\nabla_{\nk{v_*}})\left(D\nk{\Phi_\st}\right) \nk{v_*}+\left(\dt \nk{v}\right)\circ \nk{\Phi_\st}-(\pa_t+\nabla_{\nk{v_*}})\left(\pa_t\nk{d_\gt}\mu\right)\right),\] and similar derivations for $\mathbb{D}_t\nk{v_\sharp}$, one obtains  \begin{equation}\label{E7.241}
         \mode{D_{t*}\nk{v}_*}_{H^{1.5}(\Gamma_{t*})}+\mode{\mathbb{D}_t\nk{v}_\sharp}_{H^{2}(\Omega_*)}\leq \po.
    \end{equation}
    Exploiting the equality
    \begin{equation}\label{E7.242}
        D_{t*}^2=\pa_t^2+\nabla_{\nk{v_*}}\pa_t+D_{t*}\nk{v_*}\cdot\nabla+\nk{v_*}\cdot D_{t*}\nabla\text{ and }\mathbb{D}_t^2=\pa_t^2+\nabla_{\nk{v_\sharp}}\pa_t+\mathbb{D}_t \nk{v_\sharp}\cdot\nabla+\nk{v}_\sharp\cdot\mathbb{D}_t\nabla,
    \end{equation}
    the first one in \eqref{E7.242} yields $D_{t*}^2\pa_t\nk{d_\gt}\in H^{1}(\Gamma_{t*})$ and $D_{t*}^2\left(D\Phi_\st^k\right)\in H^{1.5}(\Gamma_{t*})$. Based on these, 
    applying $D_{t*}^2$ to \eqref{E7.61} and $\mathbb{D}_t^2$ to \eqref{E7.62},  respectively, a direct computation shows that
    \begin{equation}\label{E7.24}
        \begin{aligned}
            \mode{D_{t*}^2\nk{v}_*}_{H^{0.5}(\Gamma_{t*})} +\mode{\mathbb{D}_t^2\nk{v_\sharp}}_{H^1(\Omega_*)}\leq\po,
        \end{aligned}
    \end{equation}
The following norms can be derived by using \eqref{E7.22}-\eqref{E7.242}, 
\begin{equation}\label{E7.26}
 \mode{\pa_t^2\oth{\nkj{\rho}}}_{H^1(\Omega_*)}+\mode{\pa_t \oth{\nkj{\rho}}}_{H^2(\Omega_*)}+\mode{\oth{\nkj{\rho}}}_{H^3(\Omega_*)}\leq \pallk,
\end{equation}
here $\tkj{\rho}=\nkj{\rho}\circ \nk{\mathcal{T}_\st}.$

In addition,  following the idea of  Lemma 4.9 in \cite{MingWang CPAM2021}, 
which expresses the unit tangential vector   as  \[\at\circ\Phi_\st^k(s)=\frac{\pa_s\Phi_\st^k(s)}{\cmode{\pa_s\Phi_\st^k(s)}}\text{ with }\pa_s\Phi_\st^k(s)=X'(s)+\pa_s\nk{d_\gt}(s)\mu(s)+\nk{d_\gt}(s)\mu'(s),\]where $s$ denotes the arc length parameter and $\Phi^k_\st=\Phi_\st^k(s)$, and in conjunction with \eqref{E7.22}-\eqref{E7.24}, as well as the decomposition\begin{equation}\label{E6.26}
    D_{t*}^3=\pa_t^3+\nabla_{\nk{v}_*}\pa_t^2+D_{t*}\nabla_{\nk{v}_*}\pa_t+D_{t*}^2\nabla_{\nk{v}_*},
\end{equation} the following estimates are derived,
\begin{equation}\label{E7.33}
    \mode{\varpi}_{H^{3.5}(\nk{\Gamma_t})}+\mode{\dt\varpi}_{H^{1.5+\epsilon}(\nk{\Gamma_t})}+\mode{\dt^2\varpi}_{H^{1.5}(\nk{\Gamma_t})}+\mode{\dt^3\varpi}_{L^2(\nk{\Gamma_t})}\leq \po\quad \text{ for vector }\varpi=\tau_t^k,n_t^k.
\end{equation}

    Recall that $d_i^{k+1}$ is given by \eqref{E6.16}, then applying  \eqref{E7.22}-\eqref{E7.24} together with \eqref{E7.33}
     leads to some essential estimates,
    \begin{equation}\label{E7.23}
        \begin{aligned}
            \cmode{\pa_t^2\nkj{d_i}}+\cmode{\pa_t^3\nkj{d_i}}\leq&\po,\\
            \cmode{\nkj{d_i}}+\cmode{\pa_t\nkj{d_i}}\leq &\cmode{d_i(0)}+t\cmode{d_i'(0)}+\int_0^T \int_0^{t}|\mathfrak{B}^k_i(\tau)|\id \tau\id t+\int_0^T |\mathfrak{B}_i^k(t)|\id t\\\leq &\po +\int_0^T \po\id t\quad  \text{ for }i=l,r.
        \end{aligned}
    \end{equation}
Following the similar argument that used in the proof of Proposition 4.2 in \cite{MingWang CPAM2021}, utilizing \eqref{E6.191}, \eqref{E7.26} and \eqref{E7.23}, yields 
    \begin{equation}\label{E7.28}
        \begin{aligned}
            \mode{d^{k+1}_\gt}_{H^{4.5}(\Gamma_{t*})}\leq &C\left(\mode{\frac{1}{\sigma}P\left(\tkj{\rho}\right)-\kappa_{*}}_{H^{2.5}(\Gamma_{t*})}+\cmode{d_l^{k+1}}+\cmode{d_r^{k+1}}\right)\leq\pallk.
        \end{aligned}
    \end{equation}Then,
      applying   Proposition 4.12 in \cite{MingWang CPAM2021}, we arrive at  
    \begin{equation*}
        \begin{aligned}
             \mode{\pa_t d^{k+1}_\gt}_{H^{3.5}(\Gamma_{t*})}\leq &C\left(\mode{\frac{1}{\sigma}P\left(\tkj{\rho}\right)-\kappa_{*}}_{H^{1.5}(\Gamma_{t*})},\cmode{d_i^{k+1}}\right)\times\left(\mode{\pa_tP\left(\tkj{\rho}\right)}_{H^{1.5}(\Gamma_{t*})}+\cmode{\pa_t d_i^{k+1}}\right)\\\leq&\pallk,\\ 
             \mode{\pa_t^2\nkj{d_\gt}}_{H^{2.5}(\Gamma_{t*})}\leq& C\left(\mode{\frac{1}{\sigma}P\left(\tkj{\rho}\right)-\kappa_{*}}_{H^{0.5}(\Gamma_{t*})},\mode{\pa_t\tkj{\rho}}_{H^{0.5}(\Gamma_{t*})},\cmode{d_i^{k+1}},\cmode{\pa_t d_i^{k+1}}\right)\\&\quad \times\left(\mode{\pa_t^2 P\left(\tkj{\rho}\right)}_{H^{0.5}(\Gamma_{t*})}+\cmode{\pa_t^2 d_i^{k+1}}\right)\\\leq& \pallk.
        \end{aligned}
    \end{equation*}

    Moreover,  it remains to analyze the regularity of $\pa_t^3\nkj{d}_\gt$. 
     On the one hand, differentiating the equation \eqref{E6-14}$_2$ in time and utilizing the decomposition \eqref{E6.26}, we obtain the expression
     \begin{equation*}
         \begin{aligned}
               \pa_t^3\gth{\nkj{\rho}}=&(\dt^3\nkj{\rho})\circ\nk{\Phi_\st}-\left(\nabla_{\nk{v_*}}\pa_t^2+D_{t*}\nabla_{\nk{v_*}}\pa_t+D_{t*}^2\nabla_{\nk{v_*}}\right)\gth{\nkj{\rho}}
        \\
       =&\frac{\sigma}{\nk{\rho}\circ\nk{\Phi_\st}}\left((a-\Delta_\gt)\dt\nnt\nkj{\rho}-[\dt,\Delta_\gt]\nnt\nkj{\rho}+\dt\nk{F_2}-\frac{\dt\nk{\rho}}{\sigma}\dt^2\nkj{\rho}\right)\circ\nk{\Phi_\st}\\&-\left(\nabla_{\nk{v_*}}\pa_t^2+D_{t*}\nabla_{\nk{v_*}}\pa_t+D_{t*}^2\nabla_{\nk{v_*}}\right)\gth{\nkj{\rho}},
         \end{aligned}
     \end{equation*}
    here $\gth{\nkj{\rho}}=\nkj{\rho}\circ\nk{\Phi_\st}$. 
    On the other hand, recall the linearized operator (see, e.g., (2.4) in \cite{Shatah2011}) $L(S_*)d=-(\mu\cdot n_{t*})\Delta_{S_*}d+r(\pa_s d)$. Then, 
     the non-degenerate elliptic system for $\pa_t^3\nkj{d}_\gt$ can be roughly expanded as
     \begin{equation*}\label{E6.30}
         \begin{cases}
        -\left(\mu\cdot\nt\circ\nk{\Phi_\st}\right)\Delta_{S_*} \left(\pa_t^3\nkj{d_\gt}\right)=\frac{1}{\sigma}\pa_t^3 P\left(\gth{\nkj{\rho}}\right)+r\left( \pa_t^3\pa_s\nkj{d_\gt}\right),\\
        \pa_t^3\nkj{d_\gt}(X_{i*})=\pa_t^3\nkj{d_i},\ i=l,r.
    \end{cases}
     \end{equation*}
    Note that $\mu\circ\left(\nk{\Phi_\st}\right)^{-1}\cdot\nt\geq c_0>0$ for $t\in [0,T]$ and for some constant $c_0>0$. Therefore, 
    the standard elliptic analysis (after homogenizing the boundary conditions) shows that
    \begin{equation}\label{E7.73}
        \begin{aligned}
             \mode{\pa_t^3\nkj{d}_\gt}_{H^1(\Gamma_{t*})}\leq& \po\left(\mode{\dt\nnt\nkj{\rho}}_{H^1(\nk{\Gamma_t}))}+\mode{\nnt\nkj{\rho}}_{H^\frac{3}{2}(\nk{\Gamma_t})}+\mode{\dt^2\nkj{\rho}}_{L^2(\Gamma_t^k)}\right)\\&+\po\left(\mode{\dt\nkj{\rho}}_{H^1(\Gamma_t^k)}+\mode{\nkj{\rho}}_{H^2(\nk{\Gamma_t})}+\mode{\rv}_{H^2(\nk{\Omega_t})}\right) ,
        \end{aligned}
    \end{equation}
    here we have used \begin{equation}\label{E6.311}
        \mode{\nat\nk{v}\cdot\nt}_{H^{1.5+\epsilon}(\nk{\Gamma_t})}+\mode{\dt \left(\nat \nk{v}\cdot\nt\right)}_{H^{1.5}(\nk{\Gamma_t})}+\mode{\dt^2\left(\nat\nk{v}\cdot\nt\right)}_{L^2(\Gamma_t^k)}\leq \po,
    \end{equation}
    which can be derived from  \eqref{E7.33}  via the identity $\nat \nk{v}\cdot\nt=-\dt\nt\cdot\at$. Following  similar arguments as in Lemma \ref{L4.3}, the terms in the RHS of \eqref{E7.73} involving $\nkj{\rho}$ pose no difficulty.  
    One can derive from the Euler equation \eqref{E} to gain the  regularity of $\rv$,
    \[\ddt   \mode{\rv}_{H^2(\Omega_{t}^k)}\leq \po   \mode{\rv}_{H^2(\Omega_{t}^k)}+\mode{\nabla\nkj{\rho}+\g e_2}_{H^2(\Omega_t^k)}, \]
 using the Gronwall's inequality, and integrating over $[0,T]$, with  $T<1$, yields
 \begin{equation}\label{E6.32}
     \mode{\rv}_{H^2(\nk{\Omega_t})}\leq\pallk.
 \end{equation}
    Hence,  $\pa_t^3\nkj{d_\gt}$ can be controlled as stated.
\end{proof}
The following lemma shows the elliptic estimates of $\nkj{v}$. 
\begin{lemma}\label{l6.2}
    Let $\mathcal{P}_0:=P(\mathcal{E}(0))$. For some $T>0$,  suppose that $\mathcal{E}^k(t)\leq \po$ for all $t\in [0,T]$. Let $\nkj{\Gamma_t}\in\lam$ for $t\in[0,T]$.  Then the following estimates hold:
    \[
    \begin{aligned}
        \mode{\nkj{v}}_{H^{2+\epsilon}\left(\nkj{\Omega_t}\right)}+\mode{\left(\pa_t+\nkj{v}\cdot\nabla\right)\nkj{v}}_{H^2\left(\nkj{\Omega_t}\right)}\leq\pallk,\\\mode{\left(\pa_t+\nkj{v}\cdot\nabla\right)^2\nkj{v}}_{H^1\left(\nkj{\Omega_t}\right)}\leq \pallk.
    \end{aligned}    
    \] 
\end{lemma}
\begin{proof}
    It follows from \eqref{E7.13} that 
    \begin{align*}
        \ddt \mode{\nkj{\omega}}_{H^2(\nk{\Omega_t})}\leq \po \mode{\nkj{\omega}}_{H^2(\nk{\Omega_t})}+\mode{\divv \nk{v}\nabla^\perp \cdot \nk{v}}_{H^2(\nk{\Omega_t})},
    \end{align*}applying the Gronwall inequality and integrating over $[0,T]$ shows that
    \begin{equation}\label{E6.33}
       \mode{\nkj{\omega}}_{H^2(\nk{\Omega_t})}\leq C(T)\po. 
    \end{equation}
    Therefore, combining the estimates \[\mode{\divv \nkj{v}}_{H^2(\nkj{\Omega_t})}+\mode{\nabla^\perp\cdot \nkj{v}}_{H^2(\nkj{\Omega_t})}+  \mode{\nkj{v}\cdot\nkj{n_t}}_{H^{3.5}(\nkj{\Gamma_t})}\leq \pallk,\] with the Hodge decomposition from Section \ref{Section3} yields the asserted control of $ \mode{\nkj{v}}_{H^{2+\epsilon}\left(\nkj{\Omega_t}\right)}$. Next, applying $\dt^{k+1}\triangleq \pa_t+\nkj{v}\cdot\nabla$ to the div-curl system \eqref{E8.16}, yields
    \begin{align*}
         \divv \nkj{D_t}\nkj{v}=&\tr\left((\nabla\nkj{v})^2\right)-\left(\dt\left(\frac{1}{\nk{\rho}}\dt\nkj{\rho}\right)\right)\circ \nk{\mathcal{T}_\st}\circ\left(\nkj{\mathcal{T}_\st}\right)^{-1}\\&+\left(\nabla_{\nkj{v_\sharp}-\nk{v}_\sharp}\left(\frac{1}{\nk{\rho}}\dt\nkj{\rho}\circ \nk{\mathcal{T}_\st}\right)\right)\circ \left(\nkj{\mathcal{T}_\st}\right)^{-1},\\
\nabla^\perp\cdot \nkj{D_t}\nkj{v}=&-\divv \nkj{v}\left(\nabla^\perp\cdot\nkj{v}\right)\\&+\left(\dt\nkj{\omega}\circ\nk{\mathcal{T}_\st}+\nabla_{\nkj{v_\sharp}-\nk{v_\sharp}}\left(\nkj{\omega}\circ\nk{\mathcal{T}_\st}\right) \right)\circ \left(\nkj{\mathcal{T}_\st}\right)^{-1},\\
\nkj{D_t}\nkj{v}\cdot\nkj{n_t}=&\nkj{D_t}\left(\pa_t\nkj{d_\gt}\mu\circ \left(\Phi_\st^{k+1}\right)^{-1}\cdot\nkj{n_t}\right)-\nkj{v}\cdot\nkj{D_t}\nkj{n_t},\\
\nkj{D_t}\nkj{v}\cdot\nkj{n_b}=&\pa_t\varUpsilon -\nkj{v}\cdot\nabla_{\nkj{v}}\nkj{n_b}.
    \end{align*}
    Therefore, combining Lemma \ref{l6.1} and \eqref{E6.33}, one arrives at \[\mode{\divv\nkj{D_t} \nkj{v}}_{H^1(\nkj{\Omega_t})}+\mode{\nabla^\perp\cdot \nkj{D_t}\nkj{v}}_{H^1(\nkj{\Omega_t})}+  \mode{\nkj{D_t}\nkj{v}\cdot\nkj{n_t}}_{H^{1.5+\epsilon}(\nkj{\Gamma_t})}\leq \pallk.\] 
    Note that 
    \begin{equation}
        \sup_{t'\in[0,t]}\left(\cmode{\pa_{t'}\varUpsilon(t')}+\cmode{\pa_{t'}^2\varUpsilon(t')}\right)\leq\pallk.\label{E7.34}
    \end{equation}
Then $\mode{\nkj{D_t}\nkj{v}}_{H^2(\nkj{\Omega_t})}$ is  controlled as desired. 
   The argument for $\left(\nkj{D_t}\right)^2\nkj{v}$ is analogous, the details are thus omitted.
\end{proof}
Unlike Lemma \ref{l6.2}, the variable $\rv$ is determined exclusively by $\nkj{\rho}$, inequality \eqref{E6.32} shows that $\rv$ gains $H^2(\Omega_t^k)$ regularity from the Euler structure \eqref{E}. However, this alone  cannot guarantee that $\nat \rv\cdot\at|_{X_i^k}$ is controllable. We must therefore exploit further structures arising from the evolution system of  $\nkj{\rho}$.

\begin{lemma}\label{l6.3}
    Let $\mathcal{P}_0:=P(\mathcal{E}(0))$. For some $T>0$,  suppose that $\mathcal{E}^k(t)\leq \po$ for all $t\in [0,T]$. Let $\nk{\Gamma_t}\in\lam$ for $t\in[0,T]$. Then the following estimates hold:
\[\mode{\rv}_{H^{2+\epsilon}(\Omega_t^k)}\leq \pallk.\]
\end{lemma}
\begin{proof}
    The further structure of $\rv$ can be derived from the iteration system \eqref{E6-14}. In fact,  applying the operator $\divv$ to \eqref{E} and combining \eqref{E6-14}$_1$, we deduce that \begin{equation*}
        \begin{aligned}
             \dt\left(\divv \rv+\frac{1}{\nk{\rho}}\dt\nkj{\rho}\right)=&\frac{1}{(\nk{\rho})^2}\dt\nk{\rho}\left(\dt\nk{\rho}-\dt\nkj{\rho}\right)\\&+\tr\left(\nabla\nk{v}\nabla\nk{v}\right) -\tr\left(\nabla\nk{v}\nabla\rv\right)
     \quad \text{ in }\Omega_t^k,
        \end{aligned}
    \end{equation*}
    which implies that 
    \begin{equation*}
        \begin{aligned}
             \mode{\divv \rv(t)}_{H^{1+\epsilon}(\nk{\Omega_t})} \leq & \mode{\frac{1}{\nk{\rho}}\dt\nkj{\rho}(t)}_{H^{1+\epsilon}(\Omega_t^k)}+\mode{\divv \rv(t)+\frac{1}{\nk{\rho}}\dt\nkj{\rho}(t)}_{H^{1+\epsilon}(\nk{\Omega_t})}\\ \leq & \pallk+ \po e^{\po t}\\&+e^{\po t}\int_0^t\po\left(\po+ P(\nkj{\mathcal{E}}(t'))+\mode{\rv(t')}_{H^{2+\epsilon}(\nk{\Omega_{t'}})}\right)\id t'.
        \end{aligned}
    \end{equation*}

The rotational part of $\rv$ can be derived by  invoking the operator $\nabla^\perp\cdot$ to \eqref{E}, \begin{equation*}
        \dt \left(\nabla^\perp \cdot\rv\right)=\pa_2 \nk{v}_j\pa_j\rv_1-\pa_1\nk{v}_j\pa_j\rv_2\quad \text{ in }\Omega_t^k,
    \end{equation*} which leads to \begin{equation*}
        \mode{\nabla^\perp\cdot\rv(t)}_{H^{1+\epsilon}(\nk{\Omega_t})}\leq \po e^{\po t}+e^{\po t}\int_0^t \po\mode{\rv(t')}_{H^{2+\epsilon}(\nk{\Omega_{t'}})}\id t' .
    \end{equation*}

    Following the arguments of \eqref{E5.63}, together with $\dt\left(\rv\cdot\nt\right)=-\nnt\nkj{\rho}-\g e_2\cdot\nt-\rv\cdot\dt\nt$, one has
    \begin{equation*}
        \begin{aligned}
           & \ddt \mode{\rv\cdot\nt}_{H^{1.5+\epsilon}(\Gamma_t^k)}\\\leq &\po\mode{\rv\cdot\nt}_{H^{1.5+\epsilon}(\Gamma_t^k)}+\mode{\nnt\nkj{\rho}}_{H^{1.5+\epsilon}(\Gamma_t^k)}+\mode{\g e_2\cdot\nt+\rv \cdot\dt\nt}_{H^{1.5+\epsilon}(\Gamma_t^k)}\\\leq&\po\mode{\rv\cdot\nt}_{H^{1.5+\epsilon}(\Gamma_t^k)}+\po+\pallk+\po \mode{\rv}_{H^{2+\epsilon}(\nk{\Omega_t})}.
        \end{aligned}
    \end{equation*}
    Applying Gronwall inequality directly leads to 
    \begin{equation*}
        \mode{(\rv\cdot\nt)(t)}_{H^{1.5+\epsilon}(\nk{\Gamma_t})}\leq \po e^{\po t}+e^{\po t}\int_0^t\po\left(\po+ P(\nkj{\mathcal{E}}(t'))+\mode{\rv(t')}_{H^{2+\epsilon}(\nk{\Omega_{t'}})}\right)\id t'.
    \end{equation*}

    Also, combining \eqref{E} and  \eqref{E6-14}$_3$ implies that 
    \begin{equation*}
        \dt (\rv \cdot\nb)=\left(\rv-\nk{v}\right)\cdot\nabla_{\nk{v}}\nb\quad \text{ on }\Gamma_b^k.
    \end{equation*}Then, \begin{equation*}
        \mode{\rv\cdot n_b}_{H^{1.5+\epsilon}(\nk{\Gamma_b})}\leq \po e^{\po t}+e^{\po t}\int_0^t\po\left(\po+\mode{\rv(t')}_{H^{2+\epsilon}(\nk{\Omega_{t'}})}\right)\id t'.
    \end{equation*}

   Following the arguments in Lemma \ref{L4.3} and combining the inequalities derived above, 
   the regularity of $\mode{\rv}_{H^{2+\epsilon}(\Omega_t^k)}$ can be governed by 
   \begin{equation*}
       \begin{aligned}
             \mode{\rv(t)}_{H^{2+\epsilon}(\Omega_t^k)}\leq& \pallk+ \po e^{\po t}\\&+e^{\po t}\int_0^t\po\left(\po+ P(\nkj{\mathcal{E}}(t'))+\mode{\rv(t')}_{H^{2+\epsilon}(\nk{\Omega_{t'}})}\right)\id t'\\\leq & \pallk,
       \end{aligned}
   \end{equation*}
   for $t\in[0,T]$, choosing $T<1$ small enough such that $e^{\po t}\po t< \frac{1}{2}$. This completes the proof. 
\end{proof}

\begin{proof}[Proof of Proposition \ref{Prop6.1}]
   Suppose that the energy bound \eqref{E6.252} holds for  all $\mathcal{E}^{\ell}(t), \ell\leq k$. Subsequently, we prove by induction that \eqref{E6.252} also holds for $k+1$. We shall only highlight the differences from the arguments in Section~\ref{Section4} and Section~\ref{Section5}.

   Similar as Lemmas \ref{L5.2} and \ref{L5.3}, applying $\dt$ to the both sides of \eqref{E6-14}$_2$ and $\eqref{E6-14}_4$,   together with \eqref{E},  one has 
    \begin{equation*}\label{E6.36}
        \begin{cases}
            \frac{\nk{\rho}}{\sigma}\dt^3\nkj{\rho}-\nat \left(\nat \dt\nnt\nkj{\rho}-2\left(\nat \nk{v}\cdot\nt\right)\Delta_\gt\nkj{\rho}\right)\\\qquad\qquad\qquad\qquad\qquad\qquad\qquad\qquad\qquad\qquad+a\dt\nnt\nkj{\rho}=\mathfrak{R}_3^{k+1}&\text{ on }\Gamma_t^k,\\
            \pm \dt^2\nnt\nkj{\rho}= -\frac{\sigma}{\beta_c}\left(\nat \dt\nnt\nkj{\rho}-2(\nat \nk{v}\cdot\nt)\Delta_\gt\nkj{\rho}\right)(\sin\nk{\theta_i})^2+\nkj{\tilde{r}_i}& \text{ at }X_i^k(i=l,r),
        \end{cases}
    \end{equation*}with\begin{align*}
        \mathfrak{R}_3^{k+1}\triangleq&-2(\nat \nk{v}\cdot\at)\Delta_\gt\nnt\nkj{\rho}-\nat\left(\nat\nk{v}\cdot\at\right)\nat\nnt\nkj{\rho}\\&+4\left(\nat\nk{v}\cdot\at\right)\dt\nat\left(\nat\nk{v}\cdot\nt\right)+4\dt (\nat \nk{v}\cdot\at)\nat\left(\nat \nk{v}\cdot\nt\right)\\&+2\dt\left(\nat\nk{v}\cdot\nt\right)\nat\left(\nat\rv\cdot\at\right)+\text{lower-order terms},
    \end{align*}the details are  omitted. 
    Combining \eqref{E6.311} and Lemma \ref{l6.3} shows that 
 \begin{equation*}\label{E6.35}
       \nkj{\tilde{r}_i}\big|_{X_i^k}=\pm\dt^2((\nb\cdot\nt)\nk{v}\cdot\nabla_{\nk{v}}\nb)+\text{controllable terms}.
 \end{equation*}Therefore,  we focus on $\dt^2\nk{v}|_{X_i^k}$. It follows from  \eqref{E8.16} that \begin{align*}
     \dt^2\nk{v}\cdot\nt=&\dt^2\left(\left(\pa_t \nk{d_\gt}\mu \right)\circ \left(\nk{\Phi_\st}\right)^{-1}\cdot\nt\right)-\dt^2\nt\cdot\nk{v}-2\dt\nt\cdot\dt\nk{v},\\
     \dt^2\nk{v}\cdot\nb=&\pa_t^2\varUpsilon-\dt^2 \nb\cdot\nk{v}-2\dt\nb\cdot\dt\nk{v},\label{E6.38}
 \end{align*}
 combining Lemma \ref{l6.1}, \eqref{E7.33} and \eqref{E7.34}, one arrives at \[\cmode{\nkj{\tilde{r}_i}|_{X_i^k}}\leq \po\cmode{\dt^2\nk{v}|_{X_i^k}}+\pallk\leq\pallk.\]

Compared with Proposition \ref{Prop5.1}, we only point out that the third term in $\nkj{\mathfrak{R}_3}$, has an analogous structure to $\mathcal{J}_1$, yet it must be handled with appropriate adjustments.
  Utilizing the fact that
\[\nat\nk{v}\cdot\nt=\nat\left((\pa_t\nk{d_\gt}\mu)\circ (\nk{\Phi_\st})^{-1}\cdot\nt\right)-\nat \nt\cdot\nk{v},\]
let $h$ be some polynomial  function, one has
\begin{align*}
\dt\nat\left(\nat\nk{v}\cdot\nt\right)=&\dt\Delta_\gt \left((\pa_t\nk{d_\gt}\mu)\circ (\nk{\Phi_\st})^{-1}\cdot\nt\right)-\dt\nat\left(\nat\nt\cdot\nk{v}\right)\\=& \dt \left(\left(\pa_t\gth{\nk{\rho}}\right)\circ \left(\nk{\Phi_\st}\right)^{-1}\right)+r(D_{t*}\pa_t\pa_s\nk{d_\gt})\circ(\nk{\Phi_\st})^{-1}\\&-h\left(\dt\Delta_\gt\nt,\nk{v},\mu \circ\left(\nk{\Phi_\st}\right)^{-1}\right)+\text{lower-order terms}.
\end{align*} Then, Lemma \ref{l6.1} allows the  integral 
\begin{align*}\label{E6.37}
4\int_{\nk{\Gamma_t}}\dt^2\nnt \nkj{\rho}\left(\nat\nk{v}\cdot\at\right)\dt\nat\left(\nat\nk{v}\cdot\nt\right)\id S
\end{align*} to be bounded as  in Proposition \ref{Prop5.1} (following the treatment of $\mathcal{J}_1$, $\mathcal{J}_4$). The lower-order terms are routine, following mutatis mutandis  in Proposition \ref{Prop5.1}. 

Next, we  follow the arguments of Proposition \ref{L4.11}. The difference lies in the bottom boundary integral,  recall from \eqref{E6-14} and  \eqref{E8.16} that \begin{align*}
    \nb\cdot \dt^2\nabla\nkj{\rho}=&\dt^2\left(\nk{v}\cdot\nabla_{\nk{v}}\nb-\g e_2\cdot\nb\right)-\dt\left(\nabla_{\nk{v}}\nb\right)\cdot\nabla\nkj{\rho}-2\nabla_{\nk{v}}\nb\cdot\dt\nabla\nkj{\rho}
    \\=&\dt^2\nk{v}\cdot\ab\left(\nabla_{\nk{v}}\nb\cdot\ab+\nabla_\ab\nb\cdot\nk{v}\right)-2 \nabla_{\nk{v}}\nb\cdot\ab\nabla_{\ab}\dt\nkj{\rho}\\&+\left(\pa_t^2\varUpsilon -\dt^2\nb\cdot\nk{v}-2\dt\nb\cdot\dt\nk{v}\right)\left(\nabla_{\nk{v}}\nb\cdot\nb+\nabla_{\nb}\nb\cdot\nk{v}\right)+\text{lower-order terms},
\end{align*}which differs substantially from \eqref{E5.622}. Since the bottom $\Gamma_b^k$ is a line segment near the contact points $X_i^k(i=l,r)$, 
we decompose the bottom  into two parts, i.e. $\Gamma_b^k=\Gamma_{b_1}^k\cup \Gamma_{b_2}^k$, where $\Gamma_{b_1}^k$ denotes the flat portions in the neighborhoods of two contact points, and the remaining part $\Gamma_{b_2}^k$. 
Introduce a cut-off function $\chi$ defined in $\Omega_t^k$ such that
\[0\leq \chi\leq 1 \text{ smooth enough},\ \chi|_{\Gamma_t^k}=0\text{ and } \chi|_{\Gamma_{b_2}^k}=1.\] 
Denote $f_1:=\left(\nabla_{\nk{v}}\nb\cdot\ab+\nabla_{\ab}\nb\cdot\nk{v}\right)\nk{\rho}$, which shows that $f_1=0$ in $\nk{\Gamma_{b_1}}$. Then, the bottom integral can be expressed as
\begin{align*}
    \int_{\nk{\Gamma_b}}\nk{\rho}\nb\cdot\dt^2\nabla\nkj{\rho}\dt^3\nkj{\rho}\id S=&\int_{\nk{\Gamma_b}}\chi f_1\dt^2\nk{v}\cdot\ab \dt^3\nkj{\rho}\id S+\text{controllable terms}\\=&\int_{\nk{\Omega_t}}\chi\mathcal{H}_{f_1}\nabla^\perp \cdot \dt^2\nk{v}\dt^3\nkj{\rho}\id X+\int_{\nk{\Omega_t}}\chi\mathcal{H}_{f_1}\dt^2\nk{v}\cdot \nabla^\perp \dt^3\nkj{\rho}\id X\\&+\text{controllable terms},
\end{align*}where $\mathcal{H}_{f_1}$ is the harmonic extension of $f_1$  with $\Delta \mathcal{H}_{f_1}=0$ in $\ot$ and $\mathcal{H}_{f_1}|_{\Gamma_b^k}=f,\ \nnt \mathcal{H}_{f_1}|_{\Gamma_t^k}=0$. The bottom integral is controlled as desired by applying integration by parts in time, together with $\mode{\chi\mathcal{H}_{f_1}\dt^3\nk{v}}_{L^2(\nk{\Omega_t})}\leq\po$ (applying Hodge-type elliptic estimates to $\chi\mathcal{H}_{f_1}\dt^3\nk{v}$).

As a result,
one can show that when $T$ is small enough, the following estimate holds 
\begin{equation*}
    \begin{aligned}
        \sup_{t\in [0,T]}\nkj{\mathcal{E}}(t)+\int_0^T \nkj{F}(t)\id t\leq &\po+\po\int_0^TP\left(\nkj{\mathcal{E}}(t)\right) \id t.
    \end{aligned}
\end{equation*}
Therefore, by a bootstrap argument, we can prove the desired bound \eqref{E6.252} holds for $k+1$-th solutions.
\end{proof}

\subsection{Cauchy sequence}\label{S6.4}
So far we have established the uniform-in-$k$ estimates \eqref{E6.252} for all $k\in\mathbb{N}$. In this part, we are finally in a position to prove that the sequence $(\nk{\rho},\nk{v},\nk{\Omega_t})$ is indeed a Cauchy sequence, which shows that this sequence converges to $(\rho,v,\Omega_t)$.

To  simplify the notation, for any function $f$ defined on $S_t^k$, let
   \[\bkj{f}=\nkj{f}\circ\nk{\Phi_\st},\quad \delta_{\bk{f}}=\bkj{f}-\bk{f}
  =\nkj{f}\circ\nk{\Phi_\st}-\nk{f}\circ \Phi_\st^{k-1},
 \] and for any $h$ defined on $\Omega_t^k$, let \[\tkj{h}=\nkj{h}\circ \nk{\mathcal{T}_\st},\quad \delta_{\tk{h}}
 =\tkj{h}-\tk{h}=\nkj{h}\circ \nk{\mathcal{T}_\st}-\nk{h}\circ\mathcal{T}_\st^{k-1}. \]

 One can pull the evolution system \eqref{E6-14} back to $\Omega_*$ through the coordinate map $\mathcal{T}_\st^k$. For the simplicity of notations, we introduce the following operators, 
\begin{align*}
    \mathcal{B}(\dgtk{})f\triangleq& 
      \left(\rho^k\Delta (f\circ(\mathcal{T}_{\st}^k)^{-1})\right)\circ \nk{\mathcal{T}_\st},
    \\\mathcal{A}(\dgtk{})f\triangleq& \left((a-\Delta_\gt)\nnt \left(f\circ
     \pstk{}
    \right)\right)\circ \nk{\Phi_\st},
     \\\mathcal{N}_j(\dgtk{}) f\triangleq & \left(\nabla_{n_j}(f\circ (\Phi_\st^k)^{-1})\right)\circ\Phi_\st^k\quad  \text{ for }j=t,b,\\
     \mathcal{S}_i(\dgtk{})f\triangleq &\frac{\sigma}{\beta_c}\left((\sin\theta_i^k)^2\nat\nnt(f\circ (\Phi_\st^k)^{-1})\right)\circ\Phi^k_\st\quad \text{ for }i=l,r.
\end{align*}
Then, the variation $\delta_{\tk{\rho}}$ satisfies the following system \begin{equation}\label{E6.471}
    \begin{cases}
        \mathbb{D}_t^2 \delta_{\tk{\rho}}-\mathcal{B}(\dgtk{})\delta_{\tk{\rho}}=F_{1,\delta}&\text{ in }\Omega_*,\\
     \frac{1}{\sigma}\nk{\rho}\circ\nk{\Phi_\st}D_{t*}^2 \delta_{\bk{\rho}}+ \mathcal{A}(\dgtk{})\delta_{\bk{\rho}}=F_{2,\delta}& \text{ on }\Gamma_{t*},\\
        \mathcal{N}_b(\dgtk{})\delta_{\bk{\rho}}=F_{3,\delta}
        &\text{ on }\Gamma_{b*},\\
        \pm   D_{t*}\mathcal{N}_t(\dgtk{})\delta_{\bk{\rho}}+\mathcal{S}_i(\dgtk{})\delta_{\bk{\rho}}=F_{4,i,\delta}&\text{ at }X_{i*}, i=l,r,
    \end{cases}
\end{equation}
where
\begin{align*}
   F_{1,\delta}\triangleq& -\left((\pa_t+v_\sharp^k\cdot\nabla)^2-(\pa_t+v_\sharp^{k-1}\cdot\nabla)^2\right)\tk{\rho}+\left(\mathcal{B}(\dgtk{})-\mathcal{B}(\dgtk{-1})\right)\tk{\rho}+\delta_{\widetilde{F_{1}^k}},\\
    F_{2,\delta}\triangleq &-\frac{1}{\sigma}\left(\nk{\rho}\circ\nk{\Phi_\st}\left(\pa_t+\nk{v}_*\cdot\nabla\right)^2-\rho^{k-1}\circ\Phi_\st^{k-1}\left(\pa_t+v^{k-1}_*\cdot\nabla\right)^2\right)\bk{\rho}- \left(\mathcal{A}\left(\dgtk{}\right)-\mathcal{A}\left(\dgtk{-1}\right)\right)\bk{\rho}+\delta_{\bk{F_2}},\\
    F_{3,\delta}\triangleq &-\left(\mathcal{N}_b(\dgtk{})-\mathcal{N}_b(\dgtk{-1})\right)\bk{\rho}+\delta_{\bk{F_{3}}},\\
  F_{4,i,\delta}\triangleq &\mp\left(\left(\pa_t+v_*^k\cdot\nabla\right)\mathcal{N}_t(\dgtk{})-\left(\pa_t+v_*^{k-1}\cdot\nabla\right)\mathcal{N}_t(\dgtk{-1})\right)\bk{\rho}-(\mathcal{S}_i(\dgtk{})-\mathcal{S}_i(\dgtk{-1}))\bk{\rho}+\delta_{\bk{F_{4,i}}}.
\end{align*}Moreover, similar  equations for  $\left(\delta_{\tk{v}},\delta_{\tk{\mathring{v}}},\delta_{\tk{\omega}}, \delta_{\nk{d_\gt}},\delta_{\nk{d}_i}\right)$ can be derived, and we omit the details here.


Therefore, we define the energy of the difference as \begin{align*}
     \mathcal{E}_\delta^k(t)=&
     a\mode{\mathcal{N}_t(\dgtk{})\delta_{\bk{\rho}}}_{C([0,T];L^2(\Gamma_{t*}))}
     +\mode{\pa_t^2 \delta_{\tk{\rho}}}_{C([0,T];L^2(\Omega_*))}+\mode{\pa_t\delta_{\tk{\rho}}}_{C([0,T];H^1(\Omega_*))}\\&+\sigma \mode{\mathcal{N}_t(\dgtk{})\delta_{\bk{\rho}}}_{C([0,T];H^1(\Gamma_{t*}))}+\mode{\delta_{\tk{\rho}}}_{C([0,T];H^1(\Omega_*))}+\mode{\delta_{\tk{\omega}}}_{C([0,T];H^1(\Omega_*))}
     \\&+\sum_{i=l,r}\left(\cmode{\delta_{\nk{d_i}}}_{C([0,T])}+\cmode{\delta_{\id (\nk{d_i})/\id t}}_{C([0,T])}\right),
\end{align*}and the dissipation of the difference \[F_\delta^k(t)=\sum_{i=l,r}\left|\left.\left(\sin \theta_i^k\nat\nnt \left(\delta_{\bk{\rho}}\circ(\Phi_\st^k)^{-1}\right)\right)\circ \Phi_\st^k\right|_{X_{i*}}\right|^2.\]
Next, we turn to estimate the various variables defined on $\Omega_*$ and $\Gamma_{t*}$ under the energy functional $\mathcal{E}_\delta^k(t)$.

\begin{lemma}
    Let $\mathcal{P}_0:=P(\mathcal{E}(0))$. For $T>0$ small enough,  suppose that $\nk{\mathcal{E}}(t)\leq \po$ for all $k\in\mathbb{N}, t\in [0,T]$. Let $\Gamma_t^k\in\lam$.  Then,  
    the following estimates hold, 
    \begin{align*}
        \mode{\delta_{\dgtk{}}}_{H^{3.5}(\gtx)}+ \mode{\pa_t\delta_{\dgtk{}}}_{H^{2.5}(\gtx)}+\mode{\pa_t^2\delta_{\dgtk{}}}_{H^{1}(\gtx)} \leq \po\left(\mathcal{E}^k_\delta(t)+\mathcal{E}^{k-1}_\delta(t)\right),\\
\mode{\delta_{\tk{v}}}_{H^2(\Omega_*)}+\mode{\pa_t\delta_{\tk{v}}}_{H^1(\Omega_*)}+\mode{\delta_{\tk{\mathring{v}}}}_{H^2(\Omega_*)}\leq \po\left(\mathcal{E}^k_\delta(t)+\mathcal{E}^{k-1}_\delta(t)\right),
    \end{align*}
    here $\delta_{d_\gt^k}=d_\gt^{k+1}-d_\gt^{k}$. 
\end{lemma}
\begin{proof}
    The proof is analogous to the proof of Lemmas \ref{l6.1}-\ref{l6.3} and is left to the reader.
\end{proof}

Now, we are able to prove the convergence result.
\begin{proposition}
    The sequence $\left(\nk{\rho},\nk{v},\nk{\Omega}\right)$ is a Cauchy sequence.
\end{proposition}
\begin{proof}
     Define a good unknown \[\mathcal{Q}=\mathbb{D}_t\delta_{\tk{\rho}}+\delta_{v^{k-1}_\sharp}\cdot\nabla\tk{\rho}, \]with the aim of avoiding the appearance of $\mathbb{D}_t^2\delta_{v^{k-1}_\sharp}$ in the subsequent analysis.
     Under this setting, it holds that \begin{align*}
        (\pa_t+\nk{v}_\sharp\cdot\nabla)^2\tkj{\rho}-(\pa_t+v_\sharp^{k-1}\cdot\nabla)^2\tk{\rho}=&\mathbb{D}_t^2 \delta_{\tk{\rho}}+\mathbb{D}_t\left(\delta_{v^{k-1}_\sharp}\cdot\nabla\tk{\rho}\right)+\delta_{v^{k-1}_\sharp}\cdot\nabla(\pa_t+v_\sharp^{k-1}\cdot\nabla) \tk{\rho}\\=& \mathbb{D}_t \mathcal{Q}+\delta_{v^{k-1}_\sharp}\cdot\nabla(\pa_t+v_\sharp^{k-1}\cdot\nabla) \tk{\rho}.
    \end{align*}Thus, applying $\mathbb{D}_t$ to the both side of \eqref{E6.471}$_1$,  it holds that 
   \begin{equation}
     \begin{aligned}
        \mathbb{D}_t^2 \mathcal{Q}-\mathbb{D}_t\mathcal{B}(\dgtk{})\delta_{\tk{\rho}}=&-\mathbb{D}_t\left(\delta_{v^{k-1}_\sharp}\cdot\nabla(\pa_t+v_\sharp^{k-1}\cdot\nabla) \tk{\rho}\right)\\&+\mathbb{D}_t \left(\left(\mathcal{B}(\dgtk{})-\mathcal{B}(\dgtk{-1})\right)\tk{\rho}+\delta_{\widetilde{F_{1}^k}}\right)\quad \text{ in }\Omega_*,\label{E6.57}
    \end{aligned}  
   \end{equation}
    where the right-hand side satisfies \[\mode{\text{RHS of \eqref{E6.57}}}_{L^2(\Omega_*)}\leq \po\left(\mathcal{E}^{k-1}_\delta(t)+\mathcal{E}^{k-2}_\delta(t)\right).\]
    Meanwhile, \eqref{E6.471}$_2$ can be rewritten for $Q$ as 
    \begin{equation}\label{E6.58}
        \begin{aligned}
             \frac{1}{\sigma}\nk{\rho}\circ\Phi^k_\st D_{t*}Q+\mathcal{A}(\dgtk{})\delta_{\bk{\rho}}&-\delta_{\gth{2\nat \nk{v}\cdot\nt\nat(\nat \rv\cdot\at)}}\\&=  \frac{1}{\sigma}\nk{\rho}\circ\Phi^k_\st D_{t*} \left(\delta_{v^{k-1}_\sharp}\cdot\nabla\tk{\rho}\right)+F_{2,\delta}-\delta_{\gth{2\nat \nk{v}\cdot\nt\nat(\nat \rv\cdot\at)}},
         \end{aligned}
    \end{equation}
  with the right-hand side satisfies \[\mode{\text{RHS of \eqref{E6.58}}}_{L^2(\Gamma_{t*})}\leq \po\left(\mathcal{E}^k_\delta(t)+\mathcal{E}^{k-1}_\delta(t)+\mathcal{E}^{k-2}_\delta(t)\right).\]
    
  Upon taking the $L^2(\Omega_*)$-inner product of \eqref{E6.57} with $\mathbb{D}_t Q$, the subsequent analysis proceeds in analogy to that presented in  Theorem \ref{T1}, and is thus omitted. Note that $\left(\delta_{\tk{\rho}},\delta_{\tk{\mathring{v}}},\delta_{\tk{\omega}},\delta_{\nk{d}_l},\delta_{\nk{d}_r}\right)$ have zero initial data, one can show that their time derivatives also vanish on $\{t=0\}$, this implies $\mathcal{E}_\delta^k(0)=0$. Consequently, 
   we obtain the preliminary inequality 
    \[\mathcal{E}_\delta^k(t)+\int_0^t F_\delta^k(t')\id t' \leq \po\int_0^t \left(\mathcal{E}^k_\delta(t')+\mathcal{E}^{k-1}_\delta(t')+\mathcal{E}^{k-2}_\delta(t')\right)\id t',\]for $0\leq t\leq T$. Choosing $T$ sufficiently small that \[\sup_{0\leq t\leq T}\mathcal{E}^k_\delta(t)\leq \lambda\left(\sup_{0\leq t\leq T}\mathcal{E}^{k-1}_\delta(t)+\sup_{0\leq t\leq T}\mathcal{E}^{k-2}_\delta(t)\right),\] for some $\lambda=4\po T\in (0,1)$, this implies immediately that the sequence of iterates  $\{\left(\nk{\rho},\nk{v},\nk{\Omega}\right)\}_{k=2}^\infty$ is convergent, the proof is finished.
\end{proof}

In conclusion, we know immediately that the sequence of approximate solutions $\{\left(\nk{\rho},\nk{v},\nk{\Omega_t}\right)\}_{k\in\mathbb{N}}$ has a strong convergent subsequence, with the limit $(\rho,v,\Omega_t)$. The uniqueness follows from  a parallel argument.

\subsection{Back to the Euler equation}\label{S6.5}

Assume that the initial data $(\rho_0,v_0,\Omega_0)$  satisfy the following  conditions
\begin{equation}\label{E6.40}
    \begin{cases}
    P(\rho)|_{t=0}=\sigma\kappa|_{t=0}, &\text{ on }\Gamma_0,\\
      v\cdot n_b|_{t=0}=0 &\text{ on }\Gamma_b,\\
       (\nat v\cdot\nt)|_{(0,X_l(0))}>0\text{ and } (-\nat v\cdot\nt)|_{(0,X_r(0))}>0,\\
       \left.\mp \beta_cD_t^j v\cdot\ab\right|_{t=0}=\left.\sigma D_t^j(\cos\theta_s-\cos\theta_i )\right|_{t=0}   \text{ for } j=0,1,&\text{ at }X_i(0),\ i=l,r.
    \end{cases}
\end{equation}
Note that the initial setting \eqref{E6.40}$_3$ is equivalent to \eqref{E1.16} under the equality \eqref{E5.69}.

We now demonstrate that  the solution $(\rho,v,\Omega_t)$ constructed in Sections \ref{S6.3} and \ref{S6.4} is equivalent to the solution of  Euler equations  \eqref{E2} for the given initial data $(\rho_0,v_0,\Omega_0)$.
Denote 
\begin{equation}\label{E6.51}
    \mathscr{V}_0=\dt v+\nabla\rho+\g e_2,
\end{equation}
it follows from \eqref{E}, that
\begin{equation}\label{E7.41}
    \dt\left(v-\mathring{v}\right)=\mathscr{V}_0.
\end{equation}
applying operators $\divv$ and $\nabla^\perp\cdot$ to \eqref{E6.51} respectively, we deduce from the equation \eqref{E6-14}$_1$, \eqref{E7.13} and \eqref{E8.16}$_{1-2}$   that  \begin{equation}\label{E6.544}
    \divv \mathscr{V}_0=0\quad \text{ and }  \quad \nabla^\perp \cdot \mathscr{V}_0=0\quad \text{ in }\ot.
\end{equation}

On the one hand, combining 
the  limit system of \eqref{E6.191} and \eqref{E8.16}$_3$ shows that
\begin{equation}\label{E6.52}
    \dt\left(\frac{1}{\sigma}P(\rho)\right)=-\Delta_\gt v\cdot\nt-2\nat \nt\cdot\nat v\quad \text{ on }\gt,
\end{equation}
hence, $P(\rho)|_\gt=\sigma\kappa$ is automatically satisfied for $t\in[0,T]$ under the initial condition \eqref{E6.40}$_1$.

On the other hand, since equation \eqref{E6-14}$_2$ is satisfied by any evolving family of interfaces, 
combining with \eqref{E6.51} and  \eqref{E6.52} leads to \begin{equation*}
    \Delta_\gt \mathscr{V}_0\cdot\nt+2\nat \nt\cdot\nat\mathscr{V}_0=2(\nat v\cdot\nt)\nat\left(\nat \left(v-\mathring{v}\right)\cdot\at\right) \quad \text{ on }\gt,
\end{equation*}
or equivalently, \begin{equation}\label{E7.44}
    \Delta_\gt (\mathscr{V}_0\cdot\nt)+|\kappa|^2(\mathscr{V}_0\cdot\nt)=(\nat \kappa)\at\cdot\mathscr{V}_0+2(\nat v\cdot\nt)\nat\left(\nat \left(v-\mathring{v}\right)\cdot\at\right)\quad  \text{ on }\gt.
\end{equation}
Meanwhile, recalling \eqref{E6.7} which leads to the evolution of $d_i$, one has \begin{equation*}\label{E6.573}
    \mathscr{V}_0\cdot\nt=0\quad \text{ at }X_i(i=l,r),
\end{equation*}combining with \eqref{E7.44} implies that \begin{equation*}
    \mode{\mathscr{V}_0\cdot\nt}_{H^1(\gt)}\leq \po\left(\mode{\mathscr{V}_0}_{L^2(\gt)}+\mode{v-\mathring{v}}_{H^1(\gt)}\right).
\end{equation*}

Also,  the boundary condition  \eqref{E6-14}$_3$ combined with \eqref{E8.16}$_4$  yields
\begin{equation*}\label{E6.56}
    \mathscr{V}_0\cdot \nb=\dt \left(v\cdot \nb\right)=\pa_t\varUpsilon \quad \text{ on }\gb.
\end{equation*}Using \eqref{E6.544}, we obtain 
\begin{equation*}
    \pa_t \varUpsilon|\gb|=\int_{\ot}\divv \mathscr{V}_0\id X-\int_{\gt}\mathscr{V}_0\cdot\nt\id S=-\int_{\gt}\mathscr{V}_0\cdot\nt\id S.
\end{equation*}

Recall that $\curl \mathscr{V}_0=0$, then $\mathscr{V}_0=\nabla\phi$ with potential $\phi$ satisfies $\int_\ot \phi\id X=0$. 
Therefore, under the initial setting \eqref{E6.40}$_3$, combining the above analysis together with the elliptic estimates show that
\begin{align*}
    \mode{\mathscr{V}_0}_{H^\frac{3}{2}(\ot)}\leq \mode{\phi}_{H^\frac{5}{2}(\gt)}\leq &\po\left(\mode{\mathscr{V}_0\cdot\nt}_{H^1(\gt)}+\mode{\mathscr{V}_0\cdot\nb}_{H^1(\Gamma_b)}+\mode{\phi}_{L^2(\ot)}\right)\\
    \leq &\po\left(\mode{\phi}_{H^1(\ot)}+\mode{v-\mathring{v}}_{H^\frac{3}{2}(\gt)}\right)\\ \leq &\po \mode{v-\mathring{v}}_{H^\frac{3}{2}(\gt)},
\end{align*}the last inequality follows by  absorbing $\mode{\phi}_{H^1(\ot)}$ into the left-hand side.
Then, applying the energy estimates to \eqref{E7.41}, and using Sobolev interpolation, one finds that for $t\in[0,T]$, 
\begin{equation*}
    \begin{aligned}
        \mode{(v-\mathring{v})(t)}_{H^\frac{3}{2}(\ot)}\leq& C(T) \int_0^t\mode{\mathscr{V}_0(t')}_{H^\frac{3}{2}(\ot)}\id t'\\\leq&C(T)\int_0^t\po \mode{(v-\mathring{v})(t')}_{H^\frac{3}{2}(\ot)}\id t'.
    \end{aligned}
\end{equation*}
A   continuity argument on a sufficiently small interval thus gives $v=\mathring{v}$, from which the Euler equation \eqref{E2}$_1$ follows. Consequently, 
 this in turn implies that $v\cdot\nb=0$ on $\gb$ under the initial setting \eqref{E6.40}$_2$, and the contact point condition \eqref{E2}$_6$ holds under  \eqref{E6.40}$_4$.

As a result, we can finally retrieve the solution $(\rho,v, \ot)$ to the compressible water waves system \eqref{E2}.



    \bigskip

    \noindent{\bf Acknowledgement.}   The author Jingchi Huang is supported by NSF  of China  under Grant No. 12471196.
     The author Chao Wang is supported by NSF of China under Grant No. 12471189.

\end{document}